\documentclass{amsart}

\usepackage{amsmath}
\usepackage{amsthm}
\usepackage{amssymb}
\usepackage[alphabetic]{amsrefs}
\usepackage{amscd}
\usepackage{setspace}
\usepackage{color}

\newcommand{ \directedisom } [0] { \stackrel{ \sim }{ \longrightarrow } }
\newcommand{ \defeq } [0] { \stackrel{ \textup{\tiny{def}} }{ = } }
\newcommand{ \iseq } [0] { \stackrel{ \textup{\tiny{?}} }{ = } }
\newcommand{ \textinsidemath } [1] { \hspace{5pt} \text{#1} \hspace{5pt} }
\newcommand{ \textaftermath } [1] { \hspace{5pt} \text{#1} }
\newcommand{ \presuperscript } [2] { {}^{ #1 } \hspace{-1pt} #2 }
\newcommand{ \suchthat } [0] { \hspace{5pt} \vert \hspace{5pt} }

\newcommand{ \N } [0] { \mathbf{N} }
\newcommand{ \Z } [0] { \mathbf{Z} }
\newcommand{ \Q } [0] { \mathbf{Q} }
\newcommand{ \R } [0] { \mathbf{R} }
\newcommand{ \C } [0] { \mathbf{C} }
\newcommand{ \F } [0] { \mathbf{F} }

\DeclareMathOperator{\identity}{id}

\DeclareMathOperator{\inc}{inc}
\DeclareMathOperator{\pr}{pr}
\DeclareMathOperator{\ac}{ac}
\DeclareMathOperator{\kernel}{ker}
\DeclareMathOperator{\image}{im}
\DeclareMathOperator{\myhom}{Hom}
\DeclareMathOperator{\myend}{End}
\DeclareMathOperator{\aut}{Aut}
\DeclareMathOperator{\inn}{Inn}
\DeclareMathOperator{\spec}{Spec}

\newcommand{ \linhomom } [1] {\myhom_{#1\textup{\tiny{-lin}}}}
\newcommand{ \linendom } [1] {\myend_{#1\textup{\tiny{-lin}}}}
\DeclareMathOperator{\trace}{Tr}
\DeclareMathOperator{\mydet}{det}
\DeclareMathOperator{\mychar}{char}
\DeclareMathOperator{\rank}{rank}

\DeclareMathOperator{\tr}{tr}
\DeclareMathOperator{\std}{std}
\newcommand{ \conjtr } [0] {*_{\textup{\tiny{std}}}}
\newcommand{ \opp } [1] { {#1}^{\textup{\tiny{opp}}} }
\newcommand{ \reducedtrace } [0] { {\trace}^{\textup{\tiny{red}}} }

\DeclareMathOperator{\gal}{Gal}
\newcommand{ \unr } [1] { { #1 }^{ \textup{\tiny{unr}} } }
\DeclareMathOperator{\frob}{Fr}
\DeclareMathOperator{\val}{val}

\newcommand{ \Ga } [0] { \mathbf{G}^{\textup{a}} }
\newcommand{ \Gm } [0] { \mathbf{G}^{\textup{m}} }
\DeclareMathOperator{\GL}{GL}
\DeclareMathOperator{\GU}{GU}

\DeclareMathOperator{\GSp}{GSp}

\DeclareMathOperator{\aff}{aff}
\DeclareMathOperator{\fixer}{Fix}
\DeclareMathOperator{\admissible}{Adm}
\DeclareMathOperator{\SUP}{sup}
\newcommand{ \antiid } [0] { \identity^{\vee} }
\newcommand{ \localmodel } [0] { \mathbf{M}^{\textup{\tiny{loc}}} }
\newcommand{ \biggermodel } [2] { \mathbf{M}^{(#1,#2)} }
\newcommand{ \biggergroup } [2] { \mathbf{J}^{(#1,#2)} }
\newcommand{ \biggermonoid } [2] { \widetilde{\mathbf{M}}^{(#1,#2)} }
\newcommand{ \universalgroup } [4] { \widetilde{\mathbf{J}}^{(#1,#2 \hspace{1pt} ; \hspace{1pt} #3,#4)} }
\newcommand{ \convolutionobject } [4] { \mathbf{Conv}^{(#1, #2 \hspace{1pt} ; \hspace{1pt} #3, #4)} }
\newcommand{ \reversedconvolutionobject } [4] { \presuperscript{\textup{rev}}{\mathbf{Conv}}^{(#1,#2 \hspace{1pt} ; \hspace{1pt} #3, #4)} }
\newcommand{ \convolutiontarget } [4] { \mathbf{P}^{(#1,#2 \hspace{1pt} ; \hspace{1pt} #3, #4)} }
\newcommand{ \affineflagvariety } [0] { {\mathcal{F}\ell}^{\aff} }
\newcommand{ \affinegrassmannian } [0] { {\mathcal{G}r}^{\aff} }
\newcommand{ \grassmannian } [0] { \mathbf{\textup{Gr}} }
\newcommand{ \iwahori } [0] {\mathcal{I}}
\newcommand{ \specialparahoric } [0] {\mathcal{K}}
\newcommand{ \integralaffineflagvar } [0] { \mathbf{Fl}^{\aff} }
\newcommand{ \integralaffineflagvarargs } [2] { \mathbf{Fl}^{(#1,#2)} }
\newcommand{ \integraliwahoriargs } [2] { \mathbf{Iw}^{(#1,#2)} }
\newcommand{ \integralaffineflagvarmonoid } [2] { \widetilde{\mathbf{Fl}}^{(#1,#2)} }
\newcommand{ \AFVdef } [1] {\textbf{AFV{#1}}}
\newcommand{ \AGVdef } [1] {\textbf{AG{#1}}}
\newcommand{ \OLMdef } [1] {\textbf{OLM{#1}}}
\newcommand{ \subOLMdef } [1] {\OLMdef{#1}\textbf{*}}
\newcommand{ \SLMdef } [1] {\textbf{SLM{#1}}}
\newcommand{ \BLMdef } [1] {\textbf{ELM{#1}}}
\newcommand{ \morita } [0] { \textup{Mrta} }

\newcommand{ \localreflex } [0] {E}
\newcommand{ \globalreflex } [0] { \mathbf{\localreflex} }

\newcommand{ \IC } [1] { \textup{IC}_{#1} }
\newcommand{ \ICbar } [1] { \overline{\textup{IC}}_{#1} }
\newcommand{ \nearbycycles } [0] { \textup{R}\Psi }
\newcommand{ \sstrace } [1] { \tau^{\textup{\tiny{ss}}}_{#1} }

\newcommand{ \catsets } [0] { \textup{Sets} }

\newcommand{ \catgroups } [0] { \textup{Groups} }
\newcommand{ \catalgebras } [1] { {#1}\textup{-Algebras} }
\newcommand{ \catmodules } [1] { {#1}\textup{-Modules} }
\newcommand{ \etale } [0] {\'{e}tale }
\newcommand{ \gortz } [0] {G\"{o}rtz}
\newcommand{ \fadeles} [0] {\mathbb{A}_{\textup{f}}}
\newcommand{ \introshimuravariety } [0] {Sh}
\newcommand{ \introshimuraintegralmodel } [0] {\mathbf{\introshimuravariety}}
\newcommand{ \biggermodelEZ } [0] { \mathbf{M} }
\newcommand{ \biggergroupEZ } [0] { \mathbf{J} }
\newcommand{ \gaitsgoryflag } [0] { \mathbf{Fl} }

\newtheorem*{maintheorem}{Main Theorem}
\newtheorem*{corollary}{Corollary}
\newtheorem*{remark}{Remark}
\newtheorem*{terminologynotation}{Terminology/Notation}

\newtheorem{lemma}{Lemma}[subsection]
\newtheorem{prop}[lemma]{Proposition}

\newtheorem*{classificationofunitaryinvolutions}{Classification of Unitary Involutions}
\newtheorem*{correspondencebetweeninvolutionsandforms}{Correspondence between Unitary Involutions and Hermitian Forms}
\newtheorem*{extractionofhermitianforms}{Extraction of Hermitian Forms}
\newtheorem*{projectivitylemma}{Local Criteria for Projectivity}
\newtheorem*{formalsmoothness}{Infinitesimal Lifting Property}
\newtheorem*{permanenceoffinitepresentedness}{Permanence of Finite-Presentedness}
\newtheorem*{homlocalization}{Localization of Hom-Sets}
\newtheorem*{latticedefs}{Equivalent Characterizations of Lattices}
\newtheorem*{minimalindependence}{Linear Independence of Minimal Generating Sets}

\newtheorem*{definition}{Definition}
\newtheorem*{textbooklocalmodeldef}{Definition: The Original Description of the Local Model}
\newtheorem*{altdesclocalmodel}{Definition: The Alternate Description of the Local Model}
\newtheorem*{BLMdef1}{Definition: The Enlarged Model (preliminary)}
\newtheorem*{BLMdef2}{Definition: The Enlarged Model (final)}
\newtheorem*{completeaffineflagvarietydef}{Definition: The Affine Flag Variety}
\newtheorem*{affinegrassmanniandef}{Definition: The Affine Grassmannian}
\newtheorem*{integralcompleteaffineflagvarietydef}{Definition: The Integral Affine Flag Variety}
\newtheorem*{integraliwahorisubgroupdef}{Definition: The Integral Iwahori Subgroup}
\newtheorem*{convolutionscheme}{Definition: The Convolution Scheme}
\newtheorem*{reversedconvolutionscheme}{Definition: The Reversed Convolution Scheme}
\newtheorem*{convolutionschemebase}{Definition: The Convolution Base}
\newtheorem*{convolutionproduct}{Definition: The Convolution Product}
\newtheorem*{reversedconvolutionproduct}{Definition: The Reversed Convolution Product}

\newtheorem*{BBDpullbackofperverse}{Beilinson-Bernstein-Deligne Proposition 4.2.5}
\newtheorem*{HNfibersum}{Haines-Ng\^{o} Proposition 10}
\newtheorem*{HNdescendsheaf}{Haines-Ng\^{o} Lemma 21}

\theoremstyle{definition}
\newtheorem*{assumption}{Assumption}

\begin{document}

\title[Kottwitz’s conjecture for some unitary Shimura varieties]{Kottwitz's nearby cycles conjecture for \\ a class of unitary Shimura varieties}

\author{Sean Rostami}

\address{
\begin{flushleft}
University of Wisconsin\newline
Department of Mathematics\newline
480 Lincoln Dr.\newline
Madison WI 53706
\end{flushleft}
}

\email{srostami@math.wisc.edu}
\email{sean.rostami@gmail.com}


\begin{abstract}
This paper proves that the nearby cycles complexes on a certain family of PEL local models are \emph{central} with respect to the convolution product of sheaves on the corresponding affine flag varieties. As a corollary, the semisimple trace functions defined using the action of Frobenius on those nearby cycles complexes are, via the sheaf-function dictionary, in the centers of the corresponding Iwahori-Hecke algebras. This is commonly referred to as \emph{Kottwitz's Conjecture}. The reductive groups associated to the PEL local models under consideration are unramified unitary similitude groups with even dimension. The proof follows the method of \cite{HN}. Upon completion of the first version of this paper, Pappas and Zhu released a preprint (now published as \cite{PZ}) which contained within its scope the main theorem of this paper. However, the methods of \cite{PZ} are very different and some of the proofs from this paper have been useful in forthcoming work of Haines-Stroh.
\end{abstract}

\maketitle

\tableofcontents

\section*{Introduction}

The object of study in this paper is a certain projective $ \Z_p $-scheme $ \localmodel $, called a \emph{local model}, and the \emph{nearby cycles complex} $ \nearbycycles ( \overline{\Q}_{\ell} ) $ on $ \localmodel_{ \overline{\F}_p } $, a certain complex of \etale $ \ell $-adic sheaves.

\subsection*{Historical background and motivation}

The purpose of local models is to give \'{e}tale-local descriptions of various Shimura varieties in a way that uses only module-theoretic language and makes some questions and computations more tractable. A major step in computing the \emph{Hasse-Weil zeta function} of a Shimura variety is the computation of the trace of the Frobenius element considered as a linear map on the stalks of the nearby cycles complex $ \nearbycycles ( \overline{\Q}_{\ell} ) $ on $ \localmodel_{ \overline{\F}_p } $. In this paper, I prove that the nearby cycles complexes on a certain class of local models coming from unitary-type division algebras are \emph{central} with respect to a convolution product of \etale $ \ell $-adic sheaf complexes. A corollary is that the trace function associated to the Frobenius element as above is a specific, effectively computable element in the center of the corresponding Iwahori-Hecke algebra. This is known as \emph{Kottwitz's Conjecture}. A description of the local models considered and a precise statement of the theorem appear below.

Let $ \fadeles $ be the ring of finite adeles over $ \Q $. Let $ \textbf{G} $ be a linear algebraic group defined over $ \Q $, let $ h : \C^{\times} \rightarrow \textbf{G}_{\R} $ be an algebraic cocharacter and let $ \textbf{K} \subset \textbf{G}(\fadeles) $ be a compact open subgroup. The triple $ ( \textbf{G}, h, \textbf{K} ) $ (under certain additional hypotheses) is called a \emph{Shimura datum} and can be used to construct a Shimura variety $ \introshimuravariety $. This Shimura variety is defined over some number field $ \globalreflex $, called the \emph{reflex field}. Some Shimura varieties, for example those whose datum comes from a PEL datum (in particular, the case considered in this paper), have an integral model, i.e. a scheme $ \introshimuraintegralmodel $ over $ \mathcal{O}_{\globalreflex} $ such that $ \introshimuraintegralmodel_{\globalreflex} $ is the original Shimura variety $ \introshimuravariety $. The fibers over primes $ \mathfrak{p} \subset \mathcal{O}_{\globalreflex} $ of such an integral model $ \introshimuraintegralmodel $ are sometimes smooth (in which case $ \introshimuraintegralmodel $ is said to have \emph{good reduction} at $ \mathfrak{p} $) and sometimes non-smooth (in which case $ \introshimuraintegralmodel $ is said to have \emph{bad reduction} at $ \mathfrak{p} $). I now fix a prime $ \mathfrak{p} $ and consider only PEL (``polarization, endomorphisms, level-structure'') Shimura varieties with parahoric level-structure at $ \mathfrak{p} $. 

Rapoport and Zink \cite{RZ} constructed local models of many integral Shimura varieties $ \introshimuraintegralmodel $ within an axiomatic framework. In some cases, the objects constructed by \cite{RZ} were found to be unsatisfactory; some examples where $ \localmodel $ is not a flat scheme were provided by Pappas \cite{P} in the ramified unitary case and Genestier in the even-dimensional orthogonal case. Modifications were made by Pappas and Rapoport in subsequent papers \cite{PR1}, \cite{PR2}, and \cite{PR3}, and evidence that these modifications produce flat models is supplied by Smithling's papers \cite{Sm1}, \cite{Sm2} and \cite{Sm3}, which specifically address both of the problematic examples previously mentioned. Nonetheless, it follows from \cite{Go1} that the local models considered below are flat.

The Hasse-Weil zeta function is defined as a product over all primes $ \mathfrak{p} \subset \globalreflex $ of certain local factors $ Z_{\mathfrak{p}} ( s, \introshimuravariety ) $, and a standard manipulation shows that each local factor is determined by the alternating sum of traces on $ \mathbb{H}^{\bullet} ( \introshimuravariety_{\overline{\Q}_p}; \overline{\Q}_{\ell} )^{\Gamma_0} $ of $ \frob^j $ for each $ j > 0 $, where $ \Gamma_0 \subset \gal ( \overline{\Q}_p / \globalreflex_{ \mathfrak{p} } ) $ is the inertia subgroup ($ p \in \Z $ is the prime below $ \mathfrak{p} $), $ \frob \in \gal ( \overline{\Q}_p / \globalreflex_{ \mathfrak{p} } ) $ is an arbitrary lift of the Frobenius element, and $ \frob $ acts on cohomology via its action on $ \introshimuravariety_{\overline{\Q}_p} $. The $ \Gamma_0 $-invariants operation makes these traces difficult to understand, and Rapoport suggests modifying it as follows: a theorem of Grothendieck says that there is a finite-index subgroup of $ \Gamma_0 $ which acts by unipotent operators on $ \mathbb{H}^{\bullet} ( \introshimuravariety_{\overline{\Q}_p}; \overline{\Q}_{\ell} ) $ and so, after replacing each of these cohomologies by the associated graded of a suitable filtration, $ \Gamma_0 $ acts via a finite quotient and the $ \Gamma_0 $-invariants functor is exact. This new alternating trace, on the $ \Gamma_0 $-invariants of such gradings, is called the \emph{semisimple trace} and, in \S2 of \cite{R} it is shown, assuming Deligne's Weight-Monodromy Conjecture, how the original trace may be recovered from this semisimple trace.

By using nearby cycles, cohomology of $ \introshimuravariety_{\overline{\Q}_p} $ can be replaced by cohomology of $ \introshimuraintegralmodel_{\overline{\F}_p} $. Let $ X $ be a proper $ \Z_p $-scheme and let $ p : X_{\Q_p} \rightarrow X $, $ q : X_{\F_p} \rightarrow X $, $ c : X_{\overline{\Q}_p} \rightarrow X_{\Q_p} $ be the canonical maps. The nearby cycles complex on $ X_{\overline{\F}_p} $ is defined to be the derived complex $ \nearbycycles ( \overline{\Q}_{\ell} ) \defeq \overline{q}^{*} ( R \overline{p}_{*} ( c^{*} ( \overline{\Q}_{\ell} ) ) ) $, where $ \overline{\Q}_{\ell} $ denotes the constant \etale $ \ell $-adic sheaf on $ X_{\Q_p} $. By Base Change for proper morphisms, $ \mathbb{H}^{\bullet} ( X_{\overline{\Q}_p} ; \overline{\Q}_{\ell} ) = \mathbb{H}^{\bullet} ( X_{\overline{\F}_p} ; \nearbycycles ( \overline{\Q}_{\ell} ) ) $ and the action by $ \gal ( \overline{\Q}_p / \Q_p ) $ on the left is consistent with the action by $ \gal ( \overline{\F}_p / \F_p ) $ on the right.

If $ \introshimuraintegralmodel $ is proper over $ \mathcal{O}_{\globalreflex} $ then, because of the previous two paragraphs and the Grothendieck-Lefschetz Trace Formula, applied to $ \nearbycycles ( \overline{\Q}_{\ell} ) $ on $ \introshimuraintegralmodel_{\overline{\F}_p} $, the main objective is to understand the alternating trace of $ \frob $ on the $ \Gamma_0 $-invariants of the semisimplified cohomology stalks of the complex $ \nearbycycles ( \overline{\Q}_{\ell} ) $ over each point of $ \introshimuraintegralmodel ( \F_{\mathfrak{p}} ) $, where $ \F_{\mathfrak{p}} $ is the residue field of $ \globalreflex_{ \mathfrak{p} } $. By the \'{e}tale-local equivalence mentioned earlier, it is the same to calculate those traces over $ \localmodel ( \F_{\mathfrak{p}} ) $ instead. The resulting function $ \sstrace{} : \localmodel ( \F_{\mathfrak{p}} ) \rightarrow \overline{\Q}_{\ell} $, the main object of study, is called the \emph{semisimple trace function}.

Haines and Ng\^{o} \cite{HN} consider the split groups $ \GL $ and $ \GSp $ and the standard local models corresponding to these groups. They prove an instance of Kottwitz's Conjecture, that the semisimple trace function $ \sstrace{} $ in this situation is essentially the Bernstein basis function $ z_{\mu} $ in the center of the corresponding Iwahori-Hecke algebra (here $ \mu $ is a certain cocharacter occurring in the precise definition of the local model, which is omitted here). In fact, \cite{HN} proves more--that every member of a family of functions, each of which is defined similarly to $ \sstrace{} $, is a specific linear combination of Bernstein basis functions; see Theorem 11 in \cite{HN} for a precise statement. The strategy of the proof, which I follow closely in this paper, is:
\begin{enumerate}
\setlength{\itemsep}{7pt}
\item Construct an ind-scheme $ \biggermodelEZ $ over $ \Z_p $, which contains $ \localmodel $ as a closed subscheme and whose extension to $ \Q_p $, resp. to $ \F_p $, is the affine grassmannian, resp. full affine flag variety. This requires finding alternate descriptions of $ \localmodel $ that are more compatible with the usual definitions of affine flag varieties as unions of sets of lattice-chains.

\item Via the embedding of $ \localmodel_{\F_p} $ into the full affine flag variety, prove that the semisimple trace function $ \sstrace{} $ is an element of the Iwahori-Hecke algebra $ \mathcal{H} $ and construct products $ *_{\overline{\Q}_p} $ and $ *_{\overline{\F}_p} $ of complexes of \etale $ \ell $-adic sheaves on $ \biggermodelEZ_{\overline{\Q}_p} $ and $ \biggermodelEZ_{\overline{\F}_p} $ such that (via the sheaf-function dictionary) $ *_{\overline{\F}_p} $ categorifies the convolution product in $ \mathcal{H} $.

\item Show that the nearby-cycles functor $ \nearbycycles $ is a ``homomorphism'' with respect to these two products and that the product of the relevant complexes on $ \mathbf{M}_{\overline{\Q}_p} $ is commutative. It follows that the product on $ \localmodel_{\overline{\F}_p} $ of the relevant complexes is also commutative.
\end{enumerate}

On the other hand, Gaitsgory \cite{Ga} proves a similar result (albeit not in the context of Shimura varieties) for split connected reductive $ \F_p((t)) $-groups $ G $. One of the objects occurring in \cite{Ga} is an ind-scheme $ \gaitsgoryflag_X $, reportedly due to Beilinson, defined over a smooth curve $ X $ such that one fiber is the full affine flag variety for $ G $ and every other fiber is essentially the affine grassmannian for $ G $; see Proposition 3 of \cite{Ga} for a precise statement. The main result of \cite{Ga} is that the nearby cycles functor on $ \gaitsgoryflag_X $ induces the isomorphism (the composition of the Satake and Bernstein isomorphisms) from the special parahoric Hecke algebra $ \mathcal{H} ( G ( \F_p((t)) ); G ( \F_p[[t]] ) ) $ to the center of the Iwahori-Hecke algebra; see Theorem 1 in \cite{Ga} for a precise statement.

\subsection*{Subject of this paper and statement of the main theorem}

The Shimura data that I consider are similar to those occuring in Kottwitz \cite{Ko1}, except I consider \emph{Iwahori} level structure rather than \emph{hyperspecial maximal} level structure. Let $ F \supset \Q $ be an imaginary quadratic extension with ring of integers $ \mathcal{O} $. Let $ D $ be a central division $ F $-algebra and suppose that $ D $ has a unitary (2nd kind) involution $ * $. To this pair $ ( D, * ) $ is attached a certain similitude group $ \mathbf{G} $ defined over $ \Q $. Let $ 2 \neq p \in \Z $ be a prime for which $ G = \mathbf{G}_{\Q_p} $ is quasi-split and split over $ \unr{\Q}_p $. Since the case when $ p $ \emph{splits} in $ \mathcal{O} $ is known (see Haines \cite{Ha2}), I assume that $ p $ is \emph{inert} in $ \mathcal{O} $. After selecting a minuscule cocharacter $ \mu $ (which will ultimately be created using the cocharacter $ h $ mentioned previously), one can define the local model $ \localmodel $, an $ \mathcal{O}_{ \localreflex } $-scheme for a certain extension $ \localreflex / \Q_p $ which is again called the \emph{reflex field} and depends on the $ G ( \overline{\Q}_p ) $-conjugacy class of $ \mu $. By inertness, $ F_p = F \otimes_{\Q} \Q_p $ is a \emph{field}, the completion of $ F $ at $ p $, and $ D \otimes_{\Q} \Q_p = M_d ( F_p ) $. It follows that $ G $ is a (not necessarily quasi-split) unitary similitude group and is quasi-split if and only if the involution $ *_p $ is isomorphic to the standard one. See \S\ref{SSalwaystranspose} for more details. Moreover, the reflex field $ \localreflex $ must be either $ F_p $ or $ \Q_p $, and since $ F_p $ is also the splitting field of $ G $, the case of $ \localreflex = F_p $ reduces to the case of $ \GL $ and I may assume without loss of generality that $ \localreflex = \Q_p $. This also implies that the dimension $ d $ is \emph{even} (which makes non-quasisplit $ G $ a genuine possibility) and that the \emph{signature} of $ \mu $ is $ ( d/2, d/2 ) $. See \S\ref{SSreflexfield} for more details.

\gortz's idea that $ \localmodel $ can frequently be embedded into an appropriate affine flag variety holds in this case and the semisimple trace function $ \sstrace{} $ on $ \localmodel ( \F_p ) $ can therefore be interpreted as an element of the Iwahori-Hecke algebra $ \mathcal{H} $ of $ \GU_d $. Kottwitz's Conjecture is that this element of $ \mathcal{H} $ is a certain scalar multiple of the Bernstein basis function $ z_{\mu} $ associated to $ \mu $. By Haines's characterization (Theorem 5.8 in \cite{Ha1}) of minuscule Bernstein basis functions, Kottwitz's Conjecture (in the case at hand) follows from the main theorem of this paper:
\begin{maintheorem}
Suppose $ 2 \neq p \in \Z $ is \emph{inert} in $ \mathcal{O} $ and let $ \localmodel $ be the local model over $ \Z_p $ associated to the unitary-type division algebra datum $ ( D, *, \mu ) $ as above, and suppose that the similitude group attached to $ ( D, * ) $ is quasi-split.

Then $ \localmodel $ is isomorphic to the standard local model corresponding to the (unramified) unitary similitude group $ \GU_d $ of the extension $ F_p / \Q_p $, and the nearby cycles complex $ \nearbycycles ( \overline{\Q}_{\ell} ) $ on $ \localmodel_{ \overline{\F}_p } $, considered as a complex on the full affine flag variety $ \affineflagvariety_{ \overline{\F}_p } $ of $ \GU_d $, is \emph{central} with respect to the convolution product $ * $ of sheaf complexes, i.e. $ \nearbycycles ( \overline{\Q}_{\ell} ) * \mathcal{C}^{\bullet} \cong \mathcal{C}^{\bullet} * \nearbycycles ( \overline{\Q}_{\ell} ) $ naturally for every perverse Iwahori-equivariant complex of \etale $ \ell $-adic sheaves $ \mathcal{C}^{\bullet} $ on $ \affineflagvariety_{ \overline{\F}_p } $.

Via the sheaf-function dictionary, the associated semisimple trace function $ \sstrace{} $ is therefore a \emph{central} element of the Iwahori-Hecke algebra of $ \GU_d $.
\end{maintheorem}
Kottwitz's Conjecture, whenever it is true, allows $ \sstrace{} $ to be computed explicitly: the Bernstein basis functions can be computed in a systematic way using only some well-known information about linear algebraic groups and Coxeter groups. This can be done on a computer or even by hand, in low rank cases.

Upon completion of the first version of this paper, Pappas and Zhu released a preprint (now published as \cite{PZ}) which proved Kottwitz's Conjecture in all cases where the group is unramified. This includes the cases considered in this paper. This paper and \cite{PZ} constitute the first proofs of Kottwitz's conjecture in the \emph{non-split} case. However, the methods of \cite{PZ} are very different from those of this paper and some of the proofs from this paper have been useful in forthcoming work of Haines-Stroh. In addition, this paper also supplies some details that were suppressed in \cite{HN}, in both general and specific contexts due to analogies between unitary groups and symplectic groups.

\subsection*{Brief outline of this paper}

\textbf{In \S\ref{Snotation}}, I merely fix some notation and and record a few simple facts about tensor products of certain fields. \textbf{In \S\ref{SPELdatum}} (page \pageref{SPELdatum}), I choose the objects that will eventually be used to construct the local models, set some conventions, and recall various classical results about simple algebras, hermitian forms, involutions, etc. \textbf{In \S\ref{Scommalglemmas}} (page \pageref{Scommalglemmas}), I recall and also prove some commutative algebra lemmas that are used throughout the paper. \textbf{In \S\ref{Sdefinitionofthelocalmodel}} (page \pageref{Sdefinitionofthelocalmodel}), I recall the definition of the local model as it appears in \cite{RZ} and rephrase parts of the definition in equivalent ways that are more obviously related to affine flag varieties. \textbf{In \S\ref{Smorita}} (page \pageref{Smorita}), I analyze what happens to these conditions after applying Morita equivalence to change the target categories of $ \localmodel $ from $ \catmodules{ M_d( \mathcal{O} ) } $ to $ \catmodules{ \mathcal{O} } $. \textbf{In \S\ref{Sdefinitionofthelargermodels}} (page \pageref{Sdefinitionofthelargermodels}), I define an ind-scheme $ \biggermodelEZ $, prove some basic properties about it, and prove that it is a degeneration from the affine grassmannian over $ \Q_p $ to the full affine flag variety over $ \F_p $. \textbf{In \S\ref{Sautomgroup}} (page \pageref{Sautomgroup}), I define an ind-group $ \biggergroupEZ $ which acts on $ \biggermodelEZ $ and which is similarly an interpolation between the special parahoric over $ \Q_p $ to the Iwahori over $ \F_p $. I prove that the subgroups comprising $ \biggergroupEZ $ are \emph{smooth}, which is critical in order for the semisimple trace function to be an element of the Iwahori-Hecke algebra. I also give Schubert cell decompositions for the subschemes comprising $ \biggermodelEZ $. \textbf{In \S\ref{Stracefunction}} (page \pageref{Stracefunction}), I set notation and conventions for the sheaf theory I will use, I recall the precise definition for the semisimple trace function $ \sstrace{} $ that is the subject of this paper, and I verify that $ \sstrace{} $ can be interpreted as an element of the Iwahori-Hecke algebra for $ \GU_d $. I then restate (without proof) the main theorem of the paper, and show how Kottwitz's conjecture follows from the centrality of the nearby cycles complexes. The proof of the main theorem occurs in \S\ref{Smainproof} and requires the material from \S\ref{Sconvolutiondiagram}, \S\ref{Sconvdiagramproperties} and \S\ref{Sfusionproduct}. \textbf{In \S\ref{Sconvolutiondiagram}} (page \pageref{Sconvolutiondiagram}), following a well-known general recipe from \cite{Lu}, I define several objects and morphisms, the totality of which is commonly referred to as a ``convolution diagram'', and prove various properties (representability, finite-type, etc.) about those objects and morphisms. \textbf{In \S\ref{Sconvdiagramproperties}} (page \pageref{Sconvdiagramproperties}), I define an ind-group $ \widetilde{\biggergroupEZ} $, similar in spirit to $ \biggergroupEZ $, which acts on some objects in the convolution diagram, and I prove several critical and non-trivial facts about the action of $ \widetilde{\biggergroupEZ} $ on the objects in the convolution diagram. I also mention some important simplifications that occur over $ \Q_p $ and $ \F_p $. \textbf{In \S\ref{Sfusionproduct}} (page \pageref{Sfusionproduct}), I use the convolution diagram to define a product operation between complexes of \etale $ \ell $-adic sheaves, and I explain why this product induces the usual convolution product in the Hecke algebra. \textbf{In \S\ref{Smainproof}} (page \pageref{Smainproof}), I use the material from \S\ref{Sconvolutiondiagram}, \S\ref{Sconvdiagramproperties} and \S\ref{Sfusionproduct} to prove the main theorem.

The main difficulties in this proof occur in \S\ref{Sautomgroup} and \S\ref{Sconvdiagramproperties}. A collection of properties (involving connectedness, smoothness, and transitivity) related to the groups $ \biggergroupEZ $ and $ \widetilde{\biggergroupEZ} $ are necessary in order to construct the convolution product and needed to be proved from scratch.

\subsection*{Acknowledgements}

I thank my Ph.D. adviser, Thomas Haines, for suggesting this question in the first place, for all his support and advice during and after graduate school, for many helpful discussions during the original resolution of this question, and for numerous improvements to the final version (including the recent identification of an error in one of the proofs and the suggestion of a way to fix it). I thank Niranjan Ramachandran for urging me to finally submit this paper for publication. Most of the additions and improvements to this version were accomplished during the Fall 2014 semester at MSRI (DMS-1440140), in consultation with Matthias Strauch, and I also thank the organizers and staff for inviting me to the program and for all their help throughout the semester. I thank the staff of Sel. Math. New Ser. for all their assistance and the referee for a helpful report. Finally, I thank Selecta Kojak for all the good times and great riddims.

\section{Notation and conventions} \label{Snotation}

The symbols $ \N $, $ \Z $, $ \Q $, $ \R $, and $ \C $ denote respectively the natural numbers, the integers, the rational numbers, the real numbers, and the complex numbers.

Let $ F $ be a totally imaginary quadratic number field. Let $ \mathcal{O} $ be the ring of integers in $ F $. Let $ p \in \Z $ be a prime that is \emph{inert} in $ \mathcal{O} $. \emph{Assume that $ p \neq 2 $ (this is used in the proof of Proposition \ref{Palwaystranspose} (page \pageref{Palwaystranspose}) and Proposition \ref{PsmoothJ} (page \pageref{PsmoothJ}).} Set $ F_p \defeq F \otimes_{\Q} \Q_p $. Note that $ F_p $ is an unramified quadratic extension of $ \Q_p $ by \S6 of \cite{Sc}.

Let $ D $ be a central division $ F $-algebra such that $ \dim_F ( D ) = d^2 $. Let $ * : D \rightarrow D $ be a unitary (or ``2nd kind'') involution.  Denote by $ \F_p $ the finite field with $ p $ elements. Fix field embeddings $ \overline{\Q} \hookrightarrow \C $ and $ \overline{\Q} \hookrightarrow \overline{\Q}_p $. Set $ D_p \defeq D \otimes_{\Q} \Q_p $ and let $ *_p $ denote the involution of $ D_p $ induced by $ * $. Note that $ D_p $ is a central simple $ F_p $-algebra by (ii)(e) on page 374 of \cite{Sc}. The automorphism of $ F_p $ induced by the non-trivial element of $ \gal ( F / \Q ) $ is the non-trivial element of $ \gal ( F_p / \Q_p ) $ and so the involution $ *_p $ is a unitary involution.

All rings are assumed to have a multiplicative identity $ 1 $. Denote by $ \Ga $ and $ \Gm $ the additive and multiplicative algebraic groups $ R \mapsto ( R, + ) $ and $ R \mapsto ( R^{\times}, \cdot ) $ over a base ring that will always be clear from context. If $ R $ is any ring, the category of \emph{commutative} $ R $-algebras is denoted $ \catalgebras{R} $. Any $ R $-algebra that is also a field is called an ``$ R $-field''. A complex $ \cdots \rightarrow M_i \rightarrow M_{i+1} \rightarrow \cdots $ of modules is sometimes denoted more concisely by $ M_{\bullet} $.

Let $ K / k $ be a separable quadratic field extension with non-trivial Galois automorphism $ x \mapsto \overline{ x } $. If $ X $ is a matrix with entries $ X_{i,j} \in K $ then $ \overline{X} $ denotes the matrix with entries $ \overline{X}_{i,j} $ and $ X $ is called ``$ K / k $-hermitian'' iff $ \overline{X}^{\tr} = X $. An involution on a central $ K $-algebra $ A $ is called ``$ K / k $-unitary'' iff its restriction to $ Z ( A ) = K $ is $ a \mapsto \overline{a} $. The standard involution $ X \mapsto \overline{X}^{\tr} $ on square $ K $-matrices is sometimes called $ \conjtr $. A $ k $-bilinear form $ \phi : K^d \times K^d \rightarrow K $ is called ``$ K / k $-hermitian'' iff $ \phi ( x v, w ) = x \phi ( v, w ) = \phi ( v, \overline{x} w ) $ for all $ v, w \in K^d $, $ x \in K $. I frequently denote by $ \antiid $ the ``anti-identity'' matrix, the matrix with $ 1 $ in the anti-diagonal entries and $ 0 $ in all other entries, with dimension implied by context.

\section{The PEL datum} \label{SPELdatum}

\subsection{The global PEL datum} \label{SSglobalPELdatum}
The starting point is to form a global PEL datum
\begin{equation*}
( B, \iota, V, \psi )
\end{equation*}
using the pair $ ( D, * ) $ as follows:
\begin{itemize}
\setlength{\itemsep}{7pt}
\item finite-dimensional simple $ \Q $-algebra $ B \defeq \opp{D} $,

\item \emph{positive} involution $ \iota : B \rightarrow B $ defined by $ \iota(b) \defeq \xi \cdot b^{*} \cdot \xi^{-1} $ for certain $ \xi \in D $ satisfying $ \xi^{*} = - \xi $ (see \S5.2 of \cite{Ha2} for existence of such an element),

\item finite-dimensional \emph{left} $ B $-module $ V \defeq D $ with $ B $ acting by $ b \star v \defeq v b $, and

\item alternating $ \Q $-bilinear form $ \psi : V \times V \rightarrow \Q $ defined by $ \psi ( v, w ) \defeq \reducedtrace_{D/\Q} ( v \cdot \xi \cdot w^{*} ) $ which automatically satisfies $ \psi( b \star v, w ) = \psi( v, \iota(b) \star w ) $.
\end{itemize}

\begin{remark}
Two \emph{different} products $ D \times D \rightarrow \Q $ as above can induce the \emph{same} involution $ \iota $ on $ B $, due to the fact that the products induce involutions on $ \linendom{F} ( D ) $, of which $ B $ is only a proper subalgebra.
\end{remark}

From the datum $ ( B, \iota, V, \psi ) $ one can define an affine algebraic $ \Q $-group $ G $:
\begin{definition}
The functor
\begin{equation*}
G : \catalgebras{\Q} \longrightarrow \catgroups
\end{equation*}
assigns to each commutative $ \Q $-algebra $ R $ the group of all $ g \in \linendom{B} ( V ) \otimes_{\Q} R $ for which there exists a scalar $ c(g) \in R^{\times} $ such that $ \psi_R ( g v, g w ) = c(g) \cdot \psi_R ( v, w ) $ for all $ v, w \in V $.
\end{definition}
Note that $ \linendom{B} ( V ) $ here is simply $ D $ acting on $ B = \opp{D} $ on the left.

By definition of $ \psi $, another description of $ G $ is
\begin{equation*}
G(R) = \{ x \in D \otimes_{\Q} R \suchthat g^{*} \cdot g \in R^{\times} \}
\end{equation*}

I select also an $ \R $-algebra homomorphism
\begin{equation*}
h : \C  \rightarrow \linendom{B} ( V ) \otimes_{\Q} \R = D
\otimes_{\Q} \R
\end{equation*}
such that
\begin{itemize}
\setlength{\itemsep}{5pt}
\item $ h(\overline{z}) = h(z)^{*} $

\item $ B \otimes_{\Q} \R \rightarrow B \otimes_{\Q} \R $ defined by $ b \mapsto h(i)^{-1} \cdot b^{*} \cdot h(i) $ is a \emph{positive} involution
\end{itemize}
Any $ h $ satisfying the first property can be used to define a cocharacter
\begin{equation*}
\mu = \mu_h : \Gm_{\C} \rightarrow G_{\C}.
\end{equation*}
See \S\ref{SSreflexfield} (page \pageref{SSreflexfield}) for details. Set $ B_p \defeq B \otimes_{\Q} \Q_p $ and similarly for $ V, \iota, \psi $. Let $ \mathcal{O}_B \subset B = \opp{D} $ be a maximal order such that $ \mathcal{O}_{B_p} \defeq \mathcal{O}_B \otimes_{\Z} \Z_p $ is a maximal order $ \mathcal{O}_{B_p} $ in $ B_p $.

\begin{assumption}
$ \iota_p ( \mathcal{O}_{B_p} ) = \mathcal{O}_{B_p} $.
\end{assumption}
This guarantees that if $ \Lambda \subset V_p $ is an $ \mathcal{O}_{B_p} $-submodule then the dual module
\begin{equation*}
\widehat{\Lambda} \defeq \{ x \in V_p \suchthat \psi_p ( \Lambda, x )
\subset \Z_p \}
\end{equation*}
is again an $ \mathcal{O}_{B_p} $-submodule. This assumption is also used in the proof of Proposition \ref{Palwaystranspose} (page \pageref{Palwaystranspose}).

I will also need to select a certain doubly-infinite chain
\begin{equation*}
\Lambda_{\bullet} = ( \cdots \subset \Lambda_{-1} \subset \Lambda_0 \subset \Lambda_1 \subset \cdots )
\end{equation*}
of $ \mathcal{O}_{B_p} $-lattices in $ V_p = D_p $. See \S\ref{SSoriginallocalmodeldef} (page \pageref{SSoriginallocalmodeldef}) for more details.

The datum used to define the local model is the tuple
\begin{equation*}
( B, \iota, V, \psi, \mu, \mathcal{O}_B, F, \Lambda_{\bullet} )
\end{equation*}

\subsection{Standard theorems about involutions, hermitian forms, etc.} \label{SSalwaystranspose}

\begin{prop} \label{Palwaysmatrix}
The central simple $ F_p $-algebra $ D_p $ is split, i.e. $ D_p \cong M_d(F_p) $ as $ F_p $-algebras.
\end{prop}

\begin{proof}
By Wedderburn's Theorem, the central simple $ F_p $-algebra $ D_p $ is a matrix algebra over some central division $ F_p $-algebra. Since $ D_p $ has a unitary involution, Corollary 8.8.3 on page 306 of \cite{Sc} states that this division algebra has a unitary involution also. But Albert's Theorem (Theorem 10.2.2(ii) on page 353 of \cite{Sc}) states that, over a local field, the only division $ F_p $-algebra with a unitary involution is $ F_p $ itself.
\end{proof}

\begin{remark}
This result is explicitly part of Landherr's theorem (Theorem 10.2.4 on page 355 of \cite{Sc}), which uses the hypothesis that $ p $ is inert. In the above proof, the assumption that $ p $ is inert guarantees that $ D \otimes_{\Q} \Q_p $ is a \emph{simple} algebra.
\end{remark}

This means that $ G_{\Q_p} $ is \emph{always} a unitary similitude group, although depending on $ * $, perhaps $ G_{\Q_p} $ is not quasi-split.

Using Proposition \ref{Palwaysmatrix}, fix some isomorphism $ D_p \cong M_d ( F_p ) $, and consider $ *_p $ and $ \iota_p $ as unitary involutions on $ B_p = \opp{ M_d ( F_p ) } $ via this isomorphism.

In \S\ref{SSmoritaandproducts} (page \pageref{SSmoritaandproducts}), I will use the following proposition:
\begin{prop} \label{Palwaystranspose}
There exists an $ F_p $-algebra automorphism (necessarily inner by Skolem-Noether) of $ B_p = \opp{ M_d(F_p) } $ that identifies $ \iota_p $ with $ \conjtr $ and $ \mathcal{O}_{B_p} $ with $ \opp{ M_d ( \mathcal{O}_{F_p} ) } $.
\end{prop}

\begin{proof}
This is actually just Theorem 10.2.5 on page 355 of \cite{Sc} but it is not clear just from the statement (``almost all primes''), so I make some additional comments.

Let $ K / k $ be a quadratic extension of global fields with non-trivial Galois automorphism $ x \mapsto \overline{x} $. Let $ A $ be a central simple $ K $-algebra ($ \dim_K(A) = n^2 $), and $ I $ a unitary involution on $ A $. The assertion of Theorem 10.2.5 is that for all but finitely many primes $ \mathfrak{p} $ of $ k $, there is a $ K $-algebra isomorphism $ A \otimes_{K} K_{\mathfrak{p}} \cong M_n(K_{\mathfrak{p}}) $ such that $ I $ becomes identified to $ \conjtr $.

A careful reading of the proof shows that, for a particular $ \mathfrak{p} $, such an isomorphism exists provided that:
\begin{enumerate}
\setlength{\itemsep}{5pt}
\item $ \mathfrak{p} $ is non-archimedean
\item $ \mychar ( \mathcal{O}_k / \mathfrak{p} ) \neq 2 $
\item $ \mathfrak{p} $ is not ramified in $ K $
\item $ A \otimes_{\Q} \Q_{\mathfrak{p}} $ is split, i.e. a matrix algebra
\item $ I $ stabilizes a maximal order in $ A \otimes_{\Q} \Q_{\mathfrak{p}} $
\end{enumerate}
The key point is that the initial setup of this proof is unnecessarily restrictive: \cite{Sc} chooses a single \emph{global} order $ \Lambda \subset A $, creates the $ I $-stable order $ \Lambda \cap I ( \Lambda ) $, considers only those $ \mathfrak{p} $ for which the completion at $ \mathfrak{p} $ of this new $ I $-stable order is \emph{maximal}, and notes that there are only finitely many exceptions. Really, all that is needed is that for each $ \mathfrak{p} $ (with only finitely many exceptions), there is some $ I_{\mathfrak{p}} $-stable maximal order (possibly depending on $ \mathfrak{p} $).

My situation assumes (1) and (3), I have explicitly assumed (2) and (5), and Proposition \ref{Palwaysmatrix} provides (4), so the first part of the proposition is proven.

It is obvious from the proof that the isomorphism sends the $ I $-stable maximal order to $ M_d ( \mathcal{O}_{F_p} ) $.
\end{proof}

Fix an isomorphism
\begin{equation*}
( B_p, \iota_p, \mathcal{O}_{B_p} ) \cong ( \opp{ M_d(F_p) }, \conjtr, \opp{ M_d ( \mathcal{O}_{F_p} ) } )
\end{equation*}
as in Proposition \ref{Palwaystranspose}.

Recall the following classification theorem for unitary involutions:
\begin{classificationofunitaryinvolutions} [Theorem 8.7.4 on pages 301-302 of \cite{Sc}]
Let $ K / k $ be a quadratic extension with non-trivial Galois automorphism $ x \mapsto \overline{x} $. Let $ A $ be a central
simple $ K $-algebra and fix a $ K / k $-unitary involution $ I $. \textbf{Assertion}:
\begin{enumerate}
\setlength{\itemsep}{5pt}
\item \label{existinvol} If $ b \in A^{\times} $ satisfies $ b = I(b) $, then the function $ \inn(b) \circ I $ is a $ K / k $-unitary involution,

\item \label{uniqueinvol} if $ J : A \rightarrow A $ is a $ K / k $-unitary involution, then there is an $ b \in A^{\times} $ satisfying $ b = I(b) $ such that $ J = \inn(b) \circ I $, and this $ b $ is unique up to scalar in $ k^{\times} $, and

\item \label{isominvol} there exists an isomorphism $ ( A, \inn(a) \circ I ) \directedisom ( A, \inn(b) \circ I ) $ if and only if there exists $ c \in A $ and $ \alpha \in K $ such that $ b = \alpha ( c \cdot a \cdot I(c) ) $.
\end{enumerate}
\end{classificationofunitaryinvolutions}

Recall also the correspondence between unitary involutions and hermitian forms in the split-algebra case (by Proposition \ref{Palwaysmatrix} (page \pageref{Palwaysmatrix}), this is the case of interest):
\begin{correspondencebetweeninvolutionsandforms} \label{correspondencebetweeninvolutionsandforms}
Let $ K / k $ be a separable quadratic extension with non-trivial Galois automorphism $ x \mapsto \overline{x} $. For any $ K / k $-hermitian $ d \times d $ matrix $ H $, define the (necessarily $ K / k $-hermitian) form $ \phi_H : K^d \times K^d \rightarrow K $ by the formula $ \phi_H ( v, w ) \defeq v^{\tr} \cdot H \cdot \overline{w} $ and define the (necessarily $ K / k $-unitary) involution $ *_H : M_d(K) \rightarrow M_d(K) $ by the formula $ X^{*_H} \defeq H \cdot \overline{X}^{\tr} \cdot H^{-1} $. \textbf{Assertion}:
\begin{enumerate}
\setlength{\itemsep}{5pt}
\item Any $ K / k $-hermitian form on $ K^d $ is equal to $ \phi_H $ for some $ H $ as above,

\item any $ K / k $-unitary involution on $ M_d(K) $ is equal to $ *_H $ for some $ H $ as above,

\item the involution induced by $ \phi_H $ on $ M_d(K) $ is exactly $ *_H $, and

\item the function $ \phi_H \mapsto *_H $ descends to a bijection between isometry classes of $ K / k $-hermitian forms on $ K^d $ and isomorphism classes of $ K / k $-unitary involutions on $ M_d(K) $.
\end{enumerate}
\end{correspondencebetweeninvolutionsandforms}

\begin{proof}
Assertions (1) and (3) are trivial. Assertion (2) is a special case of the ``Classification of Unitary Involutions''. To verify (4), first note that if the vector space $ K^d $ is transformed by $ A \in \GL_d(K) $ then the hermitian form $ \phi_H $ is transformed into $ (v,w) \mapsto ( A(v) )^{\tr} \cdot H \cdot \overline{A(w)} = v^{\tr} \cdot ( A^{\tr} \cdot H \cdot \overline{A} ) \cdot \overline{w} $. In other words, $ \phi_H $ is transformed into $ \phi_{ A^{\tr} \cdot H \cdot \overline{A} } $. Second, note that if $ M_d(K) $ is transformed by $ \inn(A) $, then the involution $ *_H $ is transformed into $ X \mapsto A ( H ( \overline{ A^{-1} X A } )^{\tr} H^{-1} ) A^{-1} = ( A H \overline{A}^{\tr} ) \overline{X}^{\tr} ( A H \overline{A}^{\tr} )^{-1} $. In other words, $ *_H $ is transformed to into $ *_{ A \cdot H \cdot \overline{A}^{\tr} } $. This proves that the function $ \phi_H \mapsto *_H $ descends to a function from isometry classes of hermitian forms to isomorphism classes of unitary involutions. By parts (2) and (3) of this Correspondence, the function is \emph{surjective}. By part (3) of the ``Classification of Unitary Involutions'', it is \emph{injective}: if two hermitian forms induce two involutions that are isomorphic, then part (3) of the Classification guarantees a certain element ``$ c $'', and this element can be applied to $ K^d $ in order to transform one hermitian form into the other.
\end{proof}

When $ k = \Q_{p^{\prime}} $ (including $ p^{\prime} = 2 $), these classifications can be made much more specific and it is well-known (see \S10.6.5 of \cite{Sc}) that for each dimension there are exactly two isometry classes of hermitian forms, and if that dimension is \emph{even} (as in my situation), then these hermitian forms yield two \emph{non-isomorphic} unitary similitude groups, only one of which is quasi-split (if the dimension is \emph{odd}, then the unitary groups associated to the two hermitian forms are isomorphic and quasi-split). It is clear that whether or not $ G_{\Q_p} $ is quasi-split depends on $ *_p $. I address this question at the end of this subsection.

Finally, recall a correspondence between certain bilinear forms occuring frequently in the theory of PEL local models and certain hermitian forms:
\begin{extractionofhermitianforms} \label{extractionofhermitianforms}
Let $ K / k $ be a separable quadratic extension and call an alternating $ k $-bilinear form $ \presuperscript{\circ}{\psi} : K^d \times K^d \rightarrow k $ ``internally-hermitian'' iff $ \presuperscript{\circ}{\psi} ( A(v), w ) = \presuperscript{\circ}{\psi} ( v, \overline{A}^{\tr}(w) ) $ for all $ A \in M_d(K) $ and $ v, w \in K^d $. Fix an element $ \zeta \in K $ such that $ \overline{\zeta} = - \zeta $. \textbf{Assertion}: The function
\begin{equation*}
\presuperscript{\circ}{\phi} \longmapsto \{ (v,w) \mapsto \trace_{K/k} ( \zeta \cdot \presuperscript{\circ}{\phi} ( v, w ) ) \}
\end{equation*}
is a bijection between $ K / k $-hermitian forms $ K^d \times K^d \rightarrow K $ and internally-hermitian alternating $ k $-bilinear forms $ K^d \times K^d \rightarrow k $. The inverse function is
\begin{equation*}
\presuperscript{\circ}{\psi} \longmapsto \left\{ (v,w) \mapsto \frac{\zeta^{-1} \cdot \presuperscript{\circ}{\psi} ( v, w ) + \presuperscript{\circ}{\psi} ( \zeta^{-1} \cdot v, w )}{2} \right\}
\end{equation*}
\end{extractionofhermitianforms}

The following lemma will be used in \S\ref{SSsimpledesc} (page \pageref{SSsimpledesc}) to simplify and concretize the description of certain lattice-chains:
\begin{lemma} \label{Lgapfiller}
Fix a $ K / k $-hermitian form $ \presuperscript{\circ}{\phi} : K^d \times K^d \rightarrow K $ and let $ \presuperscript{\circ}{\psi} : K^d \times K^d \rightarrow k $ be the corresponding internally-hermitian alternating form. Note that the involution $ \presuperscript{\circ}{*} $ induced by $ \presuperscript{\circ}{\phi} $ on $ M_d(K) $ is the same as the one induced by $ \presuperscript{\circ}{\psi} $. Let $ \presuperscript{\circ}{G} $ be the usual similitude group associated to $ \presuperscript{\circ}{*} $, i.e. the functor assigning to each commutative $ k $-algebra $ R $ the group $ \presuperscript{\circ}{G} (R) \defeq \{ g \in M_d( K \otimes_k R ) \suchthat g^{\presuperscript{\circ}{*}} \cdot g \in R^{\times} \} $. \textbf{Assertion}: If $ k = \Q_{p^{\prime}} $ and $ p^{\prime} \neq 2 $, the following four statements are equivalent:
\begin{enumerate}
\setlength{\itemsep}{5pt}
\item \label{Lgapfillerquasisplit} $ \presuperscript{\circ}{G} $ is quasi-split

\item \label{Lgapfillerinvolution} there is a $ K $-algebra automorphism of $ M_d(K) $ transforming $ \presuperscript{\circ}{*} $ into the standard unitary involution $ X \mapsto \overline{X}^{\tr} $

\item \label{Lgapfillerhermitianform} there is a $ K $-linear automorphism of $ K^d $ transforming $ \presuperscript{\circ}{\phi} $ into the standard $ K / k $-hermitian form $ ( v, w ) \mapsto v^{\tr} \cdot \overline{w} $

\item \label{Lgapfillermaximalorder} the involution $ \presuperscript{\circ}{*} $ on $ M_d(K) $ stabilizes a maximal order
\end{enumerate}
\end{lemma}

\begin{proof}
\framebox{(\ref{Lgapfillerquasisplit}) $ \Leftarrow $ (\ref{Lgapfillerinvolution})} This is obvious from the definition of $ \presuperscript{\circ}{G} $. \framebox{(\ref{Lgapfillerinvolution}) $ \Leftrightarrow $ (\ref{Lgapfillerhermitianform})} This is immediate from the above ``Correspondence between Unitary Involutions and Hermitian Forms''. \framebox{(\ref{Lgapfillerinvolution}) $ \Rightarrow $ (\ref{Lgapfillermaximalorder})} This is trivial: by the Skolem-Noether Theorem, $ \presuperscript{\circ}{*} $ stabilizes the maximal order $ H \cdot M_d(\mathcal{O}) \cdot H^{-1} $ for some $ H $. \framebox{(\ref{Lgapfillerquasisplit}) $ \Rightarrow $ (\ref{Lgapfillerhermitianform})} This follows from the discussion following the above ``Classification of Unitary Involutions'': in the case of $ k = \Q_{p^{\prime}} $, the isomorphism class of the quasi-split unitary similitude group corresponds to the isometry class of the standard hermitian form. \framebox{(\ref{Lgapfillerinvolution}) $ \Leftarrow $ (\ref{Lgapfillermaximalorder})} This is just Proposition \ref{Palwaystranspose} (page \pageref{Palwaystranspose}).
\end{proof}

\begin{remark}
Of course, several of the implications in the lemma are true without one or both hypotheses.
\end{remark}

Because of this lemma, the following assumption will simplify the description of the local model considerably:
\begin{assumption} \label{Aquasisplit}
$ G_{\Q_p} $ is quasi-split.
\end{assumption}

\begin{remark}
This is merely a \emph{contextual} assumption; see the introduction. See \S\ref{SSsimpledesc} (page \pageref{SSsimpledesc}) for the application of this assumption and the lemma.
\end{remark}

\subsection{Determining the reflex field} \label{SSreflexfield}

Let $ D_{\C} \defeq D \otimes_{\Q} \C $ and let $ D_{+} $ and $ D_{-} $ be the $ +i $ and $ -i $ eigenspaces in $ D_{\C} $ of the linear operator $ h(i) $. Since $ h $ is $ \R $-linear and has the property from \S\ref{SSglobalPELdatum}, the minimal polynomial of $ h(i) $ divides $ T^2 + 1 $ and so $ D_{\C} = D_{+} \oplus D_{-} $. Note that acting by $ h(z) $ on $ D_{\C} $ is the same as acting by $ ( z, \overline{z} ) $ on $ D_{+} \oplus D_{-} $.

Define $ h_{\C} : \C \times \C \rightarrow D_{\C} $ to be the composite
\begin{equation*}
\C \times \C \directedisom \C \otimes_{\R} \C \stackrel{h \otimes \identity}{\longrightarrow} ( D \otimes_{\Q} \R ) \otimes_{\R} \C = D_{\C}
\end{equation*}
Note that $ h(z) \in G(\R) $ and $ h_{\C} ( z, 1 ) \in G( \C ) $ for $ z \in \C^{\times} $.

Define 
\begin{equation*}
\mu : \C^{\times} \longrightarrow G ( \C )
\end{equation*}
by $ z \mapsto h_{\C} ( z, 1 ) $. As an operator on $ D_{\C} $, this $ \mu(z) $ acts by $ z $ on $ D_{+} $ and by $ 1 $ on $ D_{-} $. In particular, $ \mu $ is \emph{minuscule}. The set of all $ \mu $ coming from all the $ h $ in a $ G(\R) $-conjugacy class is a $ G(\C) $-conjugacy class of cocharacters $ \Gm_{\C} \rightarrow G_{\C} $, and this conjugacy class is defined over some finite extension $ \globalreflex \supset \Q $, called the \emph{reflex field}.

On the other hand, if $ \epsilon, \overline{\epsilon} $ are the two embeddings $ F \hookrightarrow \C $ then there are $ \C $-algebra homomorphisms
\begin{align*}
D_{\C} &\directedisom ( D \otimes_{\epsilon} \C ) \times \opp{ ( D \otimes_{\overline{\epsilon}} \C ) } \\
d \otimes z &\longmapsto ( d \otimes z, d^* \otimes z )
\end{align*}
and of course each of these $ D \otimes_F \C $ is isomorphic as an $ \R $-algebra to $ M_d (\C) $, with the $ \C $-action depending on which embedding is used.

Since $ G(\C) \cong \GL_d(\C) \times \C^{\times} $, I can choose within the conjugacy class of cocharacters one whose image is in the diagonal torus. Using this cocharacter, the eigenspace $ D_{+} $ is, with respect to the decomposition
\begin{equation*}
D_{\C} = M_d (\C) \times \opp{ M_d (\C) },
\end{equation*}
the subspace consisting of the entries making up the top $ r $ rows (this is the \emph{definition} of $ r $) of the $ M_d (\C) $ factor together with the entries making up the bottom $ s := d - r $ rows of the $ \opp{ M_d (\C) } $ factor. Then
\begin{align*}
rd = \dim_{\C} ( D_{+} \cap ( D \otimes_{\epsilon} \C ) ) = \dim_{\C} ( D_{-} \cap \opp{ ( D \otimes_{\overline{\epsilon}} \C ) } ) \\
sd = \dim_{\C} ( D_{-} \cap ( D \otimes_{\epsilon} \C ) ) = \dim_{\C} ( D_{+} \cap \opp{ ( D \otimes_{\overline{\epsilon}} \C ) } )
\end{align*}

By page 274 in \cite{RZ}, one way to construct $ \globalreflex $ is to adjoin to $ \Q $ the traces $ \trace_{\C}( x; D_{-}) $ for all $ x \in D $. Pick $ x \in D $ and, recalling the classification of involutions over an algebraically-closed field, let $ ( X, \overline{X}^{\tr} ) $ be the image of $ x $ under
\begin{equation*}
D \longrightarrow D_{\C} = M_d(\C) \times \opp{ M_d(\C) }
\end{equation*}
With an appropriate choice of basis, the $ \C $-linear operation of $ x $ on $ D $ is given on rows by the matrix
\begin{equation*}
X \oplus \cdots \oplus X \textaftermath{($ d $ times)}
\end{equation*}
on $ M_d(\C) \hookrightarrow D_{\C} $ and by the matrix
\begin{equation*}
\overline{X}^{\tr} \oplus \cdots \oplus \overline{X}^{\tr} \textaftermath{($ d $ times)}
\end{equation*}
on $ \opp{ M_d(\C) } \hookrightarrow D_{\C} $. By the choice of $ \mu $ and the corresponding representation by rows of $ D_{+} $ and $ D_{-} $,
\begin{equation*}
\trace_{\C}( x; D_{-}) = s \trace_{\C} ( X ) + r \trace_{\C} ( \overline{X}^{\tr} ) = s \trace_{\C} ( X ) + r \overline{ \trace_{\C} ( X ) }
\end{equation*}
Since $ x \in D $, $ X \in M_d(F) \subset M_d(\C) $ and so $ \globalreflex \subset F $. It follows that if $ r = s $ then $ \trace_{\C}( x; D_{-}) \in \Q $ and $ \globalreflex = \Q $. Conversely, if $ r \neq s $ then there are certainly $ x \in D $ for which $ \trace_{\C}( x; D_{-}) \notin \Q $, and therefore $ \globalreflex = F $.

\begin{assumption}
$ \globalreflex = \Q $.
\end{assumption}

Note that the assumption $ \globalreflex = \Q $ forces $ d $ to be even.

\begin{remark}
This assumption is justified because, by an argument similar to that given in \S6.3.3 of \cite{Ha2}, the case of $ \globalreflex = F $ can be reduced to the case of $ \GL $, which is known by \cite{HN}.
\end{remark}

By using the embedding $ \overline{\Q} \hookrightarrow \overline{\Q}_p $, the cocharacter $ \mu $ defines a cocharacter
\begin{equation*}
\mu : \Gm_{\overline{\Q}_p} \longrightarrow G_{\overline{\Q}_p}
\end{equation*}
The $ G(\overline{\Q}_p) $-conjugacy class of this $ \mu $ is (at least) defined over the completion $ \localreflex $ (still called the reflex field) of $ \globalreflex $ at the prime corresponding to $ \globalreflex \hookrightarrow
\overline{\Q}_p $. The cocharacter $ \mu $ itself is split by some (possibly non-trivial) extension $ \localreflex^{\prime} \supset \localreflex $ and defines a similar weight decomposition
\begin{equation*}
M_d(F_p) \otimes_{\Q_p} \localreflex^{\prime} = M_d(F_p)_{+} \oplus M_d(F_p)_{-}
\end{equation*}
as above.

\section{Some commutative algebra} \label{Scommalglemmas}

The following will be used many times throughout the paper:

\begin{projectivitylemma} [Theorem 1 on page 109 of \cite{Bo}] \label{projectivitylemma}
Let $ \mathcal{R} $ be a commutative ring and let $ M $ be an $ \mathcal{R} $-module. \textbf{Assertion}: The following three statements are equivalent:
\begin{enumerate}
\setlength{\itemsep}{5pt}
\item \label{PROprojective} $ M $ is finitely-generated and projective.

\item \label{PROzariskifree} There are $ s_1, \ldots, s_n \in \mathcal{R} $ generating the trivial ideal $ \mathcal{R} $ such that each principal
module of fractions $ M_{s_i} $ is a free finite-rank $ \mathcal{R}_{s_i} $-module.

\item \label{PROptwisefree} $ M $ is finitely-generated and for every prime $ \mathfrak{p} \subset \mathcal{R} $, the module of fractions $ M_{\mathfrak{p}} $ is a free finite-rank $ \mathcal{R}_{\mathfrak{p}} $-module and the function $ \spec(\mathcal{R}) \rightarrow \N $ defined by $ \mathfrak{p} \mapsto \rank_{ \mathcal{R}_{\mathfrak{p}} } ( M_{\mathfrak{p}} ) $ is locally constant with respect to the Zariski topology.
\end{enumerate}
\end{projectivitylemma}

\begin{terminologynotation}
 I frequently refer to the function in (\ref{PROptwisefree}) as the ``projective rank function'' of $ M $. I frequently express the property in characterization (\ref{PROzariskifree}) by saying that $ M $ is ``Zariski-locally on $ \spec( \mathcal{R} ) $'' a free and finite-rank $ \mathcal{R} $-module or something similar.
\end{terminologynotation}

\begin{permanenceoffinitepresentedness} [Lemma 9 on page 21 of \cite{Bo}]
Let $ R $ be a commutative ring and $ 0 \rightarrow N \rightarrow M \rightarrow Q \rightarrow 0 $ an exact sequence of $ R $-modules. \textbf{Assertion}: If $ M $ is finitely-\emph{generated} and $ Q $ is finitely-\emph{presented}, then $ N $ is finitely-\emph{generated}.

In other words, \emph{any} finite set of generators of a finitely-presented module is automatically a finite-presentation.
\end{permanenceoffinitepresentedness}

The following handles a slight complication special to the case of \emph{unramified} unitary groups (see the proof of Proposition \ref{Paffineflagcontainment} (page \pageref{Paffineflagcontainment})):
\begin{lemma} \label{Lprojectivitylifting}
Let $ R $ and $ S $ be commutative rings, $ f : R \rightarrow S $ a ring homomorphism, and $ M $ a $ S $-module. Regard all $ S $-modules as $ R $-modules via $ f $. \textbf{Assertion}: If (1) $ S $ is a finitely-generated $ R $-module, (2) $ M $ is a finitely-generated $ R $-module, (3) $ M $ is a projective $ R $-module, and (4) $ f $ is faithfully flat, then $ M $ is a projective $ S $-module.
\end{lemma}

\begin{proof}
It is equivalent (see Corollary 2 on page 111 of \cite{Bo}) to prove that $ M $ is a flat and finitely-presented $ S $-module. Projective modules are flat so by (3) and (4), $ M $ is then $ S $-flat. By (2) and the definition of the $ R $-action, $ M $ is finitely-generated over $ S $. Let $ S^k \twoheadrightarrow M $ be an $ S $-module presentation and let $ N $ be the kernel (a $ S $-module). By (1), $ S^k $ is finitely-generated over $ R $. Finitely-generated projective modules are finitely-presented, so by ``Permanence of Finite-Presentedness'' $ N $ is finitely-generated over $ R $. As before, this means that $ N $ is finitely-generated over $ S $. This means that $ M $ is a finitely-presented $ S $-module.
\end{proof}

\begin{homlocalization} [Proposition 2.10 on page 68 of \cite{Ei}] \label{labelhomlocalization}
Let $ A $ be a commutative ring, $ \mathfrak{p} \subset A $ a prime, and $ M, N $ two $ A $-modules. \textbf{Assertion}: If $ M $ is finitely-presented then the $ A_{\mathfrak{p}} $-linear map $ \myhom_{A\textup{-lin}} ( M, N )_{\mathfrak{p}} \rightarrow \myhom_{A_{\mathfrak{p}}\textup{-lin}} ( M_{\mathfrak{p}}, N_{\mathfrak{p}} ) $ defined by $ f/s \mapsto \{ m/t \mapsto f(m)/st \} $ is an isomorphism.

More generally, if $ A^{\prime} $ is a flat $ A $-algebra and $ M $ is a finitely-presented $ A $-module then $ \myhom_{A\textup{-lin}} ( M, N ) \otimes_A A^{\prime} \rightarrow \myhom_{A^{\prime}\textup{-lin}} ( M \otimes_A A^{\prime}, N \otimes_A A^{\prime} ) $ is an $ A^{\prime} $-linear isomorphism. If $ M $ is free, then the flatness hypothesis on $ A^{\prime} $ is (obviously) unnecessary.
\end{homlocalization}

In many papers about local models, a submodule is sometimes assumed to be ``Zariski-locally a direct summand'' or similar. Actually, this assumption is usually equivalent to the assumption that the submodule be a direct summand, period:

\begin{lemma} \label{Llocallysummand}
Let $ R $ be a commutative ring and let $ 0 \rightarrow N \rightarrow M \rightarrow Q \rightarrow 0 $ be a short-exact-sequence of $ R $-modules. Assume that $ Q $ is finitely-presented. \textbf{Assertion}: The sequence splits if and only if for every prime $ \mathfrak{p} \subset R $ the localized sequence $ 0 \rightarrow N_{\mathfrak{p}} \rightarrow M_{\mathfrak{p}} \rightarrow Q_{\mathfrak{p}} \rightarrow 0 $ also splits.

In particular, if $ M $ is free and finite-rank and $ N $ is finitely-generated, then $ N \subset M $ is a direct summand if and only if it is a direct summand Zariski-locally on $ \spec(R) $ (any Zariski-local property implies the corresponding local property).
\end{lemma}

\begin{proof}
\framebox{$ \Rightarrow $} This is trivial. \framebox{$ \Leftarrow $} The short-exact-sequence is split if and only if the induced homomorphism $ \linhomom{R} ( Q, M ) \rightarrow \linhomom{R} ( Q, Q ) $ is surjective. Since surjectivity is a local property, this homomorphism is surjective if and only if all the localized homomorphisms $ \linhomom{R} ( Q, M )_{\mathfrak{p}} \rightarrow \linhomom{R} ( Q, Q )_{\mathfrak{p}} $ are surjective. By Localization of Hom-Sets, the localized homomorphisms are really the same as the induced homomorphisms $ \linhomom{R_{\mathfrak{p}}} ( Q_{\mathfrak{p}}, M_{\mathfrak{p}} ) \rightarrow \linhomom{R_{\mathfrak{p}}} ( Q_{\mathfrak{p}}, Q_{\mathfrak{p}} ) $. The hypothesis is exactly that these localized homomorphisms are surjective.
\end{proof}

%\begin{remark}This lemma can also be proved using the ``Local Criteria for Projectivity''.\end{remark}

I frequently use the following elementary but extremely useful consequence of Nakayama's Lemma:

\begin{minimalindependence} [Exercise 15 in Chapter 3 of \cite{AM}]
Let $ R $ be a commutative ring and let $ M $ be a free and finite-rank $ R $-module with rank equal to $ n $. If $ x_1, \ldots, x_n $ generates $ M $ then it is automatically a basis. Equivalently, any $ R $-linear surjection $ M \twoheadrightarrow M $ is automatically an isomorphism.
\end{minimalindependence}

\section{The local model} \label{Sdefinitionofthelocalmodel}

Recall from \S\ref{SSalwaystranspose} that I have identified
\begin{align*}
B \otimes_{\Q} \Q_p &= \opp{ M_d ( F_p ) } \\
\mathcal{O}_B \otimes_{\Z} \Z_p &= \opp{ M_d ( \mathcal{O}_{F_p} ) } \\
\iota_p &= \conjtr
\end{align*}

\begin{terminologynotation}
Now that the relationship between $ ( D, * ) $ and $ G $ is clear and I have restricted attention to $ \localreflex = \Q_p $, \emph{I will no longer refer to the global objects}. For simplicity of notation, refer to $ F_p, \mathcal{O}_{F_p}, *_p, \psi_p $ etc. simply as $ F, \mathcal{O}, *, \psi $ etc. I will now use $ r = s = d/2 $ without warning.
\end{terminologynotation}

\subsection{Definition of the local model} \label{SSoriginallocalmodeldef}

Here I recall the definition of the local model $ \localmodel : \catalgebras{\Z_p} \rightarrow \catsets $. The fact that the domain is $ \catalgebras{\Z_p} $ is due to the fact that $ \localreflex = \Q_p $, which implicitly depends on the cocharacter $ \mu $.

To give the definition of the local model, I need to select a certain doubly-infinite chain $ \Lambda_{\bullet} = ( \cdots \subset \Lambda_{-1} \subset \Lambda_0 \subset \Lambda_1 \subset \cdots ) $ of $ \mathcal{O}_B $-lattices in $ V $, i.e. a chain of finitely-generated \emph{left}-$ \opp{ M_d(\mathcal{O}) } $-submodules of $ M_d(F) $, each of which spans $ M_d(F) $ as an $ F $-vector space. The lattice chain is required to satisfy the following two conditions:
\begin{itemize}
\setlength{\itemsep}{5pt}
\item The chain is ``periodic'', in the sense that for any $ \Lambda $ in the chain, $ p \Lambda
$ is also in the chain.

\item The chain is ``$ \psi $-selfdual'', in the sense that for every $ \Lambda $ in
the chain, $ \widehat{\Lambda} $ (see \S\ref{SPELdatum},
page \pageref{SPELdatum}) is also in the chain.
\end{itemize}
See Definition 3.1 on page 69-70 and Definition 3.18 on page 77-78 of \cite{RZ}. I will define a particularly convenient such lattice chain in \S\ref{SSsimpledesc} (page \pageref{SSsimpledesc}). For now, assume that such a lattice chain has been selected.

\begin{textbooklocalmodeldef} [Definition 3.27 on page 89 of \cite{RZ}]
The functor
\begin{equation*}
\localmodel : \catalgebras{\Z_p} \rightarrow \catsets
\end{equation*}
assigns to each commutative $ \Z_p $-algebra $ R $ the set of (isomorphism classes of) commutative diagrams
\begin{equation*}
\begin{CD}
\cdots @>\inc \otimes \identity>> \Lambda_i \otimes_{\Z_p} R @>\inc \otimes \identity>> \Lambda_{i+1} \otimes_{\Z_p} R @>\inc \otimes \identity>> \cdots \\
@. @VVV @VVV @. \\
\cdots @>>> T_i @>>> T_{i+1} @>>> \cdots
\end{CD}
\end{equation*}
of $ ( \opp{ M_d( \mathcal{O} ) } \otimes_{\Z_p} R ) $-modules satisfying the following properties:
\begin{enumerate}
\setlength{\itemsep}{5pt}
\item \label{LMquotient} \OLMdef{\ref{LMquotient}} \hfill \\
For each $ i $, $ \Lambda_i \otimes_{\Z_p} R \rightarrow T_i $ is \emph{surjective},

\item \label{LMperiodic} \OLMdef{\ref{LMperiodic}} \hfill \\
for each $ i $, if $ p \Lambda_i = \Lambda_j $, then the isomorphism $ - \cdot p : \Lambda_i \directedisom \Lambda_j $ descends to an isomorphism $ T_i \directedisom T_j $,

\item \label{LMlocallyfree} \OLMdef{\ref{LMlocallyfree}} \hfill \\
for each $ i $, $ T_i $ is Zariski-locally on $ \spec(R) $ a free and finite-rank $ R $-module,
(i.e. $ T_i $ is $ R $-projective, by the ``Local Criteria for Projectivity'')

\item \label{LMcharpoly} \OLMdef{\ref{LMcharpoly}} \hfill \\
for each $ i $, $ \mydet_R ( b ; T_i ) = \mydet_{ \localreflex^{\prime} }( b ; M_d(F)_{-} ) $ for all $ b \in M_d(\mathcal{O}) $, and

(recall that $ \localreflex^{\prime} $ acts via $ M_d(F)_{-} \hookrightarrow M_d(F) \otimes_{\Q_p} \localreflex^{\prime} $)

\item \label{LMduality} \OLMdef{\ref{LMduality}} \hfill \\
for each $ \Lambda = \Lambda_i $, the composite $ T_{\Lambda}^{\vee} \rightarrow ( \Lambda \otimes_{\Z_p} R )^{\vee} \cong \widehat{\Lambda} \otimes_{\Z_p} R \rightarrow T_{\widehat{\Lambda}} $ is the $ 0 $ map.

(here $ T_{\Lambda} = T_i $ and $ \vee = \linhomom{R} ( -, R ) $)
\end{enumerate}
\end{textbooklocalmodeldef}

It is perhaps helpful to comment on \OLMdef{\ref{LMcharpoly}}. Let $ R $ be a commutative ring, let $ M $ be a finitely-generated \emph{projective} $ R $-module, and let $ b : M \rightarrow M $ be an $ R $-linear endomorphism. Recall from the ``Local Projectivity Criteria'' (page \pageref{projectivitylemma}) that $ M $ is Zariski-locally free. If $ S \subset R $ is a multiplicative subset such that $ S^{-1} M $ is a free $ S^{-1} R $-module, then the endomorphism $ S^{-1} b : S^{-1} M \rightarrow S^{-1} M $ induced by $ b $ has a determinant $ \mydet_{ S^{-1} R } ( S^{-1} b ; S^{-1} M ) \in S^{-1} R $. Covering $ \spec ( R ) $ by multiplicative subsets $ S $ as above gives a collection of determinants which patch together into a global determinant $ \mydet_R ( b ; M ) \in R $.

\subsection{Changing from quotients to subobjects}
\label{SSquotientstosubmodules}

Notice that in the definition of $ \localmodel $ specifying the object $ T_i $ is ``almost'' the same as specifying the kernel $ K_i \defeq \kernel ( \Lambda_i \otimes_{\Z_p} R \twoheadrightarrow T_i ) $. To more closely match the description of affine flag varieties, I need to express the points of $ \localmodel $ using the $ K_i $ instead of the $ T_i $. The equivalent conditions that must be imposed are provided by the next proposition.

\begin{prop} \label{Pchangefromquotientstosubobjects}
The local model $ \localmodel $ is isomorphic to the functor that assigns to each commutative $ \Z_p $-algebra $ R $ the set of all (isomorphism classes of) commutative diagrams
\begin{equation*}
\begin{CD}
\cdots @>\inc \otimes \identity>> \Lambda_i \otimes_{\Z_p} R @>\inc \otimes \identity>> \Lambda_{i+1} \otimes_{\Z_p} R @>\inc \otimes \identity>> \cdots \\
@. @AAA @AAA @. \\
\cdots @>>> K_i @>>> K_{i+1} @>>> \cdots
\end{CD}
\end{equation*}
of $ \opp{ M_d( \mathcal{O} ) } \otimes_{\Z_p} R $-modules satisfying the following properties:
\begin{enumerate}
\setlength{\itemsep}{5pt}
\item \label{subOLMquotient} \subOLMdef{\ref{subOLMquotient}} \hfill \\
For each $ i $, $ K_i \rightarrow \Lambda_i \otimes_{\Z_p} R $ is \emph{injective},

\item \label{subOLMperiodic} \subOLMdef{\ref{subOLMperiodic}} \hfill \\
for each $ i $, if $ p \Lambda_i = \Lambda_j $, then the isomorphism $ - \cdot p : \Lambda_i \directedisom \Lambda_j $ restricts to an isomorphism $ K_i \directedisom K_j $,

\item \label{subOLMlocallyfree} \subOLMdef{\ref{subOLMlocallyfree}} \hfill \\
for each $ i $, $ K_i \hookrightarrow \Lambda_i \otimes_{\Z_p} R $ splits $ R $-linearly,

\item \label{subOLMcharpoly} \subOLMdef{\ref{subOLMcharpoly}} \hfill \\
for each $ i $, $ \mydet_R ( b ; K_i ) = \mydet_{\localreflex^{\prime}} ( b ; M_d(F)_{+} ) $ for all $ b \in M_d(\mathcal{O}) $, and

\item \label{subOLMduality} \subOLMdef{\ref{subOLMduality}} \hfill \\
$ K_{\widehat{\Lambda}} = K_{\Lambda}^{\perp} $ for each $ \Lambda = \Lambda_i $.

(here $ K^{\perp} \defeq \{ \lambda \in \widehat{\Lambda} \otimes_{\Z_p} R \suchthat \psi_R ( K, \lambda ) = 0 \} $ for any $ K \subset \Lambda \otimes_{\Z_p} R $)
\end{enumerate}
\end{prop}

Property \subOLMdef{\ref{subOLMcharpoly}} will soon be replaced by a very simple rank requirement; see \S\ref{SStranslatingdeterminant} (page \pageref{SStranslatingdeterminant}).

There is a slight abuse of notation in \subOLMdef{\ref{subOLMduality}}, in that $ K_{\Lambda}^{\perp} $ is required to be equal to the image of $ K_{\widehat{\Lambda}} $ in $ \widehat{\Lambda} \otimes_{\Z_p} R $.

\begin{proof}
Fix a commutative $ \Z_p $-algebra $ R $, a point $ \Lambda_{\bullet} \otimes_{\Z_p} R \twoheadrightarrow T_{\bullet} $ of $ \localmodel ( R ) $, and consider the obvious analogous diagram $ K_{\bullet} \hookrightarrow \Lambda_{\bullet} \otimes_{\Z_p} R $ of kernels $ K_i \defeq \kernel ( \Lambda_i \otimes_{\Z_p} R \twoheadrightarrow T_i ) $. It is immediate from the ``Five Lemma'' that $ \Lambda_{\bullet} \otimes_{\Z_p} R \twoheadrightarrow T_{\bullet} $ satisfies \OLMdef{\ref{LMperiodic}} if and only if $ K_{\bullet} \hookrightarrow \Lambda_{\bullet} \otimes_{\Z_p} R $ satisfies \subOLMdef{\ref{subOLMperiodic}} above, and properties \OLMdef{\ref{LMquotient}} and \subOLMdef{\ref{subOLMquotient}} require no discussion. By the ``Local Projectivity Criteria'', \OLMdef{\ref{LMlocallyfree}} is equivalent to finitely-generated projectivity, and it is trivial that $ K_i \hookrightarrow \Lambda_i \otimes_{\Z_p} R $ splits $ R $-linearly if and only if $ \Lambda_i \otimes_{\Z_p} R \twoheadrightarrow T_i $ splits $ R $-linearly. Therefore, $ \Lambda_{\bullet} \otimes_{\Z_p} R \twoheadrightarrow T_{\bullet} $ satisfies \OLMdef{\ref{LMlocallyfree}} if and only if $ K_{\bullet} \hookrightarrow \Lambda_{\bullet} \otimes_{\Z_p} R $ satisfies \subOLMdef{\ref{subOLMlocallyfree}} above. It is clear from the definitions that $ \mydet_{\localreflex^{\prime}} ( b; M_d(F) ) = \mydet_{\localreflex^{\prime}} ( b; M_d(F)_{-} ) \cdot \mydet_{\localreflex^{\prime}} ( b; M_d(F)_{+} ) $ and that $ \mydet_R (b; \Lambda_i \otimes_{\Z_p} R) = \mydet_R (b; T_i ) \cdot \mydet_R (b; K_i ) $ so $ \Lambda_{\bullet} \otimes_{\Z_p} R \twoheadrightarrow T_{\bullet} $ satisfies \OLMdef{\ref{LMcharpoly}} if and only if $ K_{\bullet} \hookrightarrow \Lambda_{\bullet} \otimes_{\Z_p} R $ satisfies \subOLMdef{\ref{subOLMcharpoly}} above. Finally, fix a lattice $ \Lambda \in \Lambda_{\bullet} $. An element of $ T_{\Lambda}^{\vee} $ is the same as an $ R $-linear map $ f : \Lambda \rightarrow R $ whose kernel contains $ K_{\Lambda} $. Since $ \Lambda $ is a \emph{lattice}, any such functional $ f $ is of the form $ \psi_R ( -, \lambda ) $ for some $ \lambda \in \widehat{\Lambda} \otimes_{\Z_p} R $, and the requirement that $ K_{\Lambda} \subset \kernel ( f ) $ is equivalent to the requirement that $ \lambda \in K_{\Lambda}^{\perp} $. Therefore, $ \Lambda_{\bullet} \otimes_{\Z_p} R \twoheadrightarrow T_{\bullet} $ satisfies \OLMdef{\ref{LMduality}} if and only if $ K_{\bullet} \hookrightarrow \Lambda_{\bullet} \otimes_{\Z_p} R $ satisfies $ K_{\widehat{\Lambda}} \supset K_{\Lambda}^{\perp} $ for all $ \Lambda \in \Lambda_{\bullet} $. The Corollary to the lemma following this proof, Lemma \ref{Lreverseinclusionprover}, implies that the reverse inclusion $ K_{\widehat{\Lambda}} \subset K_{\Lambda}^{\perp} $ is true automatically and therefore \OLMdef{\ref{LMduality}} is equivalent to \subOLMdef{\ref{subOLMduality}} above.

Let $ \mathbf{M} $ be the functor described in this proposition. The previous paragraph shows that the assignment of the diagram $ K_{\bullet} \hookrightarrow \Lambda_{\bullet} \otimes_{\Z_p} R $ to the point $ \Lambda_{\bullet} \otimes_{\Z_p} R \twoheadrightarrow T_{\bullet} $ defines a natural transformation $ S : \localmodel \rightarrow \mathbf{M} $. Similarly, there is also a natural transformation $ Q : \mathbf{M} \rightarrow \localmodel $ which assigns to each point $ K_{\bullet} \hookrightarrow \Lambda_{\bullet} \otimes_{\Z_p} R $ of $ \mathbf{M} ( R ) $ the obvious diagram consisting of the canonical quotient maps $ \Lambda_i \otimes_{\Z_p} R \twoheadrightarrow \Lambda_i \otimes_{\Z_p} R / K_i $. Given a point $ \Lambda_{\bullet} \otimes_{\Z_p} R \twoheadrightarrow T_{\bullet} $ of $ \localmodel ( R ) $, an easy diagram chase produces an isomorphism between the diagrams $ \Lambda_{\bullet} \otimes_{\Z_p} R \twoheadrightarrow T_{\bullet} $ and $ \Lambda_{\bullet} \otimes_{\Z_p} R \twoheadrightarrow \Lambda_{\bullet} \otimes_{\Z_p} R / K_{\bullet} $. In other words, the diagrams represent the same point of $ \localmodel ( R ) $ and so $ Q \circ S = \identity_{\localmodel} $. A similar diagram chase shows that $ S \circ Q = \identity_{\mathbf{M}} $ also. Therefore, $ \localmodel $ is isomorphic to $ \mathbf{M} $, as desired.
\end{proof}

\begin{terminologynotation}
From now on, identify $ \localmodel $ with the functor described in Proposition \ref{Pchangefromquotientstosubobjects}.
\end{terminologynotation}

\begin{lemma} \label{Lreverseinclusionprover}
Fix a commutative $ \Z_p $-algebra $ R $ and an $ \opp{ M_d ( \mathcal{O} ) } $-lattice $ \Lambda \subset M_d(F) $. Let $ K \subset \widehat{\Lambda} \otimes_{\Z_p} R $ be an $ ( \opp{ M_d ( \mathcal{O} ) } \otimes_{\Z_p} R ) $-submodule that is, as an $ R $-module, a direct summand. \textbf{Assertion}: The sequence
\begin{equation*}
0 \longrightarrow K^{\perp} \stackrel{\inc}{\longrightarrow} \widehat{\Lambda} \otimes_{\Z_p} R \stackrel{f}{\longrightarrow} \linhomom{R} ( K, R ) \longrightarrow 0,
\end{equation*}
where $ f ( x ) \defeq \psi_R ( -, x ) \vert_K $, is exact and split. In particular, $ K^{\perp} $ is $ R $-projective and the sum of the projective rank functions $ \spec ( R ) \rightarrow \N $ of $ K $ and $ K^{\perp} $ is the constant function $ \mathfrak{p} \mapsto 2 d^2 $.
\end{lemma}

\begin{proof}
It is obvious from the definition of $ K^{\perp} $ that the sequence is exact on the left and in the middle. To see that the sequence is both right-exact and split, I construct a splitting $ s : \linhomom{R} ( K, R ) \hookrightarrow \widehat{\Lambda} \otimes_{\Z_p} R $ of $ f $. By the `summand' hypothesis, fix a splitting $ p : \Lambda \otimes_{\Z_p} R \twoheadrightarrow K $. Using $ p $, any $ R $-linear functional $ \varphi $ on $ K $ extends to an $ R $-linear functional $ \widetilde{\varphi} $ on $ \Lambda \otimes_{\Z_p} R $. Since $ \psi : \Lambda \times \widehat{\Lambda} \rightarrow \Z_p $ is a perfect pairing, since $ - \otimes_{\Z_p} R $ is a right-exact functor, and since $ \linhomom{\Z_p} ( \Lambda, \Z_p ) \otimes_{\Z_p} R \cong \linhomom{R} ( \Lambda \otimes_{\Z_p} R, R ) $ by the trivial case of ``Localization of Hom-Sets'', any such $ \widetilde{\varphi} $ is of the form $ \psi_R ( -, x_{\varphi} ) $ for some $ x_{\varphi} \in \widehat{\Lambda} \otimes_{\Z_p} R $. Define $ s $ by $ s ( \varphi ) \defeq x_{\varphi} $. It is obvious from the construction that $ s $ is $ R $-linear and that $ f \circ s = \identity $, as desired. The projectivity statement is then obvious because $ \Lambda $ is a free $ \Z_p $-module and the statement about rank functions is obvious from the short-exact-sequence because $ \rank_{\Z_p} ( \Lambda ) = 2 d^2 $.
\end{proof}

\begin{corollary}
Fix a commutative $ \Z_p $-algebra $ R $, let $ \Lambda_{\bullet} \otimes_{\Z_p} R \twoheadrightarrow T_{\bullet} $ be a point of $ \localmodel ( R ) $. Fix $ i \in \Z $ and set $ \Lambda \defeq \Lambda_i $ and $ K \defeq \kernel ( \Lambda_i \otimes_{\Z_p} R \twoheadrightarrow T_i ) $. Assume that $ K^{\perp} \subset K $. \textbf{Assertion}: $ K^{\perp} = K $.
\end{corollary}

\begin{proof}
The hypothesis of Lemma \ref{Lreverseinclusionprover} is supplied by \OLMdef{\ref{LMlocallyfree}}, so the domain and codomain of the induced surjection $ c : \widehat{\Lambda} \otimes_{\Z_p} R / K^{\perp} \twoheadrightarrow \Lambda \otimes_{\Z_p} R / K $ are both $ R $-projective and, because \OLMdef{\ref{LMcharpoly}} forces the projective rank function of $ K $ to be $ \mathfrak{p} \mapsto d^2 $, both the domain and codomain of $ c $ have the same projective rank function $ \mathfrak{p} \mapsto d^2 $. Localizing and applying ``Linear Independence of Minimal Generating Sets'' shows that $ c $ is an isomorphism. A typical ``Five Lemma'' argument then shows that $ K^{\perp} = K $.
\end{proof}

\section{Compressing the local model with Morita Equivalence} \label{Smorita}

\begin{center}
\emph{The points of the local model $ \localmodel $ are unnecessarily complicated, due to the presence of modules over matrix rings. My goal in this section is to use Morita Equivalence to work within $ \catmodules{ \mathcal{O} } $ instead of $ \catmodules{ M_d ( \mathcal{O} ) } $.}
\end{center}

Let $ \mathcal{R} $ be a commutative ring. \emph{Morita Equivalence} is the fact that the functor
\begin{equation*}
\morita_{\mathcal{R}} \defeq \linhomom{\mathcal{R}} ( \mathcal{R}^d, - ) : \catmodules{\mathcal{R}} \rightarrow \textup{(right)}\catmodules{M_d(\mathcal{R})}
\end{equation*}
is an equivalence-of-categories (the action of $ M_d(\mathcal{R}) $ is by precomposition). An explicit quasi-inverse to $ \morita_{\mathcal{R}} $ is the functor
\begin{equation*}
\morita_{\mathcal{R}}^{-1} \defeq - \otimes_{M_d(\mathcal{R})} \mathcal{R}^d
\end{equation*}

Roughly speaking, I simply want to replace the local model $ \localmodel $ by the functor which assigns to any commutative $ \Z_p $-algebra $ R $ the set of all diagrams $ \morita^{-1}_{ \mathcal{O} \otimes_{\Z_p} R } ( \Delta ) $ for all $ \Delta \in \localmodel ( R ) $. But it will be necessary to have an explicit description of the points of this new functor as certain diagrams satisfying certain conditions, and it is not immediately obvious what these conditions should be because conditions \subOLMdef{\ref{subOLMcharpoly}} and \subOLMdef{\ref{subOLMduality}} are not directly digestible by Morita equivalence. The goal of this section is to resolve these issues.

Throughout this section, I identify $ M_d(F) $ with $ \morita_{\mathcal{O}} ( F^d ) $ via the natural isomorphism $ \linhomom{\mathcal{O}} ( \mathcal{O}^d, F^d ) \directedisom \linhomom{F} ( F^d, F^d ) $ which sends any $ \varphi $ to its obvious $ F $-linear extension. Via this identification, $ \morita_{\mathcal{O}} ( M ) $ is safely considered to be an $ M_d(\mathcal{O}) $-submodule of $ M_d(F) $ for any $ \mathcal{O} $-submodule $ M \subset F^d $.

\subsection{Simplifying the determinant condition} \label{SStranslatingdeterminant}

\begin{center}
\emph{In this subsection, I show that the determinant condition \subOLMdef{\ref{subOLMcharpoly}} can be replaced by a straightforward rank requirement. Because $ \rank ( \morita ( * ) ) = d \cdot \rank ( * ) $, this allows \subOLMdef{\ref{subOLMcharpoly}} to be transformed by $ \morita^{-1} $.}
\end{center}

Let $ R $ be a commutative $ \Z_p $-algebra, set $ \mathcal{R} := \mathcal{O} \otimes_{\Z_p} R $, and consider a point $ \{ K_i \}_{i \in \Z} \in \localmodel ( R ) $. Condition \subOLMdef{\ref{subOLMcharpoly}} certainly forces the projective rank function $ \spec(R) \rightarrow \N $ of each $ K_i $ to be the constant function $ \mathfrak{p} \mapsto d(2s) = d^2 $.

The converse is also true, and is the Corollary to the following very general proposition:
\begin{prop} \label{Psamerankimpliessamedeterminant}
Let $ \mathcal{R} $ be a commutative $ \mathcal{O} $-algebra. Let $ M $ and $ N $ be left-$ M_d( \mathcal{R} ) $-modules that are projective as $ \mathcal{R} $-modules. \textbf{Assertion}: If $ M $ and $ N $ have the same projective rank functions $ \spec( \mathcal{R} ) \rightarrow \N $ then $ \mydet_{\mathcal{R}} ( b ; M ) = \mydet_{\mathcal{R}} ( b ; N ) $ for all $ b \in M_d( \mathcal{R} ) $.
\end{prop}

\begin{proof}
Let $ M^{\prime} $ and $ N^{\prime} $ be $ \mathcal{R} $-modules such that $ \morita_{\mathcal{R}} ( M^{\prime} ) = M $ and $ \morita_{\mathcal{R}} ( N^{\prime} ) = N $. Since $ M = M^{\prime} \oplus \cdots \oplus M^{\prime} $ and $ N = N^{\prime} \oplus \cdots \oplus N^{\prime} $ as $ \mathcal{R} $-modules, it is true that $ M^{\prime} $ and $ N^{\prime} $ are also projective as $ \mathcal{R} $-modules. Let $ S \subset \mathcal{R} $ be a multiplicative subset such that $ S^{-1} M^{\prime} $ and $ S^{-1} N^{\prime} $ are both free. Since $ \rank ( \morita ( * ) ) = d \cdot \rank ( * ) $, the hypothesis implies that $ \rank_{ S^{-1} \mathcal{R} } ( S^{-1} M^{\prime} ) = \rank_{ S^{-1} \mathcal{R} } ( S^{-1} N^{\prime} ) $. Therefore, $ S^{-1} M^{\prime} $ and $ S^{-1} N^{\prime} $ are isomorphic as $ S^{-1} \mathcal{R} $-modules. This means that $ S^{-1} M $ and $ S^{-1} N $ are isomorphic as right-$ M_d ( S^{-1} \mathcal{R} ) $-modules, hence also as right-$ M_d( \mathcal{R} ) $-modules (via restriction-of-scalars). Therefore, for any $ b \in M_d( \mathcal{R} ) $, $ \mydet_{ S^{-1} \mathcal{R} } ( b ; S^{-1} M ) = \mydet_{ S^{-1} \mathcal{R} } ( b ; S^{-1} N ) $. Varying $ S $ to get an open cover of $ \spec(\mathcal{R}) $, and patching the generic determinants together finishes the proof.
\end{proof}

%\begin{remark} The claim can also be proven by a short matrix computation. \end{remark}

\begin{corollary}
$ K_{\bullet} \in \localmodel ( R ) $ satisfies \subOLMdef{\ref{subOLMcharpoly}} if and only if $ \rank_R ( K_i ) = d^2 $ for all $ i $.
\end{corollary}

\begin{proof}
Let $ i $ be arbitrary, set $ M = K_i $ and $ N = \mathcal{R}^d \oplus \cdots \oplus \mathcal{R}^d $ ($ d/2 $ times) and use Lemma \ref{Lprojectivitylifting} (page \pageref{Lprojectivitylifting}) to supply the projectivity hypothesis on $ M $. Proposition \ref{Psamerankimpliessamedeterminant} implies that for $ b \in M_d(\mathcal{O}) $, $ \mydet_{\mathcal{R}} ( b ; K_i ) $ is just the product of $ d/2 $ copies of the ``ordinary'' $ \mathcal{O} $-valued determinant of $ b $. One can easily compute $ \mydet_{\localreflex^{\prime}} ( b; M_d(F)_{+} ) $ and see that the two are equal.
\end{proof}

\begin{terminologynotation}
From now on, refer to the condition in the Corollary as \subOLMdef{\ref{subOLMcharpoly}} and forget the original condition.
\end{terminologynotation}

\subsection{Morita Equivalence and products} \label{SSmoritaandproducts}

\begin{center}
\emph{In this subsection, I verify that for each product like $ \psi $ there is a product $ \presuperscript{\circ}{ \psi } $ naturally corresponding to $ \psi $ under Morita Equivalence and properties like non-degeneracy etc. are maintained by the correspondence. This is one step towards transforming \subOLMdef{\ref{LMduality}} by $ \morita ^{-1} $.}
\end{center}

Fix a commutative $ \Z_p $-algebra $ R $ and let $ S $ be a commutative $ R $-algebra. Let $ s \mapsto \overline{s} $ be an involution on $ S $ such that $ \overline{r} = r $ for all $ r \in R $. Extend this involution to $ M_d(S) $ by acting individually on entries.

Fix an $ S $-module $ N^{\prime} $. Give $ \myhom_{R\textup{-lin}} ( N^{\prime}, R ) $ an $ S $-module structure by the rule
\begin{equation*}
( s \cdot g ) ( n ) \defeq g ( \overline{s} n )
\end{equation*}
($ s \in S $, $ g \in \myhom_{R\textup{-lin}} ( N^{\prime}, R ) $, $ n \in N^{\prime} $). Give $ \myhom_{R\textup{-lin}} ( \myhom_{S\textup{-lin}} ( S^d, N^{\prime} ), R ) $ a \emph{right}-$ M_d (S) $-module structure by the rule
\begin{equation*}
( F \cdot b ) ( \varphi ) \defeq F ( \varphi \circ \overline{b}^{\tr} )
\end{equation*}
($ F \in \myhom_{R\textup{-lin}} ( \myhom_{S\textup{-lin}} ( S^d, N^{\prime} ), R ) $, $ b \in M_d(S) $, $ \varphi \in \myhom_{S\textup{-lin}} ( S^d, N^{\prime} ) $).

\begin{lemma} \label{Lmoritaswitcheroo}
Using the above actions, the function
\begin{align}
\myhom_{S\textup{-lin}} ( S^d, \myhom_{R\textup{-lin}} ( N^{\prime}, R ) ) &\longrightarrow \myhom_{R\textup{-lin}} ( \myhom_{S\textup{-lin}} ( S^d, N^{\prime} ), R ) \label{EQmoritaproduct2} \\
f &\longmapsto F_f \defeq \left\{ \varphi \mapsto \sum_{i=1}^d f ( e_i ) [ \varphi ( e_i ) ] \right\} \nonumber
\end{align}
is a \emph{right}-$ M_d(S) $-linear isomorphism.
\end{lemma}

This lemma says roughly that, in the category (right)$ \catalgebras{M_d(S)} $,
\begin{equation*}
\morita_S ( *^{\vee} ) = \morita_S ( * )^{\vee}
\end{equation*}
where $ \vee = \myhom_{R\textup{-lin}} ( -, R ) $ in both cases.

\begin{proof}
It is easy to verify that the function is a group isomorphism. The $ M_d(S) $-linearity is proved by letting $ b_{i,j} \in S $ be the entries of $ b $ and checking directly:
\begin{align*}
F_{ f \cdot b } ( \varphi )
\defeq \sum_{i=1}^d ( f \cdot b ) ( e_i ) [ \varphi ( e_i ) ] &= \sum_{i=1}^d f \left( \sum_{j=1}^d b_{j,i} e_j \right) [ \varphi ( e_i ) ] \\
\text{(because $ f $ is $ S $-linear)}
&= \sum_{i=1}^d \left( \sum_{j=1}^d ( b_{j,i} \cdot f ( e_j ) ) [ \varphi ( e_i ) ] \right) \\
\text{($ S $ acting on $ \myhom_{R\textup{-lin}} ( N^{\prime}, R ) $)}
&= \sum_{i=1}^d \left( \sum_{j=1}^d f ( e_j ) [ \overline{b}_{j,i} \cdot \varphi ( e_i ) ] \right) \\
\text{(because $ \varphi $ is $ S $-linear)}
&= \sum_{j=1}^d f ( e_j ) \left[ \varphi \left( \sum_{i=1}^d \overline{b}_{j,i} e_i \right) \right] \\
&= \sum_{j=1}^d f ( e_j ) [ \varphi ( \overline{b}^{\tr} ( e_j ) ) ] = F_f ( \varphi \circ \overline{b}^{\tr} ) \defeq ( F_f \cdot b ) ( \varphi )
\end{align*}
\end{proof}

Let $ I : M_d ( S ) \rightarrow M_d ( S ) $ be a multiplication-reversing involution such that $ I ( s ) = \overline{s} $ for all $ s \in S $.

Let $ M $ and $ M^{\prime} $ be \emph{right}-$ M_d ( S ) $-modules (or equivalently, \emph{left}-$ \opp{ M_d ( S ) } $-modules). Let $ \Psi : M \times M^{\prime} \rightarrow R $ be an $ R $-bilinear form such that $ \Psi ( x \cdot b, y ) = \Psi ( x, y \cdot I ( b ) ) $ for all $ b \in M_d ( S ) $ and $ x \in M $, $ y \in M^{\prime} $.

Give $ \myhom_{R\textup{-lin}} ( M^{\prime}, R ) $ a \emph{right}-$ M_d ( S ) $-module structure by the rule
\begin{equation*}
( F \cdot b ) ( m^{\prime} ) \defeq F ( m^{\prime} \cdot I ( b ) )
\end{equation*}
($ F \in \myhom_{R\textup{-lin}} ( M^{\prime}, R ) $, $ b \in M_d(S) $, $ m^{\prime} \in M^{\prime} $). Then the $ R $-linear adjoint map
\begin{equation} \label{EQmoritaproduct1}
\Psi^{\textup{ad}} : M \longrightarrow \myhom_{R\textup{-lin}} ( M^{\prime}, R )
\end{equation}
defined by $ m \mapsto \Psi ( m, - ) $ is automatically \emph{right}-$ M_d ( S ) $-linear:
\begin{equation*}
\Psi^{\textup{ad}} ( m \cdot b ) ( m^{\prime} ) \defeq \Psi ( m \cdot b, m^{\prime} ) = \Psi ( m, m^{\prime} \cdot I ( b ) ) \defeq \Psi^{\textup{ad}} ( m ) ( m^{\prime} \cdot I ( b ) ) = ( \Psi^{\textup{ad}} ( m ) \cdot b ) ( m^{\prime} )
\end{equation*}
($ m \in M $, $ m^{\prime} \in M^{\prime} $, $ b \in M_d(S) $).

Suppose $ M = \morita_S ( N ) $ and $ M^{\prime} = \morita_S ( N^{\prime} ) $ for some $ S $-modules $ N, N^{\prime} $. Then the codomains of (\ref{EQmoritaproduct2}) and (\ref{EQmoritaproduct1}) are identical as \emph{sets}, but possibly not as $ M_d(S) $-modules. If $ I $ is ``standard'', i.e. $ I ( b ) = \overline{b}^{\tr} $ for all $ b \in M_d(S) $, then the two codomains are identical as \emph{right}-$ M_d(S) $-modules. In that case, $ \Psi^{\textup{ad}} $ may be composed with the inverse of isomorphism (\ref{EQmoritaproduct2}) from Lemma \ref{Lmoritaswitcheroo} to produce a right-$ M_d ( S ) $-linear map
\begin{equation} \label{EQmoritaproduct3}
\myhom_{S\textup{-lin}} ( S^d, N ) \longrightarrow \myhom_{S\textup{-lin}} ( S^d, \myhom_{R\textup{-lin}} ( N^{\prime}, R ) )
\end{equation}

Because $ \morita_S = \myhom_{S\textup{-lin}} ( S^d, - ) $ is \emph{fully faithful}, (\ref{EQmoritaproduct3}) is the image of a unique $ S $-linear map
\begin{equation*}
N \rightarrow \myhom_{R\textup{-lin}} ( N^{\prime}, R )
\end{equation*}
By definition of the action of $ S $ on $ \myhom_{R\textup{-lin}} ( N^{\prime}, R ) $, the product
\begin{equation*}
\presuperscript{\circ}{ \Psi } : N \times N^{\prime} \rightarrow R
\end{equation*}
induced by $ N \rightarrow \myhom_{R\textup{-lin}} ( N^{\prime}, R ) $ is $ R $-bilinear and satisfies
\begin{equation*}
\presuperscript{\circ}{ \Psi } ( s n, n^{\prime} ) = \presuperscript{\circ}{ \Psi } ( n, \overline{s} n^{\prime} )
\end{equation*}
for all $ s \in S $, $ n \in N $, and $ n^{\prime} \in N^{\prime} $. In other words, $ \presuperscript{\circ}{ \Psi } $ is ``internally-hermitian'' in the sense of \S\ref{SSalwaystranspose} (page \pageref{extractionofhermitianforms}).

\begin{definition}
Let $ R, S, \Psi, M, N, $ etc. be as above and assume as above that $ I $ is ``standard''. The product that is \emph{Morita Equivalent} to $ \Psi : M \times M^{\prime} \rightarrow R $ is defined to be the $ R $-bilinear product $ \presuperscript{\circ}{ \Psi } : N \times N^{\prime} \rightarrow R $.
\end{definition}

Because $ \morita_S $ is \emph{exact}, the adjoint map $ m \mapsto \Psi ( m, - ) $ is injective or surjective if and only if $ n \mapsto \presuperscript{\circ}{ \Psi } ( n, - ) $ is. Therefore, the product $ \Psi $ is non-degenerate or perfect if and only if $ \presuperscript{\circ}{ \Psi } $ is.

Note that the only non-canonical ingredient in the definition of $ \presuperscript{\circ}{ \Psi } $ was the original product $ \Psi $.

\subsection{Morita Equivalence and duality} \label{SSmoritaandduality}

Let
\begin{equation*}
\presuperscript{\circ}{\psi} : F^d \times F^d \longrightarrow \Q_p
\end{equation*}
be the $ \Q_p $-bilinear form that is Morita Equivalent to $ \psi : M_d(F) \times M_d(F) \rightarrow \Q_p $  (the necessary hypothesis on the involution $ I \defeq \iota $ in \S\ref{SSmoritaandproducts} is supplied by Proposition \ref{Palwaystranspose}).

Let
\begin{equation*}
\presuperscript{\circ}{\Lambda}_{\bullet} = ( \cdots \subset \presuperscript{\circ}{\Lambda}_{-1} \subset \presuperscript{\circ}{\Lambda}_0 \subset \presuperscript{\circ}{\Lambda}_1 \subset \cdots )
\end{equation*}
be a chain of $ \mathcal{O} $-lattices in $ F^d $ such that the image in $ M_d(F) $ of $ \morita_{\mathcal{O}} ( \presuperscript{\circ}{\Lambda}_{\bullet} ) $ (see the beginning of \S\ref{Smorita}) is $ \Lambda_{\bullet} $.

As usual, if $ \presuperscript{\circ}{\Lambda} \subset F^d $ is a lattice then define
\begin{equation*}
\widehat{\presuperscript{\circ}{\Lambda}} \defeq \{ x \in F^d \suchthat \presuperscript{\circ}{\psi} ( \presuperscript{\circ}{\Lambda}, x ) \subset \Z_p \}
\end{equation*}

\begin{prop} \label{Pmoritacommutesduality}
Fix a lattice $ \presuperscript{\circ}{\Lambda} \in \presuperscript{\circ}{\Lambda}_{\bullet} $ and let $ \Lambda $ be the image in $ M_d(F) $ of $ \morita_{\mathcal{O}} ( \presuperscript{\circ}{\Lambda} ) $. \textbf{Assertion}: The image in $ M_d(F) $ of $ \morita_{\mathcal{O}} ( \widehat{\presuperscript{\circ}{\Lambda}} ) $ is $ \widehat{\Lambda} $.

In particular, $ \presuperscript{\circ}{\Lambda}_{\bullet} $ is a $ \presuperscript{\circ}{\psi} $-selfdual lattice chain in $ F^d $ if and only if $ \Lambda_{\bullet} $ is a $ \psi $-selfdual lattice chain in $ M_d(F) $.
\end{prop}

\begin{proof}
This is true simply because $ \morita_{\mathcal{O}} $ is an equivalence-of-categories and dual lattices can be expressed categorically. In more detail, note that $ \widehat{ \Lambda } $ is by definition the image of the $ M_d(\mathcal{O}) $-linear map
\begin{equation} \label{EQdualityfiberproduct1}
\myhom_{\Z_p\textup{-lin}} ( \Lambda, \Z_p ) \stackrel{ \epsilon }{ \longrightarrow } \myhom_{\Q_p\textup{-lin}} ( M_d(F), \Q_p ) \directedisom M_d(F)
\end{equation}
Here $ \epsilon $ extends maps from $ \Lambda $ to $ M_d(F) $ (using that $ \Lambda $ is a \emph{lattice}) and the isomorphism is that induced by the \emph{perfect} form $ \psi $.

Similarly, $ \widehat{ \presuperscript{\circ}{\Lambda} } $ is by definition the image of the analogous $ \mathcal{O} $-linear map
\begin{equation} \label{EQdualityfiberproduct2}
\myhom_{\Z_p\textup{-lin}} ( \presuperscript{\circ}{\Lambda}, \Z_p ) \longrightarrow \myhom_{\Q_p\textup{-lin}} ( F^d, \Q_p ) \directedisom F^d
\end{equation}
(the fact that $ \presuperscript{\circ}{\psi} $ is also perfect was used here).

Applying $ \morita_{\mathcal{O}} $ to (\ref{EQdualityfiberproduct2}) produces a certain map
\begin{equation*}
\linhomom{\mathcal{O}} ( \mathcal{O}^d, \linhomom{\Z_p} ( \presuperscript{\circ}{\Lambda}, \Z_p ) ) \rightarrow M_d(F)
\end{equation*}
(I have again used the natural identification $ \linhomom{\mathcal{O}} ( \mathcal{O}^d, F^d ) \cong M_d(F) $).

By Lemma \ref{Lmoritaswitcheroo}, this map can be identified with
\begin{equation*}
\linhomom{\Z_p} ( \linhomom{\mathcal{O}} ( \mathcal{O}^d, \presuperscript{\circ}{\Lambda} ), \Z_p ) \rightarrow M_d(F).
\end{equation*}
By construction of $ \presuperscript{\circ}{\psi} $ from $ \psi $, and because by definition $ \Lambda $ is the image in $ M_d(F) $ of $ \linhomom{\mathcal{O}} ( \mathcal{O}^d, \presuperscript{\circ}{\Lambda} ) = \morita_{\mathcal{O}} ( \presuperscript{\circ}{\Lambda} ) $, this means (modulo the aforementioned identifications) that applying $ \morita_{\mathcal{O}} $ to (\ref{EQdualityfiberproduct2}) yields (\ref{EQdualityfiberproduct1}). The claim follows since any equivalence-of-categories commutes with the `image' operator.
\end{proof}

\subsection{Morita Equivalence and orthogonality} \label{SSmoritaandorthogonality}

Let $ \presuperscript{\circ}{\psi} $ and $ \presuperscript{\circ}{\Lambda}_{\bullet} $ be as in \S\ref{SSmoritaandduality} and fix a lattice $ \presuperscript{\circ}{\Lambda} \in \presuperscript{\circ}{\Lambda}_{\bullet} $. Let $ \Lambda $ be the image in $ M_d(F) $ of $ \morita_{\mathcal{O}} (  \presuperscript{\circ}{\Lambda} ) $.

As usual, if $ \presuperscript{\circ}{K} \subset \presuperscript{\circ}{\Lambda} $ is any $ \mathcal{R} $-submodule, define
\begin{equation*}
\presuperscript{\circ}{K}^{\perp} \defeq \{ x \in \widehat{\presuperscript{\circ}{\Lambda}} \otimes_{\Z_p} R \suchthat \presuperscript{\circ}{\psi}_R ( \presuperscript{\circ}{K}, x ) = 0 \}.
\end{equation*}

\begin{prop} \label{Pmoritacommutesorthogonality}
Fix a $ \Z_p $-algebra $ R $ and set $ \mathcal{R} := \mathcal{O} \otimes_{\Z_p} R $. Let $ \presuperscript{\circ}{K} \subset \presuperscript{\circ}{\Lambda} $ be an $ \mathcal{R} $-submodule and let $ K $ be the image in $ M_d(\mathcal{R}) $ of $ \morita_{\mathcal{R}} ( \presuperscript{\circ}{K} ) $. \textbf{Assertion}: The image in $ M_d(\mathcal{R}) $ of $ \morita_{\mathcal{R}} ( \presuperscript{\circ}{K}^{\perp} ) $ is $ K^{\perp} $.
\end{prop}

\begin{proof}
By definition, $ \presuperscript{\circ}{K}^{\perp} $ is the kernel of the adjoint map
\begin{equation*}
\widehat{\presuperscript{\circ}{\Lambda}} \otimes_{\Z_p} R \longrightarrow \linhomom{R} ( \presuperscript{\circ}{\Lambda} \otimes_{\Z_p} R, R )
\end{equation*}
defined by $ x \mapsto \presuperscript{\circ}{\psi}_R(-,x)\vert_{\presuperscript{\circ}{\Lambda} \otimes_{\Z_p} R} $. Because $ \morita_{\mathcal{R}} ( * \otimes_{\Z_p} R ) \cong \morita_{\mathcal{O}} ( * ) \otimes_{\Z_p} R $ (this is the trivial case of ``Localization of Hom-Sets''), the above adjoint map may be identified with a certain map
\begin{equation*}
\morita_{\mathcal{O}} ( \widehat{\presuperscript{\circ}{\Lambda}} ) \otimes_{\Z_p} R \longrightarrow \linhomom{R} ( \morita_{\mathcal{O}} ( \presuperscript{\circ}{\Lambda} ) \otimes_{\Z_p} R, R ).
\end{equation*}
By Proposition \ref{Pmoritacommutesduality}, the domain of this map may be identified with $ \widehat{\Lambda} \otimes_{\Z_p} R $, and it follows from the construction of $ \presuperscript{\circ}{\psi} $ that the resulting map $ \widehat{\Lambda} \otimes_{\Z_p} R \rightarrow \linhomom{R} ( \Lambda \otimes_{\Z_p} R, R ) $ is exactly the adjoint map of $ \psi_R $. Because $ \morita_{\mathcal{R}} $ is an equivalence-of-categories, it commutes with the `kernel' operator and the claim follows.
\end{proof}

\subsection{Compressed description of the local model} \label{SSmoritacompressedlocalmodel}

Let $ \presuperscript{\circ}{\psi} : F^d \times F^d \rightarrow \Q_p $ and $ \presuperscript{\circ}{\Lambda}_{\bullet} $ be as in \S\ref{SSmoritaandduality}.

\begin{prop} \label{Ppostmoritalocalmodel}
The local model $ \localmodel $ is isomorphic to the functor that assigns to any commutative $ \Z_p $-algebra $ R $ the set of all (isomorphism classes of) commutative diagrams
\begin{equation*}
\begin{CD}
\cdots @>\inc \otimes \identity>> \presuperscript{\circ}{\Lambda}_i \otimes_{\Z_p} R @>\inc \otimes \identity>> \presuperscript{\circ}{\Lambda}_{i+1} \otimes_{\Z_p} R @>\inc \otimes \identity>> \cdots \\
@. @AAA @AAA @. \\
\cdots @>>> \presuperscript{\circ}{K}_i @>>> \presuperscript{\circ}{K}_{i+1} @>>> \cdots
\end{CD}
\end{equation*}
of $ ( \mathcal{O} \otimes_{\Z_p} R ) $-modules satisfying the following properties:
\begin{enumerate}
\setlength{\itemsep}{5pt}
\item For each $ i $, $ \presuperscript{\circ}{K}_i \rightarrow \presuperscript{\circ}{\Lambda}_i \otimes_{\Z_p} R $ is \emph{injective},

\item for each $ i $, if $ p \cdot \presuperscript{\circ}{\Lambda}_i = \presuperscript{\circ}{\Lambda}_j $, then the isomorphism $ - \cdot p : \presuperscript{\circ}{\Lambda}_i \directedisom \presuperscript{\circ}{\Lambda}_j $ restricts to an isomorphism $ \presuperscript{\circ}{K}_i \directedisom \presuperscript{\circ}{K}_j $,

\item for each $ i $, $ \presuperscript{\circ}{K}_i \hookrightarrow \presuperscript{\circ}{\Lambda}_i \otimes_{\Z_p} R $ splits $ R $-linearly,

(in particular, each $ \presuperscript{\circ}{K}_i $ is $ R $-projective)

\item for each $ i $, the projective rank function $ \spec(R) \rightarrow \N $ of $ \presuperscript{\circ}{K}_i $ is the constant function $ \mathfrak{p} \mapsto d $, and

\item $ \presuperscript{\circ}{K}_{\widehat{\presuperscript{\circ}{\Lambda}}} = \presuperscript{\circ}{K}_{\presuperscript{\circ}{\Lambda}}^{\perp} $ for each $ \presuperscript{\circ}{\Lambda} = \presuperscript{\circ}{\Lambda}_i $.

%(here $ K_{\Lambda}^{\perp} \defeq \{ \lambda \in \widehat{\Lambda} \otimes_{\Z_p} R \suchthat \presuperscript{\circ}{\psi}_R ( \presuperscript{\circ}{K}_{\Lambda}, \lambda ) = 0 \} $)
\end{enumerate}
\end{prop}

As usual, the equality in (5) is a slight abuse.

\begin{proof}
Let $ \mathbf{M} $ be the functor described in this proposition. First, I define a morphism $ \mathbf{M} \rightarrow \localmodel $. Fix a $ \Z_p $-algebra $ R $, set $ \mathcal{R} := \mathcal{O} \otimes_{\Z_p} R $, and consider a point $ \presuperscript{\circ}{K}_{\bullet} \rightarrow \presuperscript{\circ}{\Lambda}_{\bullet} \otimes_{\Z_p} R $ of $ \mathbf{M} ( R ) $. By the trivial case of ``Localization of Hom-Sets'', $ \morita_{\mathcal{R}} ( \Lambda \otimes_{\Z_p} R ) \cong \morita_{\mathcal{O}} ( \Lambda ) \otimes_{\Z_p} R $ for any $ \mathcal{O} $-lattice $ \Lambda \subset F^d $, so the diagram $ K_{\bullet} \rightarrow \Lambda_{\bullet} \otimes_{\Z_p} R $ that is isomorphic in this way to $ \morita_{\mathcal{R}} ( \presuperscript{\circ}{K}_{\bullet} \rightarrow \presuperscript{\circ}{\Lambda}_{\bullet} \otimes_{\Z_p} R ) $ is a diagram of the same ``type'' as the points of $ \localmodel ( R ) $ but a priori possibly lacking some of the properties defining $ \localmodel ( R ) $. Since $ \morita_{\mathcal{R}} $ is an equivalence-of-categories and is therefore exact, the diagram $ K_{\bullet} \rightarrow \Lambda_{\bullet} \otimes_{\Z_p} R $ has property \subOLMdef{\ref{LMquotient}}. It is immediate from the definition of $ \morita_{\mathcal{R}} $ that the diagram $ K_{\bullet} \rightarrow \Lambda_{\bullet} \otimes_{\Z_p} R $ has properties \subOLMdef{\ref{LMperiodic}} and \subOLMdef{\ref{LMlocallyfree}}. Since $ \rank ( \morita ( * ) ) = d \cdot \rank ( * ) $, the diagram $ K_{\bullet} \rightarrow \Lambda_{\bullet} \otimes_{\Z_p} R $ also has property \subOLMdef{\ref{LMcharpoly}}. Finally, Proposition \ref{Pmoritacommutesorthogonality} implies that the diagram $ K_{\bullet} \rightarrow \Lambda_{\bullet} \otimes_{\Z_p} R $ has property \subOLMdef{\ref{LMduality}}. Altogether, $ K_{\bullet} \rightarrow \Lambda_{\bullet} \otimes_{\Z_p} R $ is a point of $ \localmodel ( R ) $, which yields the morphism $ \mathbf{M} \rightarrow \localmodel $. Because each $ \morita_{\mathcal{R}} $ is an equivalence-of-categories, one may use inverse functors $ \morita^{-1}_{\mathcal{R}} $ to define an analogous morphism $ \localmodel \rightarrow \mathbf{M} $ that is inverse to $ \mathbf{M} \rightarrow \localmodel $.
\end{proof}

\begin{terminologynotation}
From now on, identify $ \localmodel $ with the functor described in Proposition \ref{Ppostmoritalocalmodel}.
\end{terminologynotation}

The similitude group $ G_{\Q_p} $ can also benefit, in an analogous way, from Morita Equivalence. Let $ \presuperscript{\circ}{\psi} : F^d \times F^d \rightarrow \Q_p $ be as before.

Recall from \S\ref{SPELdatum} (page \pageref{SPELdatum}) that $ M_d(F) $ acts by \emph{left}-multiplication on $ V = M_d(F) $ and the involution induced by $ \psi $ on this left-acting $ M_d(F) $ is none other than $ * $, i.e. $ \psi ( b x, y ) = \psi ( x, b^{*} y ) $ for all $ b, x, y \in M_d(F) $. If $ b \in M_d(F) $ is considered as an $ F $-linear map $ F^d \rightarrow F^d $, then $ \morita_F (b) $ is a right-$ M_d(F) $-linear map $ M_d(F) \rightarrow M_d(F) $, and this map is none other than right-multiplication by $ b $.

A careful inspection of the construction of $ \presuperscript{\circ}{\psi} $ shows that $ \psi $ can be expressed using $ \presuperscript{\circ}{\psi} $ in a very simple way: if $ x, y \in M_d(F) $, then
\begin{equation} \label{Emoritaproductidentity}
\psi ( x, y ) = \sum_{i=1}^d \presuperscript{\circ}{\psi} ( x_i, y_i )
\end{equation}
where $ x_i $ and $ y_i $ are the $ i $th columns of the matrices $ x, y $.

%This is verified as follows. Writing a matrix $ x \in M_d(F) = \morita_F ( F^d ) $ as a $ d $-tuple $ ( x_1, \ldots, x_d ) $ of column vectors corresponds to writing $ \morita_F ( F^d ) = \myhom_{F\textup{-lin}} ( F^d, F^d ) $ as the $ d $-fold product of $ \myhom_{F\textup{-lin}} ( F, F^d ) \cong F^d $. Denote by
%\begin{equation*}
%\presuperscript{\circ}{\psi}^{\textup{ad}} : F^d \longrightarrow \myhom_{\Q_p\textup{-lin}} ( F^d, \Q_p )
%\end{equation*}
%the morphism induced by $ \presuperscript{\circ}{\psi} $. Writing all matrices as $ d $-tuples of column vectors as above, and considering the isomorphism from Lemma \ref{Lmoritaswitcheroo} (page \pageref{Lmoritaswitcheroo}), the map
%\begin{equation*}
%\psi^{\textup{ad}} : M_d(F) \longrightarrow \myhom_{\Q_p\textup{-lin}} ( M_d(F), \Q_p )
%\end{equation*}
%induced by $ \psi $ is written as
%\begin{equation*}
%\psi^{\textup{ad}} ( x_1, \ldots, x_d ) = \left( ( y_1, \ldots, y_d ) \longmapsto \sum_{i=1}^d \presuperscript{\circ}{\psi}^{\textup{ad}} ( x_i ) [ y_i ] \right)
%\end{equation*}
%Since $ \presuperscript{\circ}{\psi}^{\textup{ad}} ( x_i ) [ y_i ] $ is just another way to write $ \presuperscript{\circ}{\psi} ( x_i, y_i ) $, the claim is proved.

Since $ \psi $ is non-degenerate and perfect, $ \presuperscript{\circ}{\psi} $ is also and therefore induces an involution
\begin{equation*}
\presuperscript{\circ}{*} : \myend_{ F\textup{-lin} }( F^d ) \longrightarrow \myend_{ F\textup{-lin} }( F^d )
\end{equation*}

\begin{prop} \label{Psimplifydescofgroup}
$ G_{\Q_p} $ is equal to the functor that assigns to any commutative $ \Q_p $-algebra $ R $ the group of all $ g \in M_d ( F ) \otimes_{\Q_p} R $ satisfying $ g^{\presuperscript{\circ}{*}} \cdot g \in R^{\times} $.
\end{prop}

\begin{proof}
This is true simply because $ * $ and $ \presuperscript{\circ}{*} $ are the same: if $ x, y \in F^d $ and if $ X, Y \in M_d(F) $ are the matrices whose $ i $th columns are $ x, y $ (respectively) and $ 0 $ in all other entries, then identity (\ref{Emoritaproductidentity}) collapses to $ \psi ( X, Y ) = \presuperscript{\circ}{\psi} ( x, y ) $.
\end{proof}

\begin{terminologynotation}
Now that the details of the Morita Equivalence are recorded, it will not be necessary to refer to the objects $ \psi $, $ * $, etc. and I simplify notation and use those symbols to refer instead to $ \presuperscript{\circ}{\psi} $, $ \presuperscript{\circ}{*} $, etc.
\end{terminologynotation}

\subsection{Construction of lattices and truncation of the local model} \label{SSsimpledesc}

\begin{center}
\emph{In this last subsection of \S\ref{Smorita}, I construct the lattices promised in \S\ref{SSoriginallocalmodeldef} and \S\ref{SSmoritaandduality} and make a few minor final changes to the description of the functor $ \localmodel $.}
\end{center}

First, notice that all concepts of duality and orthogonality can be rewritten to use $ \phi : F^d \times F^d \rightarrow F $, the hermitian form corresponding to $ \psi $ according to ``Extraction of Hermitian Forms'', instead of $ \psi $. The explicit correspondence between alternating forms like $ \psi $ and hermitian forms $ \phi $ on page \pageref{extractionofhermitianforms} implies that
\begin{align*}
\widehat{ \Lambda } &= \{ w \in F^d \suchthat \phi ( \Lambda, w ) \subset \mathcal{O} \} \\
K^{\perp} &= \{ \lambda \in \widehat{ \Lambda } \otimes_{\Z_p} R \suchthat \phi_R ( K, \lambda ) = 0 \}
\end{align*}
for any $ \mathcal{O} $-lattice $ \Lambda \subset F^d $, any $ \Z_p $-algebra $ R $, and any $ ( \mathcal{O} \otimes_{\Z_p} R ) $-submodule $ K \subset \Lambda \otimes_{\Z_p} R $.

Second, I claim that I may apply an element of $ \GL_d(F) $ to $ F^d $ so that $ \phi $ is transformed to the hermitian form $ \phi^{\vee} $ with Gram matrix the anti-identity matrix $ \antiid $. To see this, note that $ \phi $ is isometric to the hermitian form with Gram matrix the identity matrix $ \identity $ by Lemma \ref{Lgapfiller} and the Assumption at the end of \S\ref{SSalwaystranspose} (page \pageref{Aquasisplit}). By the well-known classification of hermitian forms over local fields (see 1.6(ii) on page 351 in \cite{Sc}), the isometry class of a hermitian form is completely determined by its dimension and the determinant of its Gram matrix in the norm-class-group $ \Q_p^{\times} / N_{F/\Q_p} ( F^{\times} ) $. Since $ \det ( \antiid ) = \pm 1 $ and since the integer norm map $ N_{F/\Q_p} : \mathcal{O}^{\times} \rightarrow \Z_p^{\times} $ is \emph{surjective} for unramified extensions $ F $, the claim is proven.

Third, define \begin{equation*}
\Lambda_i \defeq p^{-1} \mathcal{O}^i \oplus \mathcal{O}^{d-i} \textaftermath{($ 0 \leq i < d $)}
\end{equation*}
and extend periodically by $ \Lambda_i \defeq p^{-q} \cdot \Lambda_r $ whenever the Division Algorithm produces $ i = q d + r $. This is a periodic $ \phi^{\vee} $-selfdual $ \mathcal{O} $-lattice chain in $ F^d $.

\begin{remark}
Note that, by the first two paragraphs, applying the inverse of the transformation that was used to transform $ \phi $ to $ \phi^{\vee} $ yields a periodic $ \psi $-selfdual $ \mathcal{O} $-lattice chain in $ F^d $ and that, by Proposition \ref{Pmoritacommutesduality}, applying $ \morita_{\mathcal{O}} $ to this $ \psi $-selfdual lattice chain and taking the image in $ M_d(F) $ yields the lattice chain promised in \S\ref{SSoriginallocalmodeldef} (although part of the point of \S\ref{Smorita} is that it is not necessary to directly consider either of these lattices chains anymore).
\end{remark}

Combining the above paragraphs yields:

\begin{prop} \label{PSLM}
The local model $ \localmodel $ is isomorphic to the functor which assigns to any commutative $ \Z_p $-algebra $ R $ the set of all (isomorphism classes of) finite commutative diagrams
\begin{equation*}
\begin{CD}
\Lambda_0 \otimes_{\Z_p} R @>\inc \otimes \identity>> \Lambda_1 \otimes_{\Z_p} R @>\inc \otimes \identity>> \cdots @>\inc \otimes \identity>> \Lambda_{d/2} \otimes_{\Z_p} R \\
@AAA @AAA @. @AAA \\
K_0 @>>> K_1 @>>> \cdots @>>> K_{d/2}
\end{CD}
\end{equation*}
of $ ( \mathcal{O} \otimes_{\Z_p} R ) $-modules satisfying the following properties:
\begin{enumerate}
\setlength{\itemsep}{5pt}
\item \label{SLMsubobject} \SLMdef{\ref{SLMsubobject}} \hfill \\
For each $ i $, $ K_i \rightarrow \Lambda_i \otimes_{\Z_p} R $ is \emph{injective},

\item \label{SLMsummand} \SLMdef{\ref{SLMsummand}} \hfill \\
for each $ i $, the inclusion $ K_i \hookrightarrow \Lambda_i \otimes_{\Z_p} R $ splits $ R $-linearly,

(in particular, each $ K_i $ is $ R $-projective)

\item \label{SLMlocalrank} \SLMdef{\ref{SLMlocalrank}} \hfill \\
for each $ i $, the projective rank function $ \spec(R) \rightarrow \N $ of $ K_i $ is the constant function $ \mathfrak{p} \mapsto d $, and

\item \label{SLMisotropic} \SLMdef{\ref{SLMisotropic}} \hfill \\
$ K_0^{\perp} = K_0 $ and $ K_{d/2}^{\perp} = p K_{d/2} $ with respect to the restrictions $ \Lambda_0 \times \Lambda_0 \rightarrow \mathcal{O} $ and $ \Lambda_{d/2} \times p \Lambda_{d/2} \rightarrow \mathcal{O} $ of $ \phi^{\vee} $.
\end{enumerate}
\end{prop}

\begin{proof}
By Proposition \ref{Ppostmoritalocalmodel}, if $ \mathbf{M} $ is the functor described in this proposition, the above discussion shows that deleting the part of the diagram indexed by $ i < 0 $ and $ i > d/2 $ defines a natural transformation $ \localmodel \rightarrow \mathbf{M} $ (the pair of orthogonality requirements in this proposition are simply the only requirements from Proposition \ref{Ppostmoritalocalmodel} that survive the truncation). The inverse $ \mathbf{M} \rightarrow \localmodel $ is defined by reconstructing infinite diagrams from finite diagrams as follows: for $ 0 \leq i \leq d/2 $, note that $ \Lambda_{-i} = \widehat{ \Lambda }_i $, define $ K_{-i} \rightarrow \Lambda_{-i} $ to be the inclusion $ K_i^{\perp} \subset \widehat{ \Lambda }_i $, and extend the diagram to $ \vert i \vert > d/2 $ by the periodicity rules $ \Lambda_{i \pm d} = p^{\mp 1} \Lambda_i $ and $ K_{i \pm d} \defeq p^{\mp 1} K_i $.
\end{proof}

\begin{terminologynotation}
From now on, identify $ \localmodel $ with the functor described in Proposition \ref{PSLM}. From now on, the symbol ``$ \phi $'' refers to the hermitian form $ F^d \times F^d \rightarrow F $ whose Gram matrix is the anti-identity matrix $ \antiid $ (instead of ``$ \phi^{ \vee } $''). I also emphasize that from now on ``$ \Lambda_{\bullet} $'' refers to the \emph{specific} chain of $ \mathcal{O} $-lattices in $ F^d $ defined in this subsection.
\end{terminologynotation}

\section{The enlarged models} \label{Sdefinitionofthelargermodels}

\subsection{Alternate description of the local model using a single base
lattice} \label{SSaltdesc}

This subsection rephrases the definition of $ \localmodel $ so that its points consist of submodules of the single lattice $ \Lambda_0 $ and it will then be easier to see how to enlarge the models satisfactorily to a degeneration of the affine grassmannian to the full affine flag variety.

For $ i = 1, \ldots, d $, let $ \alpha_i : F^d \rightarrow F^d $ be the $ F $-linear map
\begin{equation*}
( x_1, \ldots, x_d ) \mapsto ( x_1, \ldots, x_{i-1}, p x_i,
x_{i+1}, \ldots, x_d )
\end{equation*}
Then $ \alpha_i $ induces an $ \mathcal{O} $-module isomorphism $ \Lambda_i \directedisom \Lambda_{i-1} $. Define
\begin{equation*}
\alpha_{[i]} \defeq \alpha_i \circ \cdots \circ \alpha_1
\end{equation*}
Then $ \alpha_{[i]} $ induces an $ \mathcal{O} $-module isomorphism $ \Lambda_i \directedisom \Lambda_0 $. Define
\begin{equation*}
\alpha^{[i]} \defeq \alpha_d \circ \cdots \circ \alpha_{d-i+1}
\end{equation*}

Pick a commutative $ \Z_p $-algebra $ R $, and consider a point
\begin{equation*}
\begin{array}{ccccccc}
\Lambda_0 \otimes_{\Z_p} R & \longrightarrow & \Lambda_1 \otimes_{\Z_p} R & \longrightarrow & \cdots & \longrightarrow & \Lambda_{d/2} \otimes_{\Z_p} R \\
\cup &  & \cup &  & \cdots &  & \cup \\
K_0 & \longrightarrow & K_1 & \longrightarrow & \cdots &
\longrightarrow & K_{d/2}
\end{array}
\end{equation*}
of $ \localmodel(R) $.

Specifying $ K_i \subset \Lambda_i \otimes_{\Z_p} R $ is the same as specifying $ \alpha_{[i]} ( K_i ) \subset \Lambda_0 \otimes_{\Z_p} R $, so define
\begin{equation*}
L_i \defeq \alpha_{[i]} ( K_i )
\end{equation*}
(note that $ L_0 = K_0 $). The condition that $ \Lambda_i \otimes_{\Z_p} R \rightarrow \Lambda_{i+1} \otimes_{\Z_p} R $ restrict to $ K_i \rightarrow K_{i+1} $ is equivalent to the condition that $ \alpha_{i+1} ( L_i ) \subset L_{i+1} $ (note that $ \alpha_{i+1} \circ \alpha_{[i]} = \alpha_{[i+1]} $). So, a point of $ \localmodel(R) $ is equivalent to a tuple $ ( L_0, \ldots, L_{d/2} ) $ satisfying $ \alpha_{i+1} ( L_i ) \subset L_{i+1} $ for each $ 0 \leq i < d/2 $ and the implicit equivalents of \SLMdef{}.

For \SLMdef{\ref{SLMsummand}}, note that if $ s $ is an $ R $-linear splitting of $ K_i \hookrightarrow \Lambda_i \otimes_{\Z_p} R $ then $ \alpha_{[i]}^{-1} \circ s \circ \alpha_{[i]} $ is an $ R $-linear splitting of $ L_i \subset \Lambda_0 \otimes_{\Z_p} R $, so that condition is identical: require that each inclusion $ L_i \subset \Lambda_0 \otimes_{\Z_p} R $ split $ R $-linearly. It is obvious that the projective rank condition \SLMdef{\ref{SLMlocalrank}} is the same for the $ K_i $ as for the $ L_i $.

For \SLMdef{\ref{SLMisotropic}}, note that the condition $ K_0^{\perp} = K_0 $ is equivalent to the condition $ L_0^{\perp} = L_0 $ tautologically. The fact that
\begin{equation*}
\phi ( \alpha_{[i]} ( x ), y ) = \phi ( x, \alpha^{[i]} ( y ) )
\end{equation*}
and
\begin{equation*}
\alpha_{[d/2]} \circ \alpha^{[d/2]} = p
\end{equation*}
suggests the following:
\begin{lemma}
$ K_{d/2}^{\perp} = p K_{d/2} $ ($ \subset p \Lambda_{d/2} $) with respect to $ \phi : \Lambda_{d/2} \times p \Lambda_{d/2} \rightarrow \mathcal{O} $ \textbf{if and only if} $ L_{d/2}^{\perp} = L_{d/2} $ with respect to $ \phi : \Lambda_0 \times \Lambda_0 \rightarrow \mathcal{O} $.
\end{lemma}

\begin{proof}
\framebox{$ \Rightarrow $} \framebox{$ \subset $} If $ y \in \Lambda_0 $ then $ \alpha^{[d/2]} ( y ) \in p \Lambda_{d/2} $, and since
\begin{equation*}
\phi_R ( L_{d/2}, y ) = \phi_R ( \alpha_{[d/2]} ( K_{d/2} ), y ) = \phi_R ( K_{d/2}, \alpha^{[d/2]} ( y ) ),
\end{equation*}
the fact that $ \phi ( L_{d/2}, y ) = 0 $ implies that $ \alpha^{[d/2]} ( y ) \in K_{d/2}^{\perp} = p K_{d/2} $. This means that
\begin{equation*}
p y = \alpha_{[d/2]} ( \alpha^{[d/2]} ( y ) ) \in \alpha_{[d/2]} ( p K_{d/2} ) = p L_{d/2}
\end{equation*}
and so $ y \in L_{d/2} $. \framebox{$ \supset $} This is easy to verify: if $ x \in L_{d/2} $ then $ x = \alpha_{[d/2]} ( x^{\prime} ) $ for some $ x^{\prime} \in K_{d/2} $ so
\begin{equation*}
\phi_R ( L_{d/2}, x ) = \phi_R ( \alpha_{[d/2]} ( K_{d/2} ), \alpha_{[d/2]}( x^{\prime} ) ) = p \phi_R ( K_{d/2}, x^{\prime} ) = 0
\end{equation*}
and so $ x \in L_{d/2}^{\perp} $. \framebox{$ \Leftarrow $} \framebox{$ \subset $} If $ y \in p \Lambda_{d/2} $ is such that $ \phi_R ( K_{d/2}, y ) = 0 $ then write $ y = p y^{\prime} $ for some $ y^{\prime} \in \Lambda_{d/2} $ so that
\begin{equation*}
\phi_R ( L_{d/2}, \alpha_{[d/2]} ( y^{\prime} ) ) = \phi_R ( K_{d/2}, \alpha^{[d/2]} ( \alpha_{[d/2]} ( y^{\prime} ) ) ) = \phi_R ( K_{d/2}, y ) = 0
\end{equation*}
using the same ideas as before. Since $ \alpha_{[d/2]} ( y^{\prime} ) \in \Lambda_0 $, the hypothesis gives $ \alpha_{[d/2]} ( y^{\prime} ) \in L_{d/2} $ and so $ y^{\prime} \in K_{d/2} $. \framebox{$ \supset $} This is similar.
\end{proof}

\begin{altdesclocalmodel}
According to the previous discussion, the functor $ \localmodel $ can also be described as assigning to each commutative $ \Z_p $ algebra $ R $, the set of all tuples $ ( L_0, L_1, \ldots, L_{d/2} ) $ of $ ( \mathcal{O} \otimes_{\Z_p} R ) $-submodules of $ \Lambda_0 \otimes_{\Z_p} R $ satisfying:
\begin{itemize}
\setlength{\itemsep}{5pt}
\item \label{ALMchain} $ \alpha_{i+1} ( L_i ) \subset L_{i+1} $ for all $ i $

\item \label{ALMlocalsummand} each inclusion $ L_i \subset \Lambda_0 \otimes_{\Z_p} R $ splits $ R $-linearly

\item \label{ALMlocalrank} for each $ i $, the projective rank function $ \spec(R) \rightarrow \N $ of $ L_i $ is the constant function $ \mathfrak{p} \mapsto d $

\item \label{ALMisotropic} $ L_0^{\perp} = L_0 $ and $ L_{d/2}^{\perp} = L_{d/2} $ with respect to the restriction $ \phi : \Lambda_0 \times \Lambda_0 \rightarrow \mathcal{O} $.
\end{itemize}
\end{altdesclocalmodel}

This description is the key to constructing larger schemes analogous to $ \localmodel $.

\subsection{Definition of the larger models} \label{SSlargermodels}

In this subsection, I will define a family of functors
\begin{equation*}
\biggermodel{m}{n} : \catalgebras{\Z_p} \rightarrow \catsets
\end{equation*}
over all $ m, n \in \N $ such that $ \biggermodel{0}{1} = \localmodel $, and such that the generic (resp. special) fibers form an increasing and exhaustive filtration of the affine Grassmannian (resp. full affine flag variety). Fix $ m, n \in \N $.

The anti-identity matrix $ \antiid $ induces a non-degenerate hermitian $ \Z_p[t] $-bilinear form
\begin{equation*}
\phi : \frac{t^{-m} \mathcal{O}[t]^d}{t^n \mathcal{O}[t]^d} \times \frac{t^{-m} \mathcal{O}[t]^d}{t^n \mathcal{O}[t]^d} \longrightarrow \frac{ t^{-2m} \mathcal{O}[t] }{t^{n-m} \mathcal{O}[t] }
\end{equation*}
by the rule $ ( v, w ) \mapsto v^{\tr} \cdot \antiid \cdot \overline{w} $, where $ w \mapsto \overline{w} $ is induced by the non-trivial element of $ \gal ( F / \Q_p ) $. For an $ \mathcal{R} [t] $-submodule $ L \subset t^{-m} \mathcal{R}[t]^d / t^n \mathcal{R}[t]^d $ define $ L^{\perp} $ in the usual way:
\begin{equation*}
L^{\perp} \defeq \{ v \in t^{-m} \mathcal{R}[t]^d / t^n \mathcal{R}[t]^d \suchthat \phi_R ( L, v ) = 0 \}
\end{equation*}

\begin{remark}
When $ (m,n) = (0,1) $, these $ \phi $ and $ \perp $ are the same as those from the previous subsection via the identification $ \mathcal{O}[t]^d / t \mathcal{O}[t]^d = \Lambda_0 $, so the abuse of notation is acceptable.
\end{remark}

The rule
\begin{equation*}
( x_1, \ldots, x_d ) \mapsto ( x_1, \ldots, x_{i-1}, (t+p)x_i, x_{i+1}, \ldots, x_d )
\end{equation*}
induces an $ \mathcal{O}[t] $-linear map
\begin{equation*}
\alpha_i : \frac{ t^{-m} \mathcal{O} [t]^d }{ t^n \mathcal{O} [t]^d } \longrightarrow \frac{ t^{-m} \mathcal{O} [t]^d }{ t^n \mathcal{O} [t]^d }
\end{equation*}

\begin{remark}
When $ (m,n) = (0,1) $, this map is the restriction to $ \Lambda_0 $ of the ``$ \alpha_i $'' from the previous subsection via the identification $ \mathcal{O}[t]^d / t \mathcal{O}[t]^d = \Lambda_0 $, so the abuse of notation is acceptable.
\end{remark}

As before, define
\begin{equation*}
\alpha_{[i]} \defeq \alpha_i \circ \cdots \circ \alpha_1
\end{equation*}
The following definition is based on the alternate description of $ \localmodel $ and is strongly analogous to the symplectic case in \cite{HN}:
\begin{BLMdef1} \label{BLMdef1}
Define the functor
\begin{equation*}
\biggermodel{m}{n} : \catalgebras{\Z_p} \rightarrow \catsets
\end{equation*}
by assigning to each commutative $ \Z_p $-algebra $ R $ the set ($ \mathcal{R} := \mathcal{O} \otimes_{\Z_p} R $) of tuples $ ( L_0, L_1, \ldots, L_{d/2} ) $ of $ \mathcal{R} [t] $-submodules of $ t^{-m} \mathcal{R} [t]^d / t^n \mathcal{R} [t]^d $ satisfying:
\begin{itemize}
\setlength{\itemsep}{5pt}
\item $ \alpha_{i+1} ( L_i ) \subset L_{i+1} $ for all $ 0 \leq i < d/2 $

\item each inclusion $ L_i \subset t^{-m} \mathcal{R} [t]^d / t^n \mathcal{R} [t]^d $ splits $ R $-linearly

\item for each $ i $, the projective rank function $ \spec(R) \rightarrow \N $ of $ L_i $ is the constant function $ \mathfrak{p} \mapsto d(m+n) $

(note that $ \rank_R ( t^{-m} \mathcal{R} [t]^d / t^n \mathcal{R} [t]^d ) = 2d(m + n) $)

\item $ L_0^{\perp} = L_0 $ and $ L_{d/2}^{\perp} = L_{d/2} $ with respect to $ \phi_R $.
\end{itemize}
\end{BLMdef1}
It is clear that $ \biggermodel{0}{1} = \localmodel $.

%\begin{remark}I will later embed $ \biggermodel{m}{n}_{\F_p} $ into a full affine flag variety $ \affineflagvariety_{\F_p} $ and the freedom of two parameters $ m, n $ is necessary in order to exhaust $ \affineflagvariety_{\F_p} $.\end{remark}

I now reformulate the definition of $ \biggermodel{m}{n} $ so that it is more obviously a degeneration from the affine Grassmannian over $ \Q_p $ to the full affine flag variety over $ \F_p $. Now that the variable $ t $ and the parameters $ m, n $ have been introduced, I am in a sense returning to the point of view used for $ \SLMdef{} $.

Define
\begin{equation*}
\mathcal{V} = \mathcal{O} [ t, t^{-1}, (t+p)^{-1} ]^d
\end{equation*}
and submodules
\begin{align*}
\mathcal{V}_0 &= \mathcal{O}[t]^d \\
\mathcal{V}_1 & = (t+p)^{-1} \mathcal{O}[t] \oplus \mathcal{O}[t]^{d-1} \\
&{\vdots} \\
\mathcal{V}_{d-1} &= (t+p)^{-1} \mathcal{O}[t]^{d-1} \oplus \mathcal{O}[t] \\
\mathcal{V}_d &= (t+p)^{-1} \mathcal{O}[t]^d
\end{align*}

Note that the $ \alpha_{[i]} $, as $ \mathcal{O}[t] $-linear maps of $ \mathcal{V} $, induce isomorphisms
\begin{equation} \label{alphaisom}
\alpha_{[i]} : \frac{ t^{-m} \mathcal{V}_i }{ t^n \mathcal{V}_i } \directedisom \frac{ t^{-m} \mathcal{O}[t]^d }{ t^n \mathcal{O}[t]^d }
\end{equation}
(these are the generalizations of the isomorphisms ``$ \alpha_{[i]} : \Lambda_i \directedisom \Lambda_0 $'' from \S\ref{SSaltdesc})

Let $ R $ be a commutative $ \Z_p $-algebra and consider $ ( L_0, L_1, \ldots, L_{d/2} ) \in \biggermodel{m}{n} (R) $. Let $ \overline{\mathcal{L}}_i $ be the submodule of $ t^{-m} \mathcal{V}_i / t^n \mathcal{V}_i $ such that $ \alpha_{[i]} ( \overline{\mathcal{L}}_i ) = L_i $. Let $ \mathcal{L}_i $ be the submodule of $ \mathcal{V}(R) $ satisfying
\begin{equation} \label{EmathcalLcontainments}
t^n \mathcal{V}_i(R) \subset \mathcal{L}_i \subset t^{-m} \mathcal{V}_i (R)
\end{equation}
which corresponds to $ \overline{\mathcal{L}}_i $. The requirement that $ \alpha_{i+1} ( L_i ) \subset L_{i+1} $ for all $ 0 \leq i < d/2 $ is equivalent to the requirement that $ \mathcal{L}_0 \subset \mathcal{L}_1 \subset \cdots \subset \mathcal{L}_{d/2} $.

The requirement that each inclusion $ L_i \subset t^{-m} \mathcal{R}[t]^d / t^n \mathcal{R}[t]^d $ split $ R $-linearly is equivalent to the requirement that each inclusion $ \overline{\mathcal{L}}_i \subset t^{-m} \mathcal{V}_i(R) / t^n \mathcal{V}_i(R) $ split $ R $-linearly.

The anti-identity matrix $ \antiid $ defines a non-degenerate hermitian $ \Z_p [ t, t^{-1}, (t+p)^{-1} ] $-bilinear product
\begin{equation*}
\phi : \mathcal{V} \times \mathcal{V} \longrightarrow \mathcal{O} [ t, t^{-1}, (t+p)^{-1} ]
\end{equation*}
by the rule $ ( v, w ) \mapsto v^{\tr} \cdot \antiid \cdot \overline{w} $, where $ w \mapsto \overline{w} $ is induced by the non-trivial element of $ \gal ( F / \Q_p ) $.

\begin{remark}
Since this $ \phi $ induces (restrict and descend) the previously defined $ \phi $, this abuse of notation is acceptable.
\end{remark}

For $ 0 \leq i \leq d/2 $, define
\begin{equation*}
\phi^{[i]} : \mathcal{V} \times \mathcal{V} \longrightarrow \mathcal{O} [ t, t^{-1}, (t+p)^{-1} ]
\end{equation*}
by $ \phi^{[i]} ( x, y ) \defeq \phi ( \alpha_{[i]} (x), \alpha_{[i]} (y) ) $. Note that $ \phi^{[d/2]} = (t+p) \phi $.

For $ 0 \leq i \leq d/2 $, define
\begin{equation*} \label{Dduality}
\widehat{\mathcal{L}}_i \defeq \{ x \in \mathcal{V}(R) \suchthat \phi^{[i]}_R(\mathcal{L}_i, x) \subset t^{n-m} \mathcal{R}[t] \}
\end{equation*}
(I am abusing notation: the concept of duality here depends on $ i $). This concept of duality is specifically designed to match up with the concept of ``$ \perp $'' above:
\begin{lemma} \label{isotropiciffdual}
\begin{enumerate}
\setlength{\itemsep}{5pt}
\item $ L_0^{\perp} = L_0 $ if and only if $ \widehat{\mathcal{L}}_0 = \mathcal{L}_0 $, and

\item $ L_{d/2}^{\perp} = L_{d/2} $ if and only if $ \widehat{\mathcal{L}}_{d/2} = \mathcal{L}_{d/2} $.
\end{enumerate}
\end{lemma}

\begin{proof}
\framebox{case (1)} \framebox{$ \Rightarrow $} For $ \lambda \in t^{-m} \mathcal{V}_0(R) $, denote by $ \overline{\lambda} $ the image in $ t^{-m} \mathcal{V}_0(R) / t^n \mathcal{V}_0(R) = t^{-m} \mathcal{R}[t]^d / t^n \mathcal{R}[t]^d $. \framebox{$ \subset $} Suppose $ \lambda \in \widehat{\mathcal{L}}_0 $, i.e. suppose that $ \lambda \in \mathcal{V}(R) $ satisfies $ \phi^{[0]}_R ( \mathcal{L}_0, \lambda ) \in t^{n-m} \mathcal{R}[t] $. Since $ \mathcal{L}_0 $ satisfies containments (\ref{EmathcalLcontainments}), so does $ \widehat{\mathcal{L}}_0 $ and so $ \lambda \in t^{-m} \mathcal{R}[t]^d $. Altogether, $ \overline{\lambda} \in L_0^{\perp} $ (the previous containment shows that $ \overline{\lambda} $ is in the domain of the hermitian form defining ``$ L_0^{\perp} $''). By hypothesis, $ \overline{\lambda} \in \mathcal{L}_0 $ and so $ \lambda \in \mathcal{L}_0 $. \framebox{$ \supset $} This is obvious: if $ \lambda \in \mathcal{L}_0 $ then $ \overline{\lambda} \in L_0 $ and by hypothesis $ \phi_R ( L_0, \overline{\lambda} ) = 0 $ so $ \phi^{[0]}_R ( \mathcal{L}_0, \lambda ) \in t^{n-m} \mathcal{R}[t] $ since both hermitian forms use the same Gram matrix. \framebox{$ \Leftarrow $} \framebox{$ \subset $} Suppose $ \overline{\lambda} \in L_0^{\perp} $, i.e. suppose that $ \overline{\lambda} \in t^{-m} \mathcal{R}[t]^d / t^n \mathcal{R}[t]^d $ satisfies $ \phi_R ( L_0, \overline{\lambda} ) = 0 $. Let $ \lambda \in \mathcal{L}_0 $ be any representative of $ \overline{\lambda} $. Then $ \phi^{[0]}_R ( \mathcal{L}_0, \lambda ) \in t^{n-m} \mathcal{R}[t] $ and by hypothesis, $ \lambda \in \mathcal{L}_0 $ so $ \overline{\lambda} \in L_0 $. \framebox{$ \supset $} This is obvious: if $ \overline{\lambda} \in L_0 $ then $ \lambda \in \mathcal{L}_0 $ and by hypothesis $ \phi^{[0]}_R ( \mathcal{L}_0, \lambda ) \in t^{n-m} \mathcal{R}[t] $ so $ \phi_R ( L_0, \overline{\lambda} ) = 0 $ since both hermitian forms use the same Gram matrix. \framebox{case (2)} This proof is nearly identical.
\end{proof}

The previous discussion proves the following:
\begin{BLMdef2} \label{BLMdef}
The functor $ \biggermodel{m}{n} $ can also be described as assigning to each commutative $ \Z_p $-algebra $ R $ the set of tuples $ ( \mathcal{L}_0, \mathcal{L}_1, \ldots, \mathcal{L}_{d/2} ) $ of $ ( \mathcal{O} \otimes_{\Z_p} R ) [t] $-submodules of $ \mathcal{V} ( R ) $ satisfying
\begin{enumerate}
\setlength{\itemsep}{5pt}
\item \label{BLMchain} \BLMdef{\ref{BLMchain}} \hfill \\
$ \mathcal{L}_0 \subset \mathcal{L}_1 \subset \cdots \subset \mathcal{L}_{d/2} $

\item \label{BLMbounds} \BLMdef{\ref{BLMbounds}} \hfill \\
$ t^n \mathcal{V}_i ( R ) \subset \mathcal{L}_i \subset t^{-m} \mathcal{V}_i ( R ) $ for all $ i $

\item \label{BLMsummand} \BLMdef{\ref{BLMsummand}} \hfill \\
each inclusion $ \mathcal{L}_i / t^n \mathcal{V}_i ( R ) \hookrightarrow t^{-m} \mathcal{V}_i ( R ) / t^n \mathcal{V}_i ( R ) $ splits $ R $-linearly.

\item \label{BLMlocalrank} \BLMdef{\ref{BLMlocalrank}} \hfill \\
the projective rank function $ \spec(R) \rightarrow \N $ of each $ \mathcal{L}_i / t^n \mathcal{V}_i ( R ) $ is the constant function $ \mathfrak{p} \mapsto d(m+n) $

(note that $ \rank_R ( t^{-m} \mathcal{V}_i ( R ) / t^n \mathcal{V}_i ( R ) ) = 2d(m+n) $)

\item \label{BLMselfdual} \BLMdef{\ref{BLMselfdual}} \hfill \\
$ \widehat{ \mathcal{L} }_0 = \mathcal{L}_0 $ and $ \widehat{ \mathcal{L} }_{d/2} = \mathcal{L}_{d/2} $

(these concepts of duality were defined on page \pageref{Dduality})
\end{enumerate}
\end{BLMdef2}

For future use, define
\begin{align*}
\mathcal{V}_{\textup{inf}} &= t^n \mathcal{O} [t]^d \\
\mathcal{V}_{\textup{sup}} &= t^{-m} (t+p)^{-1} \mathcal{O} [t]^d \\
\overline{\mathcal{V}}_{\textup{sup}} &= \mathcal{V}_{\textup{sup}} / \mathcal{V}_{\textup{inf}} \\
\overline{\mathcal{V}}_i &= \mathcal{V}_i / \mathcal{V}_{\textup{inf}}
\end{align*}
The first two are the largest (resp. smallest) modules contained in (resp. containing) all the modules used in $ \BLMdef{\ref{BLMbounds}} $. By convention, $ \overline{\mathcal{V}}_{\textup{sup}} ( R ) = \mathcal{V}_{\textup{sup}} ( R ) / \mathcal{V}_{\textup{inf}} ( R ) $ and similarly for $ \overline{\mathcal{V}}_i $ for any $ \Z_p $-algebra $ R $.

\subsection{The enlarged models are projective schemes} \label{SSordinarygrassmannianembedding}

For each $ 0 \leq i \leq d/2 $, let $ \grassmannian_i : \catalgebras{\Z_p} \rightarrow \catsets $ denote the (ordinary) Grassmannian of direct summands of $ t^{-m} \mathcal{V}_i / t^n \mathcal{V}_i \cong \Z_p^{2d(m+n)} $ with constant projective rank function $ d(m+n) $. Then \BLMdef{\ref{BLMsummand}} and \BLMdef{\ref{BLMlocalrank}} yield a closed embedding
\begin{equation*}
\biggermodel{m}{n} \hookrightarrow \grassmannian_0
\times \cdots \times \grassmannian_{d/2}
\end{equation*}

Take $ ( \mathcal{L}_0, \mathcal{L}_1, \ldots, \mathcal{L}_{d/2} ) \in \biggermodel{m}{n}(R) $. Consider another pair $ m^{\prime}, n^{\prime} \in \N $. If $ ( \mathcal{L}_0, \mathcal{L}_1, \ldots, \mathcal{L}_{d/2} ) \in \biggermodel{m^{\prime}}{n^{\prime}}(R) $, then necessarily $ m - n = m^{\prime} - n^{\prime} $ since $ \mathcal{L}_0 $ can only be self-dual with respect to $ t^N \phi $ for one $ N $. On the other hand, if $ m^{\prime} \geq m $ and $ n^{\prime} \geq n $, then requirement \BLMdef{\ref{BLMbounds}} for $ \biggermodel{m^{\prime}}{n^{\prime}}(R) $ trivially follows from requirement \BLMdef{\ref{BLMbounds}} for $ \biggermodel{m}{n}(R) $.

These two requirements on $ m, n, m^{\prime}, n^{\prime} $ already imply that $ \biggermodel{m}{n} \subset \biggermodel{m^{\prime}}{n^{\prime}} $: the $ R $-linear splitting of the short-exact-sequence
\begin{equation*}
0 \rightarrow t^n \mathcal{V}_i(R) / t^{n^{\prime}} \mathcal{V}_i(R)
\rightarrow \mathcal{L}_i / t^{n^{\prime}} \mathcal{V}_i(R)
\rightarrow \mathcal{L}_i / t^n \mathcal{V}_i(R) \rightarrow 0
\end{equation*}
shows that \BLMdef{\ref{BLMsummand}} is satisfied, and this sequence
also shows that
\begin{align*}
\rank_R ( \mathcal{L}_i / t^{n^{\prime}} \mathcal{V}_i(R) ) &= \rank_R ( t^n \mathcal{V}_i(R) / t^{n^{\prime}} \mathcal{V}_i(R) ) + \rank_R ( \mathcal{L}_i / t^n \mathcal{V}_i(R) ) \\
&= 2(n^{\prime} - n )d + (m+n)d \\
&= (m^{\prime} - m )d + (n^{\prime} - n )d + (m+n)d \\
&= (m^{\prime} + n^{\prime})d
\end{align*}
(the ``$ 2 $'' here is $ \rank_R ( \mathcal{R} ) $).

In summary, for each $ \Delta \in \Z $, the set of $ ( m, n ) \in \N
\times \N $ such that $ n - m = \Delta $ is totally-ordered and
\begin{equation*}
\biggermodel{0}{\Delta} \subset \biggermodel{1}{1+\Delta} \subset
\biggermodel{2}{2+\Delta} \subset \cdots
\end{equation*}

\begin{remark}
In \S\ref{SSaffineflagembedding} (page \pageref{SSaffineflagembedding}), I will embed $ \biggermodel{m}{n}_{\F_p} $ into an affine flag variety, and from that perspective, the chain associated to a particular $ \Delta \in \Z $ exhausts the corresponding connected component of the affine flag variety.
%See Theorem 5.1 in \cite{PR} for a way to calculate the component group of an affine flag variety in the non-split case.

%Notice that if $ m = n $ then the trivial tuple $ ( \mathcal{V}_0(R), \ldots, \mathcal{V}_{d/2}(R) ) $ satisfies all the conditions necessary for membership in $ \biggermodel{m}{m} (R) $, the assumption being required for the rank. From the point of view of the affine flag variety, this is because the identity component of the affine flag variety is indexed by $ \Delta = 0 $.
\end{remark}

\subsection{Equivalent characterizations of Zariski-lattices} \label{latticeequivalences}

The description of the functor-of-points of an affine flag variety uses a certain concept of lattice, but other characterizations are needed to embed local models into affine flag varieties. The list of characterizations is summarized as:

\begin{latticedefs} [Lemma 2.11 in \cite{Go2}]
Let $ \mathcal{R} $ be a commutative ring and let $ M \subset \mathcal{R}((t))^d $ be an $ \mathcal{R}[[t]] $-submodule. Each of the following three sets of conditions are equivalent to each other:
\begin{enumerate}
\setlength{\itemsep}{5pt}
\item \label{gortzlattice}
\begin{enumerate}
\item there exists $ N $ such that $ t^N \mathcal{R}[[t]]^d \subset M \subset t^{-N} \mathcal{R}[[t]]^d $
\item as an $ \mathcal{R} $-module, the quotient $ M / t^N \mathcal{R}[[t]]^d $ is
projective
\end{enumerate}
\item \label{pappaslattice}
\begin{enumerate}
\item the product $ M \otimes_{\mathcal{R}[[t]]} \mathcal{R}((t)) \rightarrow \mathcal{R}((t))^d
$ is an isomorphism
\item as an $ \mathcal{R}[[t]] $-module, $ M $ is finitely-generated and projective
\item the projective rank function $ \spec( \mathcal{R}[[t]] ) \rightarrow \N $ of $ M $ is the constant function $ \mathfrak{p} \mapsto d $
\end{enumerate}
\item \label{smithlinglattice}
\begin{enumerate}
\item the product $ M \otimes_{\mathcal{R}[[t]]} \mathcal{R}((t)) \rightarrow \mathcal{R}((t))^d
$ is an isomorphism
\item Zariski-locally on $ \spec (\mathcal{R}) $, $ M $ is a free $ \mathcal{R}[[t]]
$-module (and it is automatic that the rank is always $ d $)
\end{enumerate}
%\item \label{fpqclattice}
%\begin{enumerate}
%\item the product $ M \otimes_{\mathcal{R}[[t]]} \mathcal{R}((t)) \rightarrow \mathcal{R}((t))^d $ is an isomorphism
%\item fpqc-locally on $ \spec (\mathcal{R}) $, $ M $ is a free $ \mathcal{R}[[t]] $-module
%\end{enumerate}
\end{enumerate}
\end{latticedefs}

Condition (\ref{smithlinglattice}b) means that there are elements $ r_1, \ldots, r_n $ generating the trivial ideal $ R $ (i.e. a system of principal open sets covering $ \spec(R) $) such that each fraction module $ M [ r_i^{-1} ] $ is a free $ R_{r_i} [[t]] $-module.

\begin{terminologynotation}
Such an $ M $ will be called a ``Zariski $ \mathcal{R}[[t]] $-lattice in $ \mathcal{R}((t))^d $'' or similar.
\end{terminologynotation}

\subsection{The full affine flag variety over $ \F_p $} \label{SSaffineflagdesc}

The setup here is: I use the unramified quadratic extension $ \F_p((t)) \subset \F((t)) $, the vector space $ \F((t))^d $, and the standard hermitian form $ \Phi : \F((t))^d \times \F((t))^d \rightarrow \F((t)) $ defined by the anti-identity matrix $ \antiid $. Note that $ \Phi $ induces $ \phi $ from previous sections.
\begin{completeaffineflagvarietydef} \label{AFVdef}
The functor
\begin{equation*}
\affineflagvariety : \catalgebras{\F_p} \longrightarrow \catsets
\end{equation*}
assigns to each commutative $ \F_p $-algebra $ R $ (set $ \mathcal{R} := \F \otimes_{\F_p} R $) the set of all sequences $ ( \mathcal{F}_i )_{i \in \Z} $ of $ \mathcal{R}[[t]] $-submodules of $ \mathcal{R}((t))^d $ satisfying the following properties:
\begin{enumerate}
\setlength{\itemsep}{5pt}
\item \label{AFVperiodic} \AFVdef{\ref{AFVperiodic}} \hfill \\
For each $ i $, $ \mathcal{F}_i = t \cdot \mathcal{F}_{i+d} $,
\item \label{AFVchain} \AFVdef{\ref{AFVchain}} \hfill \\
for each $ i $, $ \cdots \subset \mathcal{F}_i \subset \mathcal{F}_{i+1} \subset \cdots $,
\item \label{AFVlattice} \AFVdef{\ref{AFVlattice}} \hfill \\
for each $ i $, $ \mathcal{F}_i $ is a Zariski $ \mathcal{R}[[t]] $-lattice in $ \mathcal{R}((t))^d $,
\item \label{AFVcomplete} \AFVdef{\ref{AFVcomplete}} \hfill \\
for each $ i $, $ \mathcal{F}_{i+1} / \mathcal{F}_i $ is, Zariski-locally on $ \spec(R) $, free with rank $ 1 $, and
\item \label{AFVsimduality} \AFVdef{\ref{AFVsimduality}} \hfill \\
there exists, Zariski-locally on $ \spec(R) $, $ u(t) \in R((t))^{\times} $ such that for all $ i $,
\begin{equation*}
\mathcal{F}_{-i} = u(t) \cdot \widehat{\mathcal{F}}_i.
\end{equation*}

(here $ \widehat{\mathcal{F}} \defeq \{ w \in \mathcal{R}((t))^d : \Phi_R ( \mathcal{F}, w ) \subset \mathcal{R}[[t]]^d \} $)
\end{enumerate}
\end{completeaffineflagvarietydef}

\begin{remark}
For a commutative ring homomorphism $ R \rightarrow S $, the function $ \affineflagvariety ( R ) \rightarrow \affineflagvariety ( S ) $ is the natural map of the completed tensor product $ - \widehat{\otimes}_R S $, which satisfies $ R[[t]] \widehat{\otimes}_R S = S[[t]] $.
%, which is designed so that $ R[[t]] \widehat{\otimes}_R S = S[[t]] $ etc.
\end{remark}

%\begin{remark}
%Note that in \AFVdef{\ref{AFVsimduality}}, the Zariski-local cover of $ \spec ( R ) $ and the corresponding similitudes are independent of $ i $.
%\end{remark}

As in \S\ref{SSsimpledesc}, a point $ ( \mathcal{F}_i )_{i \in \Z} \in \affineflagvariety ( R ) $ is completely determined by the finite chain $ \mathcal{F}_0 \subset \cdots \subset \mathcal{F}_{d/2} $ as follows: recover the Zariski-local cover of $ \spec ( R ) $ and the common degree $ k $ of the similitudes in \AFVdef{\ref{AFVsimduality}} by comparing $ \mathcal{F}_0 $ to $ \widehat{\mathcal{F}}_0 $, define $ \mathcal{F}_{-i} $ for $ 0 < i < d/2 $ Zariski-locally by $ \mathcal{F}_{-i} \defeq t^k \widehat{\mathcal{F}}_i $, and extend periodically. It is automatic from the definition that $ \cdots \subset \mathcal{F}_{-2} \subset \mathcal{F}_{-1} \subset \mathcal{F}_0 $. Using the finite chain, \AFVdef{\ref{AFVperiodic}} disappears and the only part of \AFVdef{\ref{AFVsimduality}} that survives is: there exists Zariski-locally on $ \spec(R) $ a $ u(t) \in R((t))^{\times} $ such that
\begin{align*}
\mathcal{F}_0 &= u(t) \cdot \widehat{\mathcal{F}}_0 \\
\mathcal{F}_{d/2} &= t^{-1} u(t) \cdot \widehat{\mathcal{F}}_{d/2}
\end{align*}

\subsection{The special fibers are subschemes of the full affine flag variety} \label{SSaffineflagembedding}

Let $ R $ be a commutative $ \F_p $-algebra. Consider $ ( \mathcal{L}_0, \mathcal{L}_1, \ldots, \mathcal{L}_{d/2} ) \in \biggermodel{m}{n} ( R ) $. First, note that
\begin{align*}
\mathcal{V} (\F_p) &= \F [ t, t^{-1} ]^d \subset \F((t))^d \\
\mathcal{V}_0 (\F_p) &= \F[t]^d \\
\mathcal{V}_1 (\F_p) &= t^{-1} \F[t] \oplus \F[t]^{d-1} \\
&{\vdots} \\
\mathcal{V}_{d-1} (\F_p) &= t^{-1} \F[t]^{d-1} \oplus \F[t] \\
\mathcal{V}_d(\F_p) &= t^{-1} \F[t]^d
\end{align*}

So each $ \mathcal{L}_i $ is an $ \mathcal{R}[t] $-submodule of $ \mathcal{V}(R) \subset \mathcal{R}((t))^d $ satisfying
\begin{equation} \label{absolutedegreebounds}
t^n \mathcal{R}[t]^d = \mathcal{V}_{\textup{inf}} (R) \subset \mathcal{L}_i \subset \mathcal{V}_{\textup{sup}} (R) = t^{-(m+1)} \mathcal{R}[t]^d
\end{equation}
Such modules are in canonical bijection with $ \mathcal{R}[[t]] $-submodules $ \mathcal{F}_i $ of $ \mathcal{R}((t))^d $ satisfying
\begin{equation} \label{powerseriesextension}
t^n \mathcal{R}[[t]]^d \subset \mathcal{F}_i \subset t^{-(m+1)} \mathcal{R}[[t]]^d
\end{equation}

Let $ \mathcal{F}_0, \mathcal{F}_1, \ldots, \mathcal{F}_{d/2} $ be the modules satisfying equation (\ref{powerseriesextension}) corresponding to the modules $ \mathcal{L}_0, \mathcal{L}_1, \ldots, \mathcal{L}_{d/2} $ in equation (\ref{absolutedegreebounds}).
\begin{prop} \label{Paffineflagcontainment}
$ ( \mathcal{F}_0, \mathcal{F}_1, \ldots, \mathcal{F}_{d/2} ) \in \affineflagvariety ( R ) $.
\end{prop}

Once this is proven, it is obvious that $ \biggermodel{m}{n} (R) \rightarrow \affineflagvariety ( R ) $ is \emph{injective} for all $ R $. By \S\ref{SSordinarygrassmannianembedding} (page \pageref{SSordinarygrassmannianembedding}), $ \biggermodel{m}{n} $ is a proper $ \Z_p $-scheme, so $ \biggermodel{m}{n} \hookrightarrow \affineflagvariety $ is a proper morphism. By Corollary 12.92 of \cite{GW}, $ \biggermodel{m}{n} \hookrightarrow \affineflagvariety $ is a \emph{closed embedding}.

\begin{proof}
\framebox{\AFVdef{\ref{AFVperiodic}}} This is unnecessary due to the above replacement of infinite chains by finite chains. \framebox{\AFVdef{\ref{AFVchain}}} This is automatic. \framebox{\AFVdef{\ref{AFVlattice}}} For this, it is obviously most convenient to use characterization (\ref{gortzlattice}) of Zariski-lattices (page \pageref{gortzlattice}). Equation (\ref{powerseriesextension}) gives part (a) of the characterization. Now notice that the coefficient ring of the power series here is $ \mathcal{R} = R \otimes_{\F_p} \F $, not $ R $. This means that in this situation, condition (b) actually requires $ \overline{\mathcal{L}_i} $ to be a projective as an $ \mathcal{R} $-module. \BLMdef{\ref{BLMsummand}} and \BLMdef{\ref{BLMlocalrank}} together imply only projectivity over $ R $, but since $ R \rightarrow \mathcal{R} $ makes $ \mathcal{R} $ a finitely-generated $ R $-module and a faithfully-flat $ R $-algebra, Lemma \ref{Lprojectivitylifting} (page \pageref{Lprojectivitylifting}) says that $ \overline{ \mathcal{L}_i } $ is in fact a projective $ \mathcal{R} $-module. \framebox{\AFVdef{\ref{AFVcomplete}}} By construction, it suffices to prove the same fact for the quotient $ \mathcal{L}_{i+1} / \mathcal{L}_i $. It is easy to prove but the notation becomes truly oppressive. For clarity, I prove this as Lemma \ref{Lrank1quotient} after the current proof. \framebox{\AFVdef{\ref{AFVsimduality}}} I claim that $ \widehat{\mathcal{F}}_0 = t^{m-n} \mathcal{F}_0 $ and $ \widehat{\mathcal{F}}_{d/2} = t^{m-n+1} \mathcal{F}_{d/2} $ (in other words, that $ u(t) = t^{n-m} $). This is clear because the products used in \BLMdef{\ref{BLMselfdual}} for $ \mathcal{L}_0 $ and $ \mathcal{L}_{d/2} $ are just the standard ones multiplied by $ t^{m-n} $ and $ t^{m-n} (t+p) $, and the concept of duality is the same.
\end{proof}

%\begin{remark}
%Note that the need for Lemma \ref{Lprojectivitylifting} (page \pageref{Lprojectivitylifting}) does not occur in the case of a \emph{ramified} unitary group, since in that case the coefficient ring does not change (the \emph{uniformizer} changes).
%\end{remark}

\begin{lemma} \label{Lrank1quotient}
Consider an $ \mathcal{R} $-module diagram
\begin{equation*}
\begin{array}{ccc}
\mathcal{L}_1 & \subset & V_1 \\
\cap & & \cap \\
\mathcal{L}_2 & \subset & V_2
\end{array}
\end{equation*}

\textbf{Assertion}: If (1) both quotients $ V_1 / \mathcal{L}_1 $ and $ V_2 / \mathcal{L}_2 $ are finitely-generated projective $ R $-modules, (2) the projective rank functions of $ V_1 / \mathcal{L}_1 $ and $ V_2 / \mathcal{L}_2 $ are constant on $ \spec(R) $ and the two constants are equal, and (3) the quotient $ V_2 / V_1 $ is a free $ \mathcal{R} $-module of rank $ 1 $ then $ \mathcal{L}_2 / \mathcal{L}_1 $ is a projective $ \mathcal{R} $-module with constant projective rank function $ 1 $.
\end{lemma}

To use this in the proof of Proposition \ref{Paffineflagcontainment}, use the diagram
\begin{equation*}
\begin{array}{ccc} \mathcal{L}_i & \subset
& t^{-m}
\mathcal{V}_i (R) \\
\cap & & \cap \\
\mathcal{L}_{i+1} & \subset & t^{-m} \mathcal{V}_{i+1} (R)
\end{array}
\end{equation*}
from \BLMdef{\ref{BLMbounds}}. The vertical inclusion on the right is valid (i.e. there are no problems due to non-flatness) because $ \mathcal{V}_i \subset \mathcal{V}_{i+1} $ can be written as the obvious inclusion $ \mathcal{O}^{\N} \subset \mathcal{O}^{\N} \oplus \mathcal{O} $.

Hypothesis (1) in this lemma is implied by \BLMdef{\ref{BLMsummand}} and hypothesis (2) follows from \BLMdef{\ref{BLMlocalrank}}. Hypothesis (3) is obvious from the definition of the $ \mathcal{V}_i $.

\begin{proof}
By hypothesis (1), the short-exact-sequence
\begin{equation*}
0 \rightarrow ( \mathcal{L}_2 / \mathcal{L}_1 ) \rightarrow ( V_2 /
\mathcal{L}_1 ) \rightarrow ( V_2 / \mathcal{L}_2 ) \rightarrow 0
\end{equation*}
of $ \mathcal{R} $-modules splits $ R $-linearly. By hypothesis (3), the short-exact-sequence
\begin{equation*}
0 \rightarrow ( V_1 / \mathcal{L}_1 ) \rightarrow ( V_2 /
\mathcal{L}_1 ) \rightarrow ( V_2 / V_1 ) \rightarrow 0
\end{equation*}
of $ \mathcal{R} $-modules splits $ \mathcal{R} $-linearly. These
splittings produce an $ R $-module isomorphism
\begin{equation} \label{eqnAFVcomplete}
( \mathcal{L}_2 / \mathcal{L}_1 ) \oplus ( V_2 / \mathcal{L}_2 )
\cong ( V_1 / \mathcal{L}_1 ) \oplus ( V_2 / V_1 )
\end{equation}
By hypotheses (1) and (3), the right-hand-side of equation (\ref{eqnAFVcomplete}) is a projective $ R $-module, which shows that $ \mathcal{L}_2 / \mathcal{L}_1 $ is a projective $ R $-module. By Lemma \ref{Lprojectivitylifting}, $ \mathcal{L}_2 / \mathcal{L}_1 $ is a projective $ \mathcal{R} $-module. Counting projective ranks over $ R $ on both sides of equation (\ref{eqnAFVcomplete}) and using hypothesis (2) shows that the projective rank function $ \spec(R) \rightarrow \N $ of $ \mathcal{L}_2 / \mathcal{L}_1 $ is the constant function $ \mathfrak{p} \mapsto \rank_R ( V_2 / V_1 ) $. Hypothesis (3) finishes the proof (use $ \rank_R ( \mathcal{R} ) = 2 $).
\end{proof}

\subsection{The special fibers exhaust the full affine flag variety} \label{SSspecialfibersexhaust}

Let $ R $ be a commutative $ \F_p $-algebra and set $ \mathcal{R} := \mathcal{O} \otimes_{\Z_p} R $. Define the ``standard'' periodic $ \Phi $-selfdual chain $ \lambda_{\bullet} = ( \cdots \subset \lambda_{-1} \subset \lambda_0 \subset \lambda_1 \subset \cdots ) $ of $ \F[[t]] $-lattices in $ \F((t))^d $ by extending
\begin{equation*}
\lambda_i \defeq t^{-1} \F[[t]]^i \oplus \F[[t]]^{d-i} \textaftermath{($ 0 \leq i \leq d $)}
\end{equation*}
periodically. Set $ \lambda_i ( R ) := \lambda_i \widehat{\otimes}_{\F_p} R $.

Let $ ( \mathcal{F}_i )_{i \in \Z} \in \affineflagvariety (R) $ be arbitrary and let $ u(t) \in \F_p((t)) $ be the similitude for $ ( \mathcal{F}_i )_{i \in \Z} $ guaranteed by \AFVdef{\ref{AFVsimduality}}. Set $ \Delta := \val_t ( u(t) ) \in \Z $. From the verification of \AFVdef{\ref{AFVsimduality}} in the previous subsection, I expect that $ \Delta $ will be the future value of $ m - n $. So, let $ m, n \in \N $ be such that $ m - n = \Delta $ and simultaneously so large that
\begin{equation} \label{EQadjustedlatticebounds}
t^n \lambda_i (R) \subset \mathcal{F}_i \subset t^{-m} \lambda_i (R)
\end{equation}
for all $ i = 0, 1, \ldots, d/2 $. From the remark in \S\ref{SSordinarygrassmannianembedding} (page \pageref{SSordinarygrassmannianembedding}), there is no danger in choosing $ m, n $ too large. The claim is that $ ( \mathcal{F}_i )_{i \in \Z} $ is the image of a point of $ \biggermodel{m}{n} ( R ) $.

By passing through the quotients $ t^{-m} \lambda_i ( R ) / t^n \lambda_i ( R ) \cong t^{-m} \mathcal{V}_i ( R ) / t^n \mathcal{V}_i ( R ) $, define $ \mathcal{L}_i $ to be the $ \mathcal{R}[t] $-module satisfying
\begin{equation*}
t^n \mathcal{V}_i (R) \subset \mathcal{L}_i \subset t^{-m} \mathcal{V}_i (R)
\end{equation*}
corresponding to $ \mathcal{F}_i $ for each $ i = 0, 1, \ldots, d/2 $. It is obvious that $ \mathcal{L}_0 \subset \cdots \subset \mathcal{L}_{d/2} $, so both \BLMdef{\ref{BLMchain}} and \BLMdef{\ref{BLMbounds}} are true already. Requirement \BLMdef{\ref{BLMsummand}} follows from \AFVdef{\ref{AFVlattice}} because the quotients $ t^{-m} \lambda_i ( R ) / \mathcal{F}_i $ and $ t^{-m} \mathcal{V}_i (R) / \mathcal{L}_i $ are, by construction, identified $ R $-linearly. For each $ i $, I have the short-exact-sequence of $ \F[[t]] $-modules
\begin{equation*}
0 \rightarrow \mathcal{F}_{i+1} / \mathcal{F}_i \rightarrow t^{-m} \mathcal{R}[[t]]^d / \mathcal{F}_i \rightarrow t^{-m} \mathcal{R}[[t]]^d / \mathcal{F}_{i+1} \rightarrow 0
\end{equation*}
By \AFVdef{\ref{AFVlattice}} and \AFVdef{\ref{AFVcomplete}},
\begin{equation*}
\rank_{\mathcal{R}}( t^{-m} \mathcal{R}[[t]]^d / \mathcal{F}_i ) = \rank_{\mathcal{R}} ( t^{-m} \mathcal{R}[[t]]^d / \mathcal{F}_{i+1} ) + 1
\end{equation*}
Zariski-locally on $ \spec ( \mathcal{R} ) $, which implies that
\begin{equation*}
\rank_{\mathcal{R}} ( \mathcal{L}_i / t^n \mathcal{V}_i(R) ) = \rank_{\mathcal{R}} ( \mathcal{L}_{i+1} / t^n \mathcal{V}_{i+1}(R) )
\end{equation*}
Zariski-locally on $ \spec ( \mathcal{R} ) $ (note that the quotient on the right is by a slightly larger module than on the left).

To verify \BLMdef{\ref{BLMlocalrank}}, it now suffices to show that the projective rank function $ \spec(R) \rightarrow \N $ of $ \mathcal{F}_0 / t^n \lambda_0(R) $ is the constant function $ d(m+n) $. Dualizing chain (\ref{EQadjustedlatticebounds}) yields
\begin{equation*}
t^m \lambda_0(R) \subset \widehat{\mathcal{F}}_0 \subset t^{-n} \lambda_0(R)
\end{equation*}
The quotient $ t^{-n} \lambda_0 (R) / \widehat{\mathcal{F}}_0 $ is a projective $ \mathcal{R} $-module and has the same (constant) projective rank function as $ \mathcal{F}_0 / t^n \lambda_0(R) $. This means that
\begin{align*}
\rank_{\mathcal{R}} ( \mathcal{F}_0 / t^n \lambda_0(R) )
&= \rank_{\mathcal{R}} ( t^{-n} \lambda_0 (R) / \widehat{\mathcal{F}}_0 ) \\
&= \rank_{\mathcal{R}} ( t^{-m} \lambda_0 (R) / t^{n-m} \widehat{\mathcal{F}}_0 )
= \rank_{\mathcal{R}} ( t^{-m} \lambda_0 (R) / \mathcal{F}_0 )
\end{align*}
The first and last ranks are equal and must sum to $ d(m+n) $, so the claim is proven (recall that $ \rank_R ( \mathcal{R} ) = 2 $).

It is automatic that \BLMdef{\ref{BLMselfdual}} is satisfied, because the definition of $ \widehat{\mathcal{F}}_0 $ and condition \AFVdef{\ref{AFVsimduality}} guarantee that $ \mathcal{L}_0 = \{ x \in \mathcal{V}(R) \suchthat \phi_R ( \mathcal{L}_0, x ) \subset t^{-\Delta} \mathcal{R}[t] \} $ and by choice $ \Delta = m - n $ (and similarly for $ \mathcal{L}_{d/2} $).

\subsection{The affine Grassmannian over $ \Q_p $} \label{SSaffinegrassdesc}

For the affine Grassmannian, the setup is: I use the unramified quadratic extension $ \Q_p((t)) \subset F((t)) $, the vector space $ F((t))^d $, and the standard hermitian form $ \Phi : F((t))^d \times F((t))^d \rightarrow F((t)) $ defined by the anti-identity matrix $ \antiid $.
\begin{affinegrassmanniandef} \label{AGVdef}
The affine Grassmannian $ \affinegrassmannian $ is the functor that assigns to any commutative $ \Q_p $-algebra $ R $ ($ \mathcal{R} := R \otimes_{\Q_p} F $) the set of all $ \mathcal{R}[[t]] $-submodules $ \mathcal{F} $ of $ \mathcal{R}((t))^d $ satisfying:
\begin{enumerate}
\setlength{\itemsep}{5pt}
\item \label{AGVlattice} \AGVdef{\ref{AGVlattice}} \hfill \\
$ \mathcal{F} $ is a Zariski $ \mathcal{R}[[t]] $-lattice in $ \mathcal{R}((t))^d $
\item \label{AGVsimduality} \AGVdef{\ref{AGVsimduality}} \hfill \\
there exists, Zariski-locally on $ \spec(R) $, $ u(t) \in R((t))^{\times} $ such that
\begin{equation*}
\mathcal{F} = u(t) \cdot \widehat{\mathcal{F}}.
\end{equation*}

\end{enumerate}
\end{affinegrassmanniandef}

\subsection{The generic fibers are subschemes of the affine Grassmannian} \label{SSaffinegrassmannianembedding}

Let $ R $ be a commutative $ \Q_p $-algebra. Consider a point $ ( \mathcal{L}_0, \mathcal{L}_1, \ldots, \mathcal{L}_{d/2} ) \in \biggermodel{m}{n} ( R ) $. As before, note that
\begin{align*}
\mathcal{V}(\Q_p) &= F [ t, t^{-1}, (t+p)^{-1} ]^d \subset F((t))^d \\
\mathcal{V}_0(\Q_p) &= F[t]^d \\
\mathcal{V}_1(\Q_p) &= (t+p)^{-1} F[t] \oplus F[t]^{d-1} \\
&{\vdots} \\
\mathcal{V}_{d-1}(\Q_p) &= (t+p)^{-1} F[t]^{d-1} \oplus F[t] \\
\mathcal{V}_d(\Q_p) &= (t+p)^{-1} F[t]^d
\end{align*}

Since $ t+p $ is a unit in $ \Q_p[[t]] $, the discussion in \S\ref{SSlargermodels} (page \pageref{SSlargermodels}) implies that over $ \Q_p $ the ``chain'' $ \mathcal{L}_0 \subset \mathcal{L}_1 \subset \cdots \subset \mathcal{L}_{d/2} $ collapses and I may ignore all of them except one, say $ \mathcal{L} := \mathcal{L}_0 $.

Using a similar argument as in \S\ref{SSaffineflagembedding}, notice that passing through the isomorphism of quotients
\begin{equation*}
\frac{ t^{-m}(t+p)^{-1} \mathcal{R}[t]^d }{ t^n \mathcal{R}[t]^d } \directedisom \frac{ t^{-m} \mathcal{R}[[t]]^d }{ t^n
\mathcal{R}[[t]]^d }
\end{equation*}
(this uses the fact that $ t+p $ is a \emph{unit} in $ \Q_p[[t]] $) associates to $ \mathcal{L} $ an $ \mathcal{R}[[t]] $-submodule $ \mathcal{F} $ of $ \mathcal{R}((t))^d $ satisfying
\begin{equation*}
t^n \mathcal{R}[[t]]^d \subset \mathcal{F} \subset t^{-m}
\mathcal{R}[[t]]^d
\end{equation*}

I claim that $ \mathcal{L} \mapsto \mathcal{F} $ defines an injective function $ \biggermodel{m}{n} (R) \hookrightarrow \affinegrassmannian_{ \Q_p } $ and that the collection over all $ \Q_p $-algebras $ R $ of these functions defines a natural transformation $ \biggermodel{m}{n}_{\Q_p} \rightarrow \affinegrassmannian_{ \Q_p } $.

\framebox{Proof of \AGVdef{\ref{AGVlattice}}} The proof here is the same as the proof of \AFVdef{\ref{AFVlattice}} (page \pageref{AFVlattice}): apply Lemma \ref{Lprojectivitylifting} (page \pageref{Lprojectivitylifting}) to $ R \rightarrow R \otimes_{\Q_p} F $.

\framebox{Proof of \AGVdef{\ref{AGVsimduality}}} The proof here is exactly the same as the proof of \AFVdef{\ref{AFVsimduality}} (page \pageref{AFVsimduality}), noticing that $ t+p $ is a unit in $ \Q_p[[t]] $ etc.

\subsection{The generic fibers exhaust the affine Grassmannian} \label{SSgenericfibersexhaust}

The proof of the fact that the schemes $ \biggermodel{m}{n}_{\Q_p} $ exhaust $ \affinegrassmannian_{ \Q_p } $ is nearly identical to the case of the special fibers exhausting the full affine flag variety, noticing as usual that $ t+p $ is a unit in $ \Q_p[[t]] $ etc.

\section{Automorphisms of the enlarged models} \label{Sautomgroup}

\begin{center}
\emph{The group $ \biggergroup{m}{n} $ defined in this subsection is a degeneration of the special parahoric subgroup $ \specialparahoric $ over $ \Q_p[[t]] $ to the Iwahori subgroup $ \iwahori $ over $ \F_p[[t]] $, much in the same way that $ \biggermodel{m}{n} $ degenerates the affine Grassmannian $ \affinegrassmannian $ over $ \Q_p $ to the full affine flag variety $ \affineflagvariety $ over $ \F_p $.}
\end{center}

\subsection{Definition and basic properties}

\begin{definition}
The functor
\begin{equation*}
\biggergroup{m}{n} : \catalgebras{\Z_p} \rightarrow \catgroups
\end{equation*}
assigns to any commutative $ \Z_p $-algebra $ R $ ($ \mathcal{R} := R \otimes_{\Z_p} \mathcal{O} $) the group of all $ \mathcal{R} [t] $-linear automorphisms $ g $ of the quotient $ \overline{\mathcal{V}}_{\textup{sup}}(R) $ satisfying:
\begin{itemize}
\setlength{\itemsep}{5pt}
\item $ g ( \overline{\mathcal{V}}_i (R) ) = \overline{\mathcal{V}}_i (R) $

\item there exists $ c(g) \in R[t] $ representing a \emph{unit} of $ R[t] / t^{m+n}(t+p) R[t] $ such that $ \overline{\phi}_R ( g ( x ), g ( y ) ) = c ( g ) \cdot \overline{\phi}_R ( x, y ) $ for all $ x, y \in \overline{\mathcal{V}}_{\textup{sup}} ( R ) $.
\end{itemize}
Here $ \overline{\phi}_R $ denotes the product 
\begin{equation*} \label{EhermitianformusedforJ}
\overline{\phi}_R : \overline{\mathcal{V}}_{\textup{sup}}(R) \times \overline{\mathcal{V}}_{\textup{sup}}(R) \longrightarrow \frac{ t^{-2m} (t+p)^{-2} \mathcal{R} [t] }{ t^{n-m} (t+p)^{-1} \mathcal{R} [t] }
\end{equation*}
defined by the anti-identity matrix $ \antiid $.
\end{definition}

%\begin{remark}The ``precision'' $ t^{m+n}(t+p) $ for $ c(g) $ is sufficient because the codomain of $ \overline{\phi} $ is isomorphic as an $ \mathcal{R} [t] $-module to $ \mathcal{R}[t] / t^{m+n}(t+p) \mathcal{R}[t] $.\end{remark}

\begin{lemma} \label{LJaction}
This functor $ \biggergroup{m}{n} $ is represented by a finite-type affine group $ \Z_p $-scheme and the rule $ g \cdot ( \mathcal{L}_0, \mathcal{L}_1, \ldots, \mathcal{L}_{d/2} ) \defeq ( g( \mathcal{L}_0 ), g( \mathcal{L}_1 ), \ldots, g( \mathcal{L}_{d/2} ) ) $ defines an action of $ \biggergroup{m}{n} $ on $ \biggermodel{m}{n} $.
\end{lemma}

\begin{proof}
\framebox{affine} The condition that any $ g \in \biggergroup{m}{n} $ must stabilize the filtration $ \overline{\mathcal{V}}_i $ together with the condition that it be a similitude present $ \biggergroup{m}{n} $ as a closed subscheme of $ \aut_{\mathcal{O} [t]\textup{-lin}}( \overline{\mathcal{V}}_{\textup{sup}} ) $. The condition that $ g \in \aut_{\mathcal{O} [t]\textup{-lin}}( \overline{\mathcal{V}}_{\textup{sup}} ) $ be $ \mathcal{O} [t] $-linear rather than simply $ \mathcal{O} $-linear presents $ \aut_{\mathcal{O} [t]\textup{-lin}}( \overline{\mathcal{V}}_{\textup{sup}} ) $ as a closed subscheme of $ \GL_{2d(m+n)} $ (this condition is the same as requiring that $ g $ commute with the operator $ t $). \framebox{finite-type} This is obvious from the proof of affine-ness. \framebox{action} Let $ R $ be a commutative $ \Z_p $-algebra and set $ \mathcal{R} := \mathcal{O} \otimes_{\Z_p} R $. Take $ ( \mathcal{L}_i )_{i=0}^{d/2} \in \biggermodel{m}{n} (R) $ and take $ g \in \biggergroup{m}{n} (R) $. It is obvious from the definition that $ ( g ( \mathcal{L}_i ) )_{i=0}^{d/2} $ satisfies \BLMdef{\ref{BLMchain}}, \BLMdef{\ref{BLMbounds}}, \BLMdef{\ref{BLMsummand}} and \BLMdef{\ref{BLMlocalrank}}. To prove \BLMdef{\ref{BLMselfdual}}, note that restricting the above $ \overline{\phi} $ to $ t^{-m} \mathcal{R}[t]^d / t^n \mathcal{R}[t]^d $ agrees with the $ \phi $ used in \S\ref{SSlargermodels} (page \pageref{SSlargermodels}). Also note that the image of $ \mathcal{L}_0 $ in $ t^{-m} \mathcal{R}[t]^d / t^n \mathcal{R}[t]^d $ is identical to what was called $ L_0 $ in \S\ref{SSlargermodels} (page \pageref{SSlargermodels}). Therefore, by Lemma \ref{isotropiciffdual} (page \pageref{isotropiciffdual}), I need to show that $ \phi_R ( g(L_0), x ) = 0 $ if and only if $ x \in g(L_0) $. \framebox{$ \Rightarrow $} Since $ 0 = \phi_R ( g(L_0), x ) = c(g) \phi_R ( L_0, g^{-1} ( x ) ) $, the definition of $ c(g) $ implies that $ \phi_R ( L_0, g^{-1} ( x ) ) = 0 $ also. So $ g^{-1} ( x ) \in L_0 $ and $ x \in g(L_0) $. \framebox{$ \Leftarrow $} This is trivial since $ c(g) \in R[t] $. The case of $ \mathcal{L}_{d/2} $ is similar. It is trivial that this is a group action and that these actions form a natural transformation $  \biggergroup{m}{n} \times \biggermodel{m}{n} \rightarrow \biggermodel{m}{n} $.
\end{proof}

\subsection{The automorphism group is smooth} \label{SSsmoothJ}

\begin{center}
\emph{The purpose of this subsection is to prove that $ \biggergroup{m}{n} $ is a smooth. This is necessary to connect the equivariant sheaf theory of $ \biggermodel{m}{n}_{\Q_p} $ to the equivariant sheaf theory of $ \biggermodel{m}{n}_{\F_p} $, i.e. to apply base-change for pullbacks.}
\end{center}

By Lemma \ref{LJaction} (page \pageref{LJaction}), $ \biggergroup{m}{n} $ is finite-type, so Corollary 4.5 (and its proof) of \cite{DG} says that $ \biggergroup{m}{n} \rightarrow \spec(\Z_p) $ is smooth if and only if the following property is true:

\begin{center}
\begin{formalsmoothness}
\label{formalsmoothness}For each local commutative $ \Z_p $-algebra $ R $ and each ideal $ I \subset R $ satisfying $ I^2 = 0 $, the map $ \biggergroup{m}{n} ( R ) \rightarrow \biggergroup{m}{n} ( R / I ) $ is \emph{surjective}.
\end{formalsmoothness}
\end{center}

Fix such an $ R $ and $ I \subset R $. Set
\begin{align*}
S &\defeq R[t] / t^{m+n} (t+p) R[t] \\
\overline{S} &\defeq ( R / I ) [t] / t^{m+n} (t+p) ( R / I )[t]
\end{align*}

Set $ \mathcal{R} := \mathcal{O} \otimes_{\Z_p} R $ and let $ \mathcal{I} $ be the extension of the ideal $ I $ in $ \mathcal{R} $. Set
\begin{equation*}
\mathcal{S} \defeq \mathcal{R}[t] / t^{m+n} (t+p) \mathcal{R}[t]
\end{equation*}
and
\begin{equation*}
\overline{\mathcal{S}} \defeq ( \mathcal{R} / \mathcal{I} ) [t] / t^{m+n} (t+p) ( \mathcal{R} / \mathcal{I} )[t]
\end{equation*}
Let $ I_t $ be the extension of $ I $ in $ S $ and let $ \mathcal{I}_t $ be the extension of $ \mathcal{I} $ in $ \mathcal{S} $. I use without warning the equalities
\begin{align*}
\overline{S} &= S / I_t \\
\mathcal{S} &= \mathcal{O} \otimes_{\Z_p} S \\
\overline{\mathcal{S}} &= \mathcal{O} \otimes_{\Z_p} \overline{S} = \mathcal{S} / \mathcal{I}_t
\end{align*}

Let $ \overline{g} \in \biggergroup{m}{n} ( R / I ) $ be arbitrary. By definition, $ \overline{g} $ is an $ \overline{\mathcal{S}} [t] $-linear automorphism of $ \overline{\mathcal{V}}_{\textup{sup}} ( R / I ) $. I make the obvious $ \mathcal{R}[t] $-linear identification
\begin{equation} \label{EQidentifyV0modI}
\overline{\mathcal{V}}_{\textup{sup}} ( R / I ) = \overline{\mathcal{S}}^d
\end{equation}
so that $ \overline{g} $ is identified with an automorphism of $ \overline{\mathcal{S}}^d $ and the hermitian form $
\overline{\phi}_{R/I} $ used in the definition (page \pageref{EhermitianformusedforJ}) of $ \biggergroup{m}{n} ( R / I ) $ is identified with the standard hermitian form $ \overline{\mathcal{S}}^d \times \overline{\mathcal{S}}^d \rightarrow \overline{\mathcal{S}} $ defined by the anti-identity matrix $ \antiid $.

To construct a lift $ g \in \biggergroup{m}{n} (R) $ of $ \overline{g} $, I use the following point of view: having $ g $ is equivalent to having $ v_1, \ldots, v_d \in \mathcal{S}^d $ such that
\begin{itemize}
\setlength{\itemsep}{5pt}
\item $ v_1, \ldots, v_d $ is an $ \mathcal{S} $-module basis of $ \mathcal{S}^d $
\item $ v_i \in \mathcal{S}^i \oplus (t+p) \mathcal{S}^{d-i} $ for
all $ i $.
\item there is $ c \in S^{\times} $ such that
\begin{equation*}
\overline{\phi}_R ( v_i, v_j ) =
\begin{cases}
c & \hspace{5pt} i + j = d+1 \\
0 & \hspace{5pt} i + j \neq d+1
\end{cases}
\end{equation*}
\end{itemize}
The link between the two points of view is
\begin{align*}
v_i &= g( e_i ) \\
c &= c(g)
\end{align*}
Note that it is automatic from $ R[t] $-linearity that $ g $ stabilizes each $ \mathcal{S}^i \oplus (t+p) \mathcal{S}^{d-i} $.

To simplify notation further, set
\begin{align*}
M &{:=} \mathcal{S}^d \\
\sigma &{:=} (t+p) \\
N_i &{:=} \mathcal{S}^i \oplus \sigma \mathcal{S}^{d-i} \textaftermath{($ i = 0, \ldots, d $)}
\end{align*}
Let $ \overline{M} $ denote $ M \otimes_\mathcal{S} ( \mathcal{S} / \mathcal{I}_t ) = M / \mathcal{I}_t M $. Let $ \overline{N}_i $ be the image of $ N_i $ in $ \overline{M} $.

First, note that lifting is easy if the similitude condition is not involved:
\begin{lemma} \label{Leasysmoothess}
If $ \overline{v}_1, \ldots, \overline{v}_d $ is an $ (\mathcal{S} / \mathcal{I}_t) $-module basis for $ \overline{M} $ such that $ \overline{v}_i \in \overline{N}_i $ for all $ i $, then there exists an $ \mathcal{S} $-module basis $ v_1, \ldots, v_d $ for $ M $ such that such that $ v_i \in N_i $ and
\begin{equation*}
\overline{v}_1 \equiv v_1 \textup{ mod } \mathcal{I}_t M.
\end{equation*}
\end{lemma}

\begin{proof}
Let $ v_i \in N_i $ be \emph{arbitrary} lifts of $ \overline{v}_i $. Nakayama's Lemma implies that since $ \overline{v}_1, \ldots, \overline{v}_d $ generates $ \overline{M} $, the set $ v_1, \ldots, v_d $ generates $ M $ (to apply Nakayama's Lemma, note that $ \mathcal{I}_t^2 = 0 $ ). By ``Linear Independence of Minimal Generating Sets'', $ v_1, \ldots, v_d $ must be a basis.
\end{proof}

Now I extend this lemma to handle the similitude condition:
\begin{prop} \label{PsmoothJ}
Let the notation be as above. Let $ e_1, \ldots, e_d $ be the standard basis for $ \mathcal{S}^d $ and set $ \overline{v}_i := \overline{g} ( \overline{e}_i ) $. \textbf{Assertion}: There is $ c \in S^{\times} $ and $ w_i \in N_i $ ($ i = 1, \ldots, d $) that form an $ \mathcal{S} $-module basis for $ M $ and that satisfy
\begin{itemize}
\setlength{\itemsep}{5pt}
\item $ w_1 \equiv \overline{v}_1 \textup{ mod } \mathcal{I}_t M $, and

\item $ \overline{\phi}_R ( w_i, w_j ) =
\begin{cases}
c & \hspace{5pt} i + j = d+1 \\
0 & \hspace{5pt} i + j \neq d+1
\end{cases} $
\end{itemize}
\end{prop}

By the previous discussion, Proposition \ref{PsmoothJ} implies that $ \biggergroup{m}{n} $ is smooth.

\begin{proof}
Let $ \jmath $ denote the involution of $ \mathcal{S} $ induced by the non-trivial element of $ \gal ( F / \Q_p ) $. Note that the ideal $ \mathcal{I}_t $ is $ \jmath $-stable because it was extended from $ I \subset R $. Let $ v_1, \ldots, v_d $ be the basis guaranteed by Lemma \ref{Leasysmoothess} and choose a representative $ c \in S^{\times} $ of $ c ( \overline{g} ) $ such that $ \jmath( c ) = c $ (this is possible because $ c ( \overline{g} ) \in \overline{S}^{\times} $ by definition).

By assumption, the similitude condition holds modulo $ \mathcal{I} $, i.e.
\begin{equation*}
\overline{\phi}_{R/I} ( \overline{v}_i, \overline{v}_j ) =
\begin{cases}
c ( \overline{g} ) & \hspace{5pt} i + j = d+1 \\
0 & \hspace{5pt} i + j \neq d+1
\end{cases}
\end{equation*}
More succinctly,
\begin{equation*}
\overline{\phi}_{R/I} ( \overline{v}_i, \overline{v}_j ) =
c ( \overline{g} ) \delta_{ i, d+1-j }
\end{equation*}
where $ \delta $ is the Kronecker delta. This means that there are $ x_{i,j} \in \mathcal{I}_t $ such that
\begin{equation} \label{EQlocallabel6}
\overline{\phi}_R ( v_i, v_j ) = c \delta_{i,d+1-j} + x_{i,j}
\end{equation}
Since $ c $ is independent of $ i, j $ and since $ \jmath ( c ) = c $, it is true that $ x_{j,i} = \jmath ( x_{i,j} ) $:
\begin{align*}
x_{j,i}
&= \overline{\phi}_R ( v_j, v_i ) - c \delta_{j,d+1-i} \\
\text{(because $ \overline{\phi} $ is hermitian)} &= \jmath ( \overline{\phi}_R ( v_i, v_j ) ) - c \delta_{j,d+1-i} \\
\text{(because $ \delta_{i,d+1-j} = \delta_{j,d+1-i} $)} &= \jmath ( \overline{\phi}_R ( v_i, v_j ) ) - c \delta_{i,d+1-j} \\
\text{(because $ \jmath ( c ) = c $)} &= \jmath ( x_{i,j} )
\end{align*}

By bi-additivity,
\begin{equation*}
\overline{\phi}_R ( v_i + m_i, v_j + m_j ) = \overline{\phi}_R (v_i, v_j) + \overline{\phi}_R (m_i, v_j) + \overline{\phi}_R ( v_i, m_j ) + \overline{\phi}_R ( m_i, m_j )
\end{equation*}
Because of this and the equality $ x_{j,i} = \jmath ( x_{i,j} ) $, it suffices to find $ m_1, \ldots, m_d \in \mathcal{I}_t M $ such that $ m_i \in N_i $ and $ \overline{\phi}_R (m_i, v_j) = - \frac{1}{2} x_{i,j} $ and then to take $ w_i \defeq v_i + m_i $, since then
\begin{equation*}
\overline{\phi}_R ( w_i, w_j ) = c \delta_{i,d+1-j} + x_{i,j} - \frac{1}{2} x_{i,j} - - \frac{1}{2} \jmath ( x_{j,i} ) = c \delta_{i,d+1-j}
\end{equation*}
(note that $ \mathcal{I}_t^2 = 0 $ implies $ \overline{\phi}_R ( m_i, m_j ) = 0 $).

\begin{remark}
Here I have used the assumption that $ p \neq 2 $.
\end{remark}

Note that
\begin{equation} \label{EQlocallabel5}
N_{d-i} = \{ m \in M \suchthat \overline{\phi} ( m, N_i ) \subset \sigma \mathcal{S} \}
\end{equation}

Fix $ i $. Consider the $ \mathcal{S} $-linear functional $ M \rightarrow \mathcal{S} $ defined by $ v_j \mapsto - \frac{1}{2} x_{i,j} $. Since $ \overline{\phi}_R $ is perfect, this functional is $ \overline{\phi}_R ( m_i, - ) $ for some $ m_i \in M $. In fact, $ m_i \in \mathcal{I}_t M $ since $ x_{i,j} \in \mathcal{I}_t $. I claim that this functional automatically sends $ N_{d-i} $ into $ \sigma \mathcal{S} $. It then follows from inclusion ``$ \supset $'' of duality (\ref{EQlocallabel5}) that $ m_i \in N_i $, and the proof will be finished.

Since
\begin{equation*}
v_1, \ldots, v_{d-i}, \sigma v_{d-i+1}, \ldots, \sigma v_d
\end{equation*}
is an $ \mathcal{S} $-linear basis for $ N_{d-i} $, it suffices to show that
\begin{equation*}
\overline{\phi}_R ( m_i, v_j ) = - \frac{1}{2} x_{i,j} \in \sigma \mathcal{S}
\end{equation*}
for the subset $ 1 \leq j \leq d-i $ of indices. This inequality implies that $ i \neq d+1-j $ and the defining relation (\ref{EQlocallabel6}) gives
\begin{equation*}
\overline{\phi}_R ( m_i, v_j ) = - \frac{1}{2} \overline{\phi}_R ( v_i, v_j )
\end{equation*}
So it suffices to show that $ \overline{\phi}_R ( v_i, v_j ) \in \sigma \mathcal{S} $. But this is just inclusion ``$ \subset $'' of duality (\ref{EQlocallabel5}).
\end{proof}

\begin{remark}
This proof is a variant of Proposition A.13 from \cite{RZ} extended to handle polarized chains.
\end{remark}

\subsection{Conventions for Weyl groups, cocharacters, etc.} \label{SSweylconventions}

Consider the unitary similitude group $ \GU_d $ associated to the quadratic extension $ \F((t)) / \F_p((t)) $, the vector space $ \F((t))^d $, and the standard hermitian form defined by the anti-identity matrix $ \antiid $. Let $ A $ be the usual maximal $ \F_p((t)) $-split diagonal torus of $ \GU_d $. Since $ \GU_d $ is quasi-split, the centralizer $ C_{\GU_d} ( A ) $ is a maximal torus $ T $ (also consisting of diagonal elements) and is defined over $ \F_p((t)) $. The \emph{relative extended affine Weyl group} of $ \GU_d $ with respect to $ A $ is the quotient
\begin{equation*}
\widetilde{W} \defeq N ( \F_p((t)) ) / T ( \F_p((t)) )_0
\end{equation*}
where $ N $ is the normalizer $ N_{\GU_d} ( A ) $ and $ T ( \F_p((t)) )_0 $ is the unique maximal compact open subgroup of $ T ( \F_p((t)) ) $.

Let $ W = N / T $ be the \emph{relative finite Weyl group} of $ G $ with respect to $ A $ and $ X_{*} ( A ) $ the abelian group of algebraic group homomorphisms $ A \rightarrow \Gm $. The extended affine Weyl group $ \widetilde{W} $ has a semidirect product decomposition
\begin{equation*}
\widetilde{W} = X_{*} ( A ) \rtimes W
\end{equation*}
and parametrizes the double cosets of $ \GU_d ( \F_p((t)) ) $ modulo an Iwahori subgroup (this parametrization is called the \emph{Bruhat-Tits Decomposition}). Implicit in this parametrization is the fact that elements of $ \widetilde{W} $ can be represented by $ \F_p((t)) $-points, and therefore $ \widetilde{W} $ can be considered as a subset (usually not a subgroup) in many different ways of $ \GU_d ( \F_p((t)) ) $. In more detail, elements of $ W $ can be represented by elements of $ \GU_d ( \F_p((t)) ) $, and I consider $ X_{*} ( A ) $ also as a subset of $ A ( \F_p((t)) ) \subset \GU_d ( \F_p((t)) ) $ via the map $ \lambda \mapsto \lambda ( t ) $. I fix such an inclusion
\begin{equation*}
\widetilde{W} \hookrightarrow \GU_d ( \F_p((t)) )
\end{equation*}
and use it without warning from now on.

%\begin{remark}In general, the group which parametrizes the Bruhat-Tits decomposition relative to an Iwahori subgroup is the \emph{Iwahori-Weyl group}, defined completely generally in \cite{HRa} using the Kottwitz homomorphism from \cite{Ko2}; see Definition 7, Proposition 8, and Remark 9 of \cite{HRa}. But since $ \GU_d $ here is \emph{unramified}, the Iwahori-Weyl group coincides with what is usually called the extended affine Weyl group and the Kottwitz homomorphism also has a more direct definition; see Remark 10 of \cite{HRa}, and Lemma 3.0.1(III), Corollary 11.1.2(c), and Proposition 11.1.4 of \cite{HRo}.\end{remark}

Finally, I denote by $ \Phi_{\aff} $ the affine root system for $ \GU_d $. Let $ W_{\aff} \subset \widetilde{W} $ be the subgroup generated by reflections across the kernels of elements of $ \Phi_{\aff} $. Fix a Chevalley-Bruhat partial order $ \leq $ and length function $ \ell $ on $ W_{\aff} $, which I require to be consistent with the Iwahori subgroup defined in the next subsection. I consider $ \ell $ on $ \widetilde{W} $ by extending trivially.

Similar conventions are in place for the unitary similitude group associated to the (unramified) quadratic extension $ F((t)) / \Q_p((t)) $.

\subsection{Description of the Iwahori subgroup} \label{SSiwahoridesc}

Return to the setup of \S\ref{SSaffineflagdesc} (page \pageref{SSaffineflagdesc}), i.e. denote by $ \GU_d $ the unitary similitude group associated to the (unramified) quadratic extension $ \F((t)) / \F_p((t)) $, the vector space $ \F((t))^d $, and the standard hermitian form $ \Phi $ defined by the anti-identity matrix $ \antiid $. Recall from \S\ref{SSspecialfibersexhaust} the chain $ \lambda_{\bullet} $ of $ \F[[t]] $-lattices in $ \F((t))^d $ and notice that
\begin{equation*}
\lambda_i \otimes_{\F_p[[t]]} \F_p = \Lambda_i \otimes_{\Z_p} \F_p
\end{equation*}

\begin{definition}
Denote by 
\begin{equation*}
\iwahori : \catalgebras{\F_p} \longrightarrow \catgroups
\end{equation*}
the positive loop group of the Iwahori subgroup, over $ \F_p[[t]] $, of the unitary similitude group on $ ( \F((t))^d, \Phi ) $ corresponding to the lattice chain $ \lambda_{\bullet} $.
\end{definition}
In particular, $ \iwahori ( \F_p ) = \{ g \in \GU_d ( \F_p((t)) ) \suchthat g ( \lambda_i ) = \lambda_i \textinsidemath{and} \kappa ( g ) = 1 \} $ where $ \kappa $ denotes the ``Kottwitz homomorphism'' of $ \GU_d $. Note that if $ g \in \iwahori ( \F_p ) $ then $ c(g) \in \F_p[[t]]^{\times} $ since $ g $ and $ g^{-1} $ both stabilize $ \lambda_0 = \F[[t]]^d $.

I claim that there is a group homomorphism
\begin{equation} \label{Eiwahoritobiggergroup}
\iwahori ( \F_p ) \longrightarrow \biggergroup{m}{n} ( \F_p )
\end{equation}
such that acting by $ \iwahori ( \F_p ) $ on the image of the embedding
\begin{equation*}
\biggermodel{m}{n} ( \F_p ) \hookrightarrow \affineflagvariety ( \F_p )
\end{equation*}
is the same as acting directly on $ \biggermodel{m}{n} ( \F_p ) $ via (\ref{Eiwahoritobiggergroup}).

If $ g \in \iwahori ( \F_p ) $ then by definition $ g $ descends to an $ \F [t] $-linear automorphism $ \overline{g} $ of the quotient
\begin{equation*}
t^{-(m+1)} \F [[t]]^d / t^n \F [[t]]^d = t^{-(m+1)} \F [t]^d / t^n \F [t]^d = \overline{\mathcal{V}}_{\textup{sup}}( \F_p )
\end{equation*}
and stabilizes the subquotients
\begin{equation*}
\lambda_i / \mathcal{V}_{\textup{inf}}(\F_p) = \overline{\mathcal{V}}_i(\F_p)
\end{equation*}
Modulo $ t^{m+n+1} $, the element $ c(g) $ becomes a multiplier $ c( \overline{g} ) $ as in the definition of $ \biggergroup{m}{n} ( \F_p ) $. This shows that $ g \mapsto \overline{g} $ is a function $ \iwahori ( \F_p ) \rightarrow \biggergroup{m}{n} ( \F_p ) $. It is clear that this is a group homomorphism and that the actions of $ \biggergroup{m}{n} ( \F_p ) $ on $ \biggermodel{m}{n} ( \F_p ) $ and $ \iwahori ( \F_p ) $ on $ \affineflagvariety ( \F_p ) $ are compatible.

\subsection{A ``Bruhat-Tits Decomposition'' of the local model} \label{SSbruhattitsdecomp}

Let $ \widetilde{W} $ be the Iwahori-Weyl group from \S\ref{SSweylconventions} (page \pageref{SSweylconventions}) and recall the notation and conventions there. For each $ w \in \widetilde{W} $, there is the affine Schubert cell $ C_w \defeq \iwahori ( \F_p ) \cdot w \cdot \iwahori ( \F_p ) $. These Schubert cells are the $ \F_p $-points of certain $ \F_p $-schemes, which will also be called Schubert cells: define the functor
\begin{equation*}
C_w : \catalgebras{\F_p} \rightarrow \catsets
\end{equation*}
to be the fpqc-sheafification of the functor that assigns to any commutative $ \F_p $-algebra $ R $ the $ \iwahori ( R ) $-orbit of $ w_R $ in $ \GU_d ( R((t)) ) / \iwahori ( R ) $ (here $ w_R $ denotes the image of $ w $ under $ \GU_d ( \F_p((t)) ) \rightarrow \GU_d ( R((t)) ) $).

Because of (\ref{Eiwahoritobiggergroup}), the subset $ \biggermodel{m}{n} ( \F_p ) \subset \affineflagvariety ( \F_p ) $ is $ \iwahori ( \F_p ) $-stable and there is a Bruhat-Tits Decomposition of $ \biggermodel{m}{n} ( \F_p ) $:
\begin{equation*}
\biggermodel{m}{n} ( \F_p ) = \coprod_{w \in \widetilde{W}^{(m,n)}} C_w ( \F_p )
\end{equation*}
for a certain finite subset $ \widetilde{W}^{(m,n)} \subset \widetilde{W} $.

\begin{prop} \label{Padmissiblesetlocalmodel}
In the initial case $ \biggermodel{0}{1}_{\F_p} = \localmodel_{\F_p} $, the Bruhat-Tits Decomposition consists of the cells $ C_w $ for all $ w \in \admissible ( \mu ) $.
\end{prop}

Roughly speaking, this follows from Theorem 4.5 of \cite{KR} because the relevant combinatorial formalism of the unramified unitary similitude group $ \GU_d $ is the same as that of the symplectic similitude group $ \GSp_d $ defined by the antisymmetric analogue of $ \antiid $. The following proof simply provides a few details.

\begin{proof}
Fix $ w \in \widetilde{W} $ and suppose that $ C_w \subset \localmodel_{\F_p} $. Transforming the standard flag $ t \cdot \lambda_i $ by a nice representative of $ w $, one sees that there are integers $ n_i(j) $ such that the image of $ \localmodel ( \F_p ) $ in $ \affineflagvariety ( \F_p ) $ contains the ``split'' flag $ \mathcal{F}_{\bullet} $ defined by $ \mathcal{F}_i = \bigoplus_{j = 1}^{d} t^{n_i(j)} \F [[t]] $. Denote by $ \overline{n}_i \in \Z^d $ the tuple $ ( n_i(1), n_i(2), \ldots, n_i(d) ) $ of exponents. I claim that $ \overline{n} = ( \overline{n}_{2d}, \ldots, \overline{n}_{d+2}, \overline{n}_{d+1} ) \in \bigoplus^d \Z^d $ is a \emph{minuscule $ \GSp_d $-alcove of size $ d/2 $}, in the sense of \cite{KR} (the indices are chosen for compatibility with the notation of \cite{KR}). That $ \overline{n} $ is an alcove follows from the fact that $ \mathcal{F}_{\bullet} $ is a periodic chain: $ n_i(j) \leq n_{i^{\prime}}(j) $ for all $ j $ whenever $ i \geq i^{\prime} $ by \AFVdef{\ref{AFVchain}}, $ \overline{n}_{i+d} = \overline{n}_i + ( 1, 1, \ldots, 1 ) $ for all $ i $ by \AFVdef{\ref{AFVperiodic}}. That $ \overline{n} $ is a $ \GSp_d $-alcove follows immediately from \AFVdef{\ref{AFVperiodic}} and \AFVdef{\ref{AFVsimduality}} because of the choice of Gram matrix for $ \Phi $. That $ \overline{n} $ has size $ d/2 $ follows from \BLMdef{\ref{BLMlocalrank}} (note that $ \dim_{\F_p} = 2 \cdot \dim_{\F} $). Finally, if $ \omega_i \in \Z^d $ denotes the tuple which has entry $ 0 $ in the first $ i $ positions and entry $ 1 $ in all other positions, then \BLMdef{\ref{BLMbounds}} implies that $ \omega_i + ( 1, 1, \ldots, 1 ) \geq \overline{n}_{2d-i} \geq \omega_i $ for all $ 0 \leq i < d $, i.e. that $ \overline{n} $ is minuscule. The matrix representation of $ \GU_d $ inside $ \GL_d \otimes_{\Q_p} F $ yields an embedding of $ \widetilde{W} $ inside $ \Z^d \rtimes S_d $, and the various choices made in \S\ref{SSweylconventions} (Gram matrix of $ \phi $, maximal $ F $-split torus, Iwahori subgroup, etc.) ensure that the image of this $ \widetilde{W} \hookrightarrow \Z^d \rtimes S_d $ is identical in all the relevant ways (Bruhat-Chevalley order, semidirect-product decompositions, etc.) to the extended affine Weyl group of $ \GSp_d $ described in \S4.1 of \cite{KR}. Therefore, Theorem 4.5(1) of \cite{KR} implies that $ w \in \admissible ( \mu ) $. Conversely, it is immediate that $ C_{\lambda} \subset \localmodel_{\F_p} $ for all $ \lambda \in W \cdot \mu $ and since $ \localmodel_{\F_p} \hookrightarrow \affineflagvariety_{\F_p} $ is a closed embedding, it follows that $ C_w \subset \localmodel_{\F_p} $ for all $ w \in \admissible ( \mu ) $.
\end{proof}

\subsection{A ``Cartan Decomposition'' of the local model} \label{SScartandecomp}

\begin{definition}
Denote by 
\begin{equation*}
\specialparahoric : \catalgebras{\Q_p} \longrightarrow \catgroups
\end{equation*}
the positive loop group $ R \mapsto \GU_d ( R[[t]] ) $.
\end{definition}

As in \S\ref{SSbruhattitsdecomp} (page \pageref{SSbruhattitsdecomp}), I claim that there is a group homomorphism
\begin{equation} \label{Especialparahorictobiggergroup}
\specialparahoric ( \Q_p ) \longrightarrow \biggergroup{m}{n} ( \Q_p )
\end{equation}
such that acting by $ \specialparahoric ( \Q_p ) $ on the image of the embedding
\begin{equation*}
\biggermodel{m}{n} ( \Q_p ) \hookrightarrow \affinegrassmannian ( \Q_p )
\end{equation*}
is the same as acting directly on $ \biggermodel{m}{n} ( \Q_p ) $ via (\ref{Especialparahorictobiggergroup}). The group homomorphism is defined in the same way: any $ g \in \specialparahoric ( \Q_p ) $ restricts to an automorphism $ \overline{g} $ of $ \overline{\mathcal{V}}_{\textup{sup}} ( \Q_p ) $ which stabilizes the subquotient $ \lambda_0 / \mathcal{V}_{\textup{inf}} ( \Q_p ) = \overline{\mathcal{V}}_0 ( \Q_p ) $ and, because $ t+p $ is a unit in $ \Q_p[[t]] $, also stabilizes the other $ \overline{\mathcal{V}}_i ( \Q_p ) $. The automorphism $ \overline{g} $ automatically satisfies the similitude condition necessary for membership in $ \biggergroup{m}{n} ( \Q_p ) $.

Recall the notation and conventions in \S\ref{SSweylconventions} (page \pageref{SSweylconventions}). Let $ O_{\lambda} $ denote the $ \Q_p $-subschemes of $ \affinegrassmannian $ forming the Cartan Decomposition. Similar to \S\ref{SSbruhattitsdecomp}, (\ref{Especialparahorictobiggergroup}) implies that that $ \biggergroup{m}{n} ( \Q_p ) $ is a $ \specialparahoric ( \Q_p ) $-stable subset of $ \affinegrassmannian ( \Q_p ) $, so there is a Cartan Decomposition:
\begin{equation*}
\biggermodel{m}{n} ( \Q_p ) = \coprod_{ \lambda \in X_{*}^{(m,n)} } O_{\lambda} ( \Q_p )
\end{equation*}
for a certain finite set $ X_{*}^{(m,n)} $ of (dominant) cocharacters $ \lambda \in X_{*} ( A ) $.

\begin{remark}
In some sense, the objects $ \affinegrassmannian $ and $ \specialparahoric $ are merely conceptual aids and one could instead work in the more agreeable world of finite-type schemes by using only $ \biggermodel{m}{n}_{\Q_p} $ and $ \biggergroup{m}{n}_{\Q_p} $ for suitable $ m, n \in \N $. For example, the Cartan cells could be defined as $ \biggergroup{m}{n}_{\Q_p} $-orbits in $ \biggermodel{m}{n}_{\Q_p} $.
\end{remark}

\begin{prop}
In the initial case $ \biggermodel{0}{1}_{\Q_p} = \localmodel_{\Q_p} $, the Cartan Decomposition consists of the single cell $ O_{\mu} $.
\end{prop}

\begin{proof}
Set $ \delta := d/2 $ and $ \mathcal{F}_0^{\std} := t \cdot F[[t]]^d $. Fix $ \lambda \in X_{*} ( A ) $ and suppose that $ O_{\lambda} \subset \localmodel_{\Q_p} $. Then the image of $ \localmodel ( \Q_p ) $ in $ \affinegrassmannian ( \Q_p ) $ contains the ``normal form'' lattice $ \mathcal{F}_0 := \lambda ( t ) \cdot \mathcal{F}_0^{\std} = \bigoplus_{i = 1}^{\delta} t^{1 + n_i} F[[t]] \oplus \bigoplus_{i = \delta}^{1} t^{1 + k - n_i} F[[t]] $, where $ n_1, \ldots, n_{\delta}, k \in \Z $ are the standard coordinates of $ \lambda $. The fact that $ \mathcal{F}_0 $ is in the image of $ \localmodel ( \Q_p ) \hookrightarrow \affinegrassmannian ( \Q_p ) $ (more specifically, properties \BLMdef{\ref{BLMbounds}} and \BLMdef{\ref{BLMlocalrank}}) means that the integers $ n_i, k $ satisfy the following two properties: (1) $ 0 \leq \min_i ( 1 + n_i, 1 + k - n_i ) \leq \max_i ( 1 + n_i, 1 + k - n_i ) \leq 1 $ and (2) $ \dim_{\Q_p} ( \mathcal{F}_0 / \mathcal{F}_0^{\std} ) = d $. Property (2) implies that $ k = -1 $. Property (1) then implies that $ -m_i, -n_i \in \{ 0, 1 \} $ for all $ i $, and such $ \lambda $ are all $ W $-conjugate to $ \mu $, as desired.
\end{proof}

\section{The trace function} \label{Stracefunction}

\subsection{Generalities on nearby cycles and trace functions}

Let $ X $ be a separated finite-type $ \Z_p $-scheme. Let $ \mathcal{C}^{\bullet} $ be a complex of \etale $ \ell $-adic sheaves on $ X_{\Q_p} $. The pullback $ \overline{\mathcal{C}}^{\bullet} $ of $ \mathcal{C}^{\bullet} $ along $ X_{ \overline{\Q}_p } \rightarrow X_{ \Q_p } $ has a natural continuous action by $ \Gamma \defeq \gal ( \overline{\Q}_p / \Q_p ) $ i.e. a collection for all $ \gamma \in \Gamma $ of functor isomorphisms $ \gamma_{*} ( \overline{\mathcal{C}}^{\bullet} ) \directedisom \overline{\mathcal{C}}^{\bullet} $ consistent with composition (by abuse of notation, $ \gamma $ also denotes the induced $ \Q_p $-scheme automorphism of $ X_{ \overline{\Q}_p } $). Define
\begin{equation*}
\nearbycycles ( \mathcal{C}^{\bullet} ) \defeq \imath^{*} ( \textup{R}\jmath_{*} ( \overline{\mathcal{C}}^{\bullet} ) ),
\end{equation*}
the \emph{nearby cycles} complex of \etale $ \ell $-adic sheaves on $ X_{ \overline{\F}_p } $, where $ X_{ \overline{\F}_p } \stackrel{ \imath }{ \longrightarrow } X_{ \overline{\Z}_p } \stackrel{ \jmath }{ \longleftarrow } X_{ \overline{\Q}_p } $ are the structure morphisms. The complex $ \nearbycycles ( \mathcal{C}^{\bullet} ) $ on $ X_{ \overline{\F}_p } $ inherits the action by $ \Gamma $. In particular, each $ \gamma \in \Gamma $ induces an endomorphism of the cohomology stalk $ \mathcal{H}^i ( \nearbycycles ( \mathcal{C}^{\bullet} ) )_x $ for every $ x \in X ( \F_p ) $ and $ i \in \Z $.

Grothendieck's quasi-unipotent inertia theorem applies to the continuous representation of $ \Gamma $ on the finite-dimensional $ \overline{\Q}_{\ell} $-vector space $ \mathcal{H}^i ( \nearbycycles ( \mathcal{C}^{\bullet} ) )_x $ to yield ``semisimplifications''
\begin{equation*}
\textup{ss} ( \mathcal{H}^i ( \nearbycycles ( \mathcal{C}^{\bullet} ) )_x )
\end{equation*}
on which the inertia subgroup $ \Gamma_0 $ via a \emph{finite} quotient ($ \Gamma $ acts on the semisimplification by acting individually on each summand). See \S3 of \cite{HN} for a detailed discussion of these semisimplifications.

The action of $ \Gamma $ on $ \textup{ss} ( \mathcal{H}^i ( \nearbycycles ( \mathcal{C}^{\bullet} ) )_x )^{\Gamma_0} $ factors through $ \gal ( \overline{\F}_p / \F_p ) $, and one defines
\begin{equation*}
\sstrace{ \nearbycycles ( \mathcal{C}^{\bullet} ) } ( x ) \defeq \sum_i (-1)^i \trace ( \frob ; \textup{ss} ( \mathcal{H}^i ( \nearbycycles ( \mathcal{C}^{\bullet} ) )_x )^{\Gamma_0} )
\end{equation*}
for all $ x \in X ( \F_p ) $. The exactness of the fixed-points functor $ V \mapsto V^G $ for a \emph{finite} group $ G $ makes this function more well-behaved; see for example the proof of Lemma 10 in \cite{HN}.

More generally, define
\begin{equation*}
\sstrace{ \mathcal{C}^{\bullet} } : X ( \F_p ) \longrightarrow \overline{\Q}_{\ell}
\end{equation*}
for any complex of \etale $ \ell $-adic sheaves $ \mathcal{C}^{\bullet} $ on $ X_{ \overline{\F}_p } $ which has an action by $ \gal ( \overline{\Q}_p / \Q_p ) $ consistent with the action of $ \gal ( \overline{\F}_p / \F_p ) $ on $ X_{ \overline{\F}_p } $. For example, $ \mathcal{C}^{\bullet} $ could be the pullback to $ X_{\overline{\F}_p} $ of a complex on $ X_{\F_p} $ with $ \gal ( \overline{\Q}_p / \Q_p ) $ acting via the composition $ \gal ( \overline{\Q}_p / \Q_p ) \rightarrow \gal ( \unr{\Q}_p / \Q_p ) \rightarrow \gal ( \overline{\F}_p / \F_p ) $ (in which case the semisimplification is unnecessary).

\subsection{Definition of the main trace functions} \label{SStracedef}

Fix $ m, n \in \N $. Let $ O_{\lambda} $ be a cell from the Cartan Decomposition of $ \biggermodel{m}{n}_{ \Q_p } $. Let $ \IC{ \lambda } $ be the (perverse) \etale $ \ell $-adic intersection complex on $ \overline{O}_{\lambda} $ (the reduced closure). Then applying the construction from the previous subsection to this special case yields
\begin{equation*}
\sstrace{ \lambda } \defeq \sstrace{ \nearbycycles ( \IC{ \lambda } ) } : \biggermodel{m}{n} ( \F_p ) \longrightarrow \overline{\Q}_{\ell}
\end{equation*}
By the embedding in \S\ref{SSaffineflagembedding} (page \pageref{SSaffineflagembedding}), I can extend by $ 0 $ and consider $ \sstrace{\lambda} $ to be a function on $ \affineflagvariety ( \F_p ) $.

\subsection{The trace functions are Iwahori-invariant} \label{SStracefunctionisiwahoriinvariant}

In order to show that $ \sstrace{\lambda} $ is in the Iwahori-Hecke algebra $ \mathcal{H} $ of $ \GU ( \F((t))^d, \Phi ) ( \F_p((t)) ) $ with respect to $ \iwahori ( \F_p[[t]] ) $, I must show that it is invariant under left-translations by $ \iwahori ( \F_p[[t]] ) $ (invariance under right-translations is automatic from the domain of $ \affineflagvariety ( \F_p ) $). Because of the group homomorphism $ \iwahori ( \F_p[[t]] ) \rightarrow \biggergroup{m}{n} ( \F_p ) $ (see \S\ref{SSbruhattitsdecomp} (page \pageref{SSbruhattitsdecomp})) and the definition of $ \sstrace{\lambda} $, it suffices to show that $ \nearbycycles ( \IC{ \lambda } ) $ is $ \biggergroup{m}{n}_{ \overline{\F}_p } $-equivariant, in the sense that there is an isomorphism
\begin{equation} \label{EQgoalequivariance}
\ac_{\overline{\F}_p}^{*} ( \nearbycycles ( \IC{ \lambda } ) ) \directedisom
\pr_{\overline{\F}_p}^{*} ( \nearbycycles ( \IC{ \lambda } ) )
\end{equation}
of \etale sheaf complexes subject to a ``cocycle'' (group action axiom) condition. Here
\begin{equation*}
\ac, \pr : \biggergroup{m}{n} \times_{\spec(\Z_p)}
\biggermodel{m}{n} \rightarrow \biggermodel{m}{n}
\end{equation*}
are the left-action (see \S\ref{LJaction}) and projection morphisms.

By Proposition \ref{PsmoothJ} (page \pageref{PsmoothJ}), the morphism $ \biggergroup{m}{n} \rightarrow \spec(\Z_p) $ is smooth, so the projection $ \pr $, which is simply the morphism supplied by the fiber product, is also smooth (since smoothness is preserved under base-change). It follows from ``smooth base change'' (the fact that pullback by a smooth morphism commutes with (derived) pushforward in a base-change diagram), that
\begin{equation} \label{EQprbasechange}
\pr_{\overline{\F}_p}^{*} ( \nearbycycles ( \IC{ \lambda } ) ) \cong \nearbycycles ( \pr_{\Q_p}^{*} ( \IC{ \lambda } ) )
\end{equation}

On the other hand, the endomorphism of the functor $ \biggergroup{m}{n} \times \biggermodel{m}{n} $ defined by $ ( g, x ) \mapsto ( g, g(x) ) $ is an \emph{automorphism} (over $ \Z_p $). Since the action morphism $ \ac $ is the composition of this automorphism with the (smooth) projection $ \pr $, this shows that the action morphism $ \ac $ is smooth, and so by the same reasoning as for $ \pr $,
\begin{equation} \label{EQacbasechange}
\ac_{\overline{\F}_p}^{*} ( \nearbycycles ( \IC{ \lambda } ) ) \cong \nearbycycles ( \ac_{\Q_p}^{*} ( \IC{ \lambda } ) )
\end{equation}

The intersection complex $ \IC{ \lambda } $ is $ \biggergroup{m}{n}_{\Q_p} $-equivariant by definition so combining (\ref{EQprbasechange}) and (\ref{EQacbasechange}) yields (\ref{EQgoalequivariance}).

\subsection{Statement of theorem}

By the previous subsection, $ \sstrace{\mu} $ is identified with an element of the Iwahori-Hecke algebra $ \mathcal{H} $.
\begin{maintheorem}
$ \sstrace{\mu} \in Z ( \mathcal{H} ) $.
\end{maintheorem}
The remainder of the paper develops, following \cite{HN}, the tools needed to prove this theorem, and the proof of theorem itself occurs in \S\ref{Smainproof} (page \pageref{Smainproof}).

\begin{corollary}[of the Main Theorem]
$ \sstrace{\mu} = (-1)^{ \ell ( \mu ) } q(\mu)^{\frac{1}{2}} z_{\mu} $, where $ z_{\mu} $ denotes the usual Bernstein basis function attached to $ \mu $.
\end{corollary}

The value $ q ( \mu ) $ is defined as the index $ [ \iwahori \mu \iwahori : \iwahori ] $. The sign $ (-1)^{ \ell ( \mu ) } $ is due to the shift by $ - \dim ( O_{\mu} ) = - \ell ( \mu ) $ imposed on the intersection complexes to make them perverse.

\begin{proof}
Theorem 5.8 in \cite{Ha1} characterizes\footnote{Theorem 5.8 in \cite{Ha1} assumes \emph{constant} parameter systems, but this is not necessary to the conclusion; see Lemma 5.3.3 in \cite{roro} for some explanation of why this is so.} the (normalized) Bernstein basis function attached to $ \mu $ as that which is central, supported on $ \admissible ( \mu ) $ and has value $ 1 $ on the dominant cell $ C_{\mu}(\F_p) $. The Main Theorem provides the first of these requirements. By definition, $ \sstrace{\mu} $ is supported on $ \localmodel ( \F_p ) $, so Proposition \ref{Padmissiblesetlocalmodel} provides the second requirement. To verify the third requirement it suffices, by Lemma 8.6 \cite{Ha2}, to check that some point of $ C_{\mu}(\F_p) $ is a smooth point of $ \localmodel_{\overline{\F}_p} $. But this is immediate because the cell $ C_{\mu} $ is itself smooth (in fact, an affine space) and simultaneously a Zariski-open subset of $ \localmodel_{\overline{\F}_p} $ (the combinatorial closure relation among Schubert cells shows that the complement in $ \localmodel_{\overline{\F}_p} $ is closed).
\end{proof}

\section{Definition of the convolution diagram} \label{Sconvolutiondiagram}

\subsection{The full affine flag variety over $ \Z_p $} \label{SSintegralcompleteaffineflagvariety}

Here I recall the definition of the affine flag variety and Iwahori subgroup over $ \Z_p $ as limits of projective $ \Z_p $-schemes in a way that is compatible with the definition of $ \biggermodel{m}{n} $ and $ \biggergroup{m}{n} $. By extending scalars, this integral affine flag variety gives the usual affine flag varieties over $ \Q_p $ and $ \F_p $. The construction is just a slight variation on the previous theme.

\begin{integralcompleteaffineflagvarietydef} \label{labelintegralcompleteaffineflagvarietydef}
Fix $ \mu, \nu \in \N $. The functor
\begin{equation*}
\integralaffineflagvarargs{\mu}{\nu} : \catalgebras{\Z_p} \longrightarrow \catsets
\end{equation*}
assigns to each commutative $ \Z_p $-algebra $ R $ the set of all tuples $ ( \mathcal{F}_0, \ldots, \mathcal{F}_{d/2} ) $ of $ \mathcal{R}[t] $-submodules of $ \mathcal{V}(R) $ such that
\begin{itemize}
\setlength{\itemsep}{5pt}
\item $ \mathcal{F}_0 \subset \cdots \subset \mathcal{F}_{d/2} $

\item $ (t+p)^{\nu} \mathcal{V}_i(R) \subset \mathcal{F}_i \subset (t+p)^{-\mu} \mathcal{V}_i(R) $ for each $ 0 \leq i \leq d/2 $

\item each inclusion $ \mathcal{F}_i / (t+p)^{\nu} \mathcal{V}_i(R) \hookrightarrow (t+p)^{-\mu} \mathcal{V}_i(R) / (t+p)^{\nu} \mathcal{V}_i(R)
$ splits $ R $-linearly

\item the projective rank function $ \spec(R) \rightarrow \N $ of each $ \mathcal{F}_i / (t+p)^{\nu} \mathcal{V}_i(R) $ is the constant function $ \mathfrak{p} \mapsto d(\mu+\nu) $

\item $ \mathcal{F}_0 $ is self-dual\footnote{The notion of duality here is similar to the one occurring in \BLMdef{\ref{BLMselfdual}}: it is required that $ \mathcal{F}_0 = \{ x \in \mathcal{V}(R) \suchthat \phi_R ( \mathcal{F}_0, x ) \subset (t+p)^{\nu-\mu} \mathcal{R}[t] \} $ and similarly for $ \mathcal{F}_{d/2} $.} with respect to $ (t+p)^{\mu-\nu} \phi_R $ and $ \mathcal{F}_{d/2} $ is self-dual with respect to $ (t+p)^{\mu-\nu+1} \phi_R $.
\end{itemize}
\end{integralcompleteaffineflagvarietydef}

Define
\begin{align*}
\mathcal{U}_{\textup{inf}} &\defeq (t+p)^{\nu} \mathcal{O}[t]^d \\
\mathcal{U}_{\textup{sup}} &\defeq (t+p)^{-\mu-1} \mathcal{O}[t]^d \\
\overline{\mathcal{U}}_{\textup{sup}} &\defeq \mathcal{U}_{\textup{sup}} / \mathcal{U}_{\textup{inf}}
\end{align*}

For the purpose of this subsection, redefine $ \overline{\phi} $ to be the hermitian $ \Z_p[t] $-bilinear form
\begin{equation*}
\overline{\phi} : \overline{\mathcal{U}}_{\textup{sup}} \times \overline{\mathcal{U}}_{\textup{sup}} \longrightarrow \frac{ (t+p)^{-2(\mu+1)} \mathcal{O}[t] }{ (t+p)^{\nu-\mu-1} \mathcal{O}[t] }
\end{equation*}
defined by the anti-identity matrix $ \antiid $, and redefine $ \overline{\mathcal{V}}_i $ to be the image of $ \mathcal{V}_i $ in $ \overline{\mathcal{U}}_{\textup{sup}} $.

\begin{integraliwahorisubgroupdef}
Fix $ \mu, \nu \in \N $. The functor
\begin{equation*}
\integraliwahoriargs{\mu}{\nu} : \catalgebras{\Z_p} \longrightarrow \catgroups
\end{equation*}
assigns to each commutative $ \Z_p $-algebra $ R $ the group of all $ \mathcal{R}[t] $-linear automorphisms $ g $ of $ \overline{\mathcal{U}}_{\textup{sup}} $ that stabilize each $ \overline{\mathcal{V}}_i $ and are similitudes with respect to the product $ \overline{\phi}_R $ with multiplier $ c(g) \in R[t] $ representing a \emph{unit} in $ R[t] / (t+p)^{\mu+\nu+1} R[t] $.
\end{integraliwahorisubgroupdef}
Like $ \biggergroup{m}{n} $, each of these $ \integraliwahoriargs{\mu}{\nu} $ is a smooth affine algebraic group $ \Z_p $-scheme (proof omitted).

These schemes have a few purposes. First, the full affine flag varieties over $ \Q_p $ and $ \F_p $ are just the fibers of
\begin{equation*}
\integralaffineflagvar \defeq \bigcup_{(\mu,\nu)} \integralaffineflagvarargs{\mu}{\nu},
\end{equation*}
hence the name.

Second, $ \integralaffineflagvar $ has a Bruhat-Tits Decomposition $ \integralaffineflagvar = \coprod C_w $ over $ \Z_p $, i.e. the Schubert cells $ C_w $ are $ \Z_p $-schemes (The abuse of notation is acceptable because the extension to $ \F_p $ of this $ C_w $ is the ``$ C_w $'' from \S\ref{SSbruhattitsdecomp} (page \pageref{SSbruhattitsdecomp})).

Third, one can define the \etale $ \ell $-adic intersection complex $ \IC{w} $ on $ \overline{C}_w $ and the restrictions $ \IC{w} \vert_{ \integralaffineflagvar_{ \F_p } } $ and $ \IC{w} \vert_{ \integralaffineflagvar_{ \Q_p } } $ are just the corresponding intersection complexes on the affine flag varieties over $ \F_p $ and $ \Q_p $. Because of this, \S5.2 of \cite{HN} shows that
\begin{equation*}
\ICbar{w} \vert_{ \integralaffineflagvar_{ \overline{\F}_p } } \directedisom \nearbycycles ( \IC{w} \vert_{ \integralaffineflagvar_{ \Q_p } } )
\end{equation*}

\begin{remark}
\S5.2 of \cite{HN} applies because the fields involved here are algebraically-closed: an argument similar to that given in \S6.3.3 of \cite{Ha2} proves that the schemes used here simplify to the $ \GL $-versions after passing to the algebraic closure.
\end{remark}

\subsection{Group-like schemes to act on $ \biggermodel{m}{n} $ and $ \integralaffineflagvarargs{\mu}{\nu} $} \label{SSmonoiddefs}

\begin{center}
\emph{The purpose of this subsection is to define two schemes of certain endomorphisms which play the role, in the ``truncated case'' of $ \biggermodel{m}{n} $ and $ \integralaffineflagvarargs{\mu}{\nu} $, of the algebraic group acting on its affine flag variety by left-multiplication.}
\end{center}

Fix $ m,n,\mu,\nu \in \N $. Define
\begin{equation*}
\overline{\mathcal{W}}_{\textup{sup}} \defeq \frac{ t^{-m} (t+p)^{-\mu-1} \mathcal{O}[t]^d }{ t^n (t+p)^{\nu} \mathcal{O}[t]^d }
\end{equation*}
By convention, $ \overline{\mathcal{W}}_{\textup{sup}} ( R ) = t^{-m} (t+p)^{-\mu-1} \mathcal{R}[t]^d / t^n (t+p)^{\nu} \mathcal{R}[t]^d $ for any $ \Z_p $-algebra $ R $ ($ \mathcal{R} := R \otimes_{\Z_p} \mathcal{O} $). This $ \overline{\mathcal{W}}_{\textup{sup}} $ is a universal container for all modules occurring in the definitions of both $ \biggermodel{m}{n} $ and $ \integralaffineflagvarargs{\mu}{\nu} $.

For this section, redefine
\begin{equation*}
\overline{\phi} : \overline{\mathcal{W}}_{\textup{sup}} \times \overline{\mathcal{W}}_{\textup{sup}} \longrightarrow \frac{ t^{-2m} (t+p)^{-2(\mu+1)} \mathcal{O}[t] }{ t^{n-m} (t+p)^{\nu - \mu - 1} \mathcal{O}[t] }
\end{equation*}
be the hermitian $ \Z_p[t] $-bilinear form defined by the anti-identity matrix $ \antiid $, and redefine $ \overline{\mathcal{V}}_i $ to be the image of $ \mathcal{V}_i $ in $ \overline{\mathcal{W}}_{\textup{sup}} $.

\begin{definition}
The functor\footnote{It is a slight abuse of notation to supress the indices $ \mu $ and $ \nu $ from the name of the functor.}
\begin{equation*}
\biggermonoid{m}{n} : \catalgebras{\Z_p} \longrightarrow \catsets
\end{equation*}
assigns to each commutative $ \Z_p $-algebra $ R $ ($ \mathcal{R} := R \otimes_{\Z_p} \mathcal{O} $) the set of all $ \mathcal{R}[t] $-linear maps $ g : \overline{\mathcal{W}}_{\textup{sup}} (R) \rightarrow \overline{\mathcal{W}}_{\textup{sup}} (R) $ such that
\begin{enumerate}
\setlength{\itemsep}{5pt}
\item \label{Dbiggermonoidcontainment} each $ \overline{\mathcal{L}}_i \defeq g ( t^{-m} \overline{\mathcal{V}}_i(R) ) $ satisfies $ t^n \overline{\mathcal{V}}_i(R) \subset \overline{\mathcal{L}}_i \subset t^{-m} \overline{\mathcal{V}}_i(R) $

\item \label{Dbiggermonoidprojectivity} each inclusion $ \overline{\mathcal{L}}_i / t^n \overline{\mathcal{V}}_i(R) \hookrightarrow t^{-m} \overline{\mathcal{V}}_i(R) / t^n \overline{\mathcal{V}}_i(R) $ splits $ R $-linearly

\item \label{Dbiggermonoidranks} the projective rank function $ \spec(R) \rightarrow \N $ of each $ \overline{\mathcal{L}}_i / t^n \overline{\mathcal{V}}_i(R) $ is the constant function $ \mathfrak{p} \mapsto d(m+n) $

\item \label{Dbiggermonoidsimilitude} there exists a $ c(g) \in R[t] $ representing a unit in $ R[t] / t^{m+n} (t+p)^{\mu+\nu+1} R[t] $ such that $ \overline{\phi}_R ( g(x), g(y) ) = c(g) t^{m+n} \overline{\phi}_R ( x, y ) $ for all $ x, y \in \overline{\mathcal{W}}_{\textup{sup}} (R) $

\item \label{Dbiggermonoidlifting} Setting
\begin{align*}
t^{-m} \widetilde{\mathcal{V}}_i &{\defeq} \frac{ t^{-m} (t+p)^{-1} \mathcal{O}[t]^i \oplus t^{-m} \mathcal{O}[t]^{d-i} }{ t^n (t+p)^{\mu+\nu+1} \mathcal{O}[t]^d } \\
\widetilde{\mathcal{L}}_i &{\defeq} \frac{ \mathcal{L}_i }{ t^{m+n} (t+p)^{\mu+\nu+1} \mathcal{L}_0 }
\end{align*}
there exists, Zariski-locally on $ \spec(R) $, an $ \mathcal{R}[t] $-linear isomorphism
\begin{equation*}
\widetilde{g} : t^{-m} \widetilde{\mathcal{V}}_0 (R) \directedisom \widetilde{\mathcal{L}}_0
\end{equation*}
inducing\footnote{The sense in which $ \widetilde{g} $ induces $ g $ is that multiplication by $ (t+p)^{\mu+1} $ within $ \mathcal{V} $ yields an identification of $ t^{-m} \widetilde{\mathcal{V}}_0 $ with $ \overline{\mathcal{W}}_{\textup{sup}} $. In the rest of the paper, this principle will be referred to as a ``shift''.} $ g : t^{-m} \overline{\mathcal{V}}_0 (R) \rightarrow \overline{\mathcal{L}}_0 $ such that
\begin{equation*}
\widetilde{g} ( (t+p) t^{-m} \widetilde{\mathcal{V}}_i (R) ) = (t+p) \widetilde{\mathcal{L}}_i
\end{equation*}
for all $ i $ and such that\footnote{The ordinary product $ \overline{\phi}_R $ is well-defined on $ \widetilde{\mathcal{L}}_0 $ due to \BLMdef{\ref{BLMselfdual}}.}
\begin{equation*}
\overline{\phi}_R ( \widetilde{g} ( x ), \widetilde{g} ( y ) ) = c(g) t^{m+n} \overline{\phi}_R ( x, y )
\end{equation*}
for all $ x, y \in t^{-m} \widetilde{\mathcal{V}}_0 (R) $
\end{enumerate}
\end{definition}
%\footnote{The meaning of ``Zariski-locally'' here is that there are multiplicative subsets $ S_1, \ldots, S_n $ covering $ \spec(R) $ and compatible $ S_i^{-1}\mathcal{R}[t] $-linear isomorphisms $ t^{-m} \widetilde{\mathcal{V}}_0 (S_i^{-1} R) \directedisom S_i^{-1} \widetilde{\mathcal{L}}_0 $ inducing $ S_i^{-1} g $.}

It is clear that any tuple $ ( \overline{\mathcal{L}}_0, \ldots, \overline{\mathcal{L}}_{d/2} ) $ coming from a point of $ \biggermonoid{m}{n} ( R ) $ has all the properties necessary to be the image in $ t^{-m} \overline{\mathcal{V}}_i(R) / t^n \overline{\mathcal{V}}_i(R) = t^{-m} \mathcal{V}_i(R) / t^n \mathcal{V}_i(R) $ of a point $ ( \mathcal{L}_0, \ldots, \mathcal{L}_{d/2} ) \in \biggermodel{m}{n}(R) $ except possibly the duality condition $ \BLMdef{\ref{BLMselfdual}} $. That condition follows from the similitude condition (notice that $ \overline{\phi} $ here restricts, descends and retracts to the $ \overline{\phi} $ from \S\ref{SSaltdesc} (page \pageref{SSaltdesc})), taking into account that the factor of $ t^{m+n} $ comes from a normalization: $ t^{m-n} \overline{\phi}_R $ and $ t^{2m} \overline{\phi}_R $ send $ \overline{\mathcal{L}}_0 \times \overline{\mathcal{L}}_0 $ and $ t^{-m} \overline{\mathcal{V}}_0 (R) \times t^{-m} \overline{\mathcal{V}}_0 (R) $ (respectively) into $ \mathcal{R}[t] $, and $ \widetilde{g} $ should identify $ t^{2m} \overline{\phi}_R $ to $ t^{m-n} \overline{\phi}_R $, hence the above requirement.

Therefore, I may define
\begin{align}
\label{Ebiggermonoidtobiggermodel} \biggermonoid{m}{n}(R) &\longrightarrow \biggermodel{m}{n}(R) \\
\nonumber g &\longmapsto ( \mathcal{L}_0, \ldots, \mathcal{L}_{d/2} )
\end{align}

%\begin{remark}
%Recall that if $ m = n $ then the trivial tuple $ ( \mathcal{V}_0(R), \ldots, \mathcal{V}_{d/2}(R) ) $ is an element of $ \biggermodel{m}{m} (R) $, so the map $ w \mapsto t^m w $ is a sort of ``identity element'' of $ \biggermonoid{m}{m} $. As before, this is related to the identity component of the affine flag varieties.
%\end{remark}

\begin{prop} \label{Pbiggermonoidisrepresentable}
The $ \Z_p $-scheme $ \biggermonoid{m}{n} $ is finite-type.
\end{prop}

\begin{proof}
Let $ \biggermonoid{m}{n}_{\textup{weak}} : \catalgebras{\Z_p} \rightarrow \catsets $ be the functor defined only by conditions (\ref{Dbiggermonoidcontainment}), (\ref{Dbiggermonoidprojectivity}), (\ref{Dbiggermonoidranks}) and (\ref{Dbiggermonoidsimilitude}). Conditions (\ref{Dbiggermonoidcontainment}) and (\ref{Dbiggermonoidsimilitude}) obviously define a finite-type scheme, and Lemma 18 from \cite{HN} handles conditions (\ref{Dbiggermonoidprojectivity}) and (\ref{Dbiggermonoidranks}), so $ \biggermonoid{m}{n}_{\textup{weak}} $ is a finite-type scheme. Now consider $ \biggermonoid{m}{n} \subset \biggermonoid{m}{n}_{\textup{weak}} $. It is clear that the similitude part of condition (\ref{Dbiggermonoidlifting}) is no problem, so I now check the Zariski-local existence statement.

Let $ R $ be a commutative $ \Z_p $-algebra and temporarily fix $ g \in \biggermonoid{m}{n}_{\textup{weak}}(R) $. The set of \emph{possibly-non-invertible} $ \widetilde{g} $ inducing $ g $ \emph{globally} on $ \spec(R) $ is clearly the $ R $-points of a finite-dimensional affine space (choose additional matrix entries from $ \widetilde{\mathcal{L}}_0 / \overline{\mathcal{L}}_0 $). Let $ k $ be the dimension of this affine space. Then for fixed $ g $, the condition of invertibility can be expressed (by Cramer's Rule) as a polynomial equation by using the associated $ k $-variable determinant $ \det_g $. Now I extend this to the Zariski-local case.

Define for each $ n \in \N $ and $ \ell_1, \ldots, \ell_n, m_1, \ldots, m_n \in \N $ the ideal $ I ( n ; \{ \ell_i \} ; \{ m_j \} ) $ in
\begin{equation*}
\Z_p [ W_1, \ldots, W_n ; X_1, \ldots, X_n ; Y_1, \ldots, Y_n ; Z_1^{(1)}, \ldots, Z_k^{(1)}, \ldots, Z_1^{(n)}, \ldots, Z_k^{(n)} ]
\end{equation*}
by the equations
\begin{align*}
W_1 Y_1 + \cdots + W_n Y_n &= 1 \\
Y_1^{\ell_1} \cdot ( \mydet_g ( Z_1^{(1)}, \ldots, Z_k^{(1)} ) \cdot X_1 - Y_1^{m_1} \cdot 1 ) &= 0 \\
&{\vdots} \\
Y_n^{\ell_n} \cdot ( \mydet_g ( Z_1^{(n)}, \ldots, Z_k^{(n)} ) \cdot X_n - Y_n^{m_n} \cdot 1 ) &= 0 \\
\end{align*}
In any $ R $-valued solution to this system,
\begin{itemize}
\setlength{\itemsep}{5pt}
\item the values $ Y_1, \ldots, Y_n $ will, because of the first equation, be generators of the trivial ideal $ R $, i.e. a principal open cover of $ \spec(R) $ (the values $ W_i $ are auxiliary),

\item the values $ Z_1^{(i)}, \ldots, Z_k^{(i)} $ define the (at the moment possibly-non-invertible) Zariski-local lift of $ g $ over the principal open subset $ \spec(R_{Y_i}) $, and

\item the last $ n $ equations exactly express (by definition of the fraction ring $ R_{Y_i} $) that the determinant of the Zariski-local lift is a unit, i.e. that each Zariski-local lift is invertible.
\end{itemize}

Since $ \biggermonoid{m}{n}_{\textup{weak}} $ is already known to be a scheme, it is clear that the above system of equations can be extended (simply add more variables and the ideal defining $ \biggermonoid{m}{n}_{\textup{weak}} $) to eliminate the assumption that $ g $ is fixed. Taking $ I $ to be the sum of all the above ideals in the obvious countably-generated polynomial ring, it is then clear that the subfunctor $ \biggermonoid{m}{n} \subset \biggermonoid{m}{n}_{\textup{weak}} $ is representable (it is the image under the forgetful morphism $ ( \widetilde{g}, g ) \mapsto g $ of scheme defined by the ideal $ I $). This proves representability, and finite-type is then obvious.
\end{proof}

Similarly,
\begin{definition}
The functor\footnote{It is a slight abuse of notation to supress the indices $ m $ and $ n $ from the name of the functor.}
\begin{equation*}
\integralaffineflagvarmonoid{\mu}{\nu} : \catalgebras{\Z_p} \longrightarrow \catsets
\end{equation*}
assigns to each commutative $ \Z_p $-algebra $ R $ ($ \mathcal{R} := R \otimes_{\Z_p} \mathcal{O} $) the set of all $ \mathcal{R}[t] $-linear maps $ h : \overline{\mathcal{W}}_{\textup{sup}} (R) \rightarrow \overline{\mathcal{W}}_{\textup{sup}} (R) $ such that
\begin{itemize}
\setlength{\itemsep}{5pt}
\item each $ \overline{\mathcal{F}}_i \defeq h ( (t+p)^{-\mu} \overline{\mathcal{V}}_i(R) ) $ satisfies $ (t+p)^{\nu} \overline{\mathcal{V}}_i(R) \subset \overline{\mathcal{F}}_i \subset (t+p)^{-\mu} \overline{\mathcal{V}}_i(R) $

\item each inclusion $ \overline{\mathcal{F}}_i / (t+p)^{\nu} \overline{\mathcal{V}}_i(R) \hookrightarrow (t+p)^{-\mu} \overline{\mathcal{V}}_i(R) / (t+p)^{\nu} \overline{\mathcal{V}}_i(R) $ splits $ R $-linearly

\item the projective rank function $ \spec(R) \rightarrow \N $ of each $ \overline{\mathcal{F}}_i / (t+p)^{\nu} \overline{\mathcal{V}}_i(R) $ is the constant function $ \mathfrak{p} \mapsto d(\mu+\nu) $

\item there exists a $ c(h) \in R[t] $ representing a unit in $ R[t] / t^{m+n} (t+p)^{\mu+\nu+1} R[t] $ such that $ \overline{\phi}_R ( h(x), h(y) ) = c(h) (t+p)^{\mu+\nu} \overline{\phi}_R ( x, y ) $ for all $ x, y \in \overline{\mathcal{W}}_{\textup{sup}} (R) $

\item Setting
\begin{align*}
(t+p)^{-\mu} \widetilde{\mathcal{V}}_i &{\defeq} \frac{ (t+p)^{-\mu-1} \mathcal{O}[t]^i \oplus (t+p)^{-\mu} \mathcal{O}[t]^{d-i} }{ t^{m+n+1} (t+p)^{\nu} \mathcal{O}[t]^d } \\
\widetilde{\mathcal{F}}_i &{\defeq} \frac{ \mathcal{F}_i }{ t^{m+n} (t+p)^{\mu+\nu+1} \mathcal{F}_0 }
\end{align*}
there exists, Zariski-locally on $ \spec(R) $, an $ \mathcal{R}[t] $-linear isomorphism
\begin{equation*}
\widetilde{h} : (t+p)^{-\mu} \widetilde{\mathcal{V}}_0 (R) \directedisom \widetilde{\mathcal{F}}_0
\end{equation*}
inducing $ h : (t+p)^{-\mu} \overline{\mathcal{V}}_0(R) \rightarrow \overline{\mathcal{F}}_0 $ such that
\begin{equation*}
\widetilde{h} ( (t+p)^{-\mu+1} \widetilde{\mathcal{V}}_i (R) ) = (t+p) \widetilde{\mathcal{F}}_i
\end{equation*}
for all $ i $ and such that
\begin{equation*}
\overline{\phi}_R ( \widetilde{h} ( x ), \widetilde{g} ( y ) ) = c(h) (t+p)^{\mu+\nu} \overline{\phi}_R ( x, y )
\end{equation*}
for all $ x, y \in (t+p)^{-\mu} \widetilde{\mathcal{V}}_0 (R) $
\end{itemize}
\end{definition}

As before,
\begin{prop}
The $ \Z_p $-scheme $ \integralaffineflagvarmonoid{\mu}{\nu} $ is finite-type.
\end{prop}

\begin{proof}
This is nearly identical to the proof for $ \biggermonoid{m}{n} $ (page \pageref{Pbiggermonoidisrepresentable}).
\end{proof}

As before, I may define
\begin{align}
\label{Eaffineflagvarmonoidtoaffineflagvar} \integralaffineflagvarmonoid{\mu}{\nu}(R) &\longrightarrow \integralaffineflagvarargs{\mu}{\nu}(R) \\
\nonumber h &\longmapsto ( h( (t+p)^{-\mu} \overline{\mathcal{V}}_0(R) ), \ldots, h( (t+p)^{-\mu} \overline{\mathcal{V}}_{d/2}(R) ) )
\end{align}

%\begin{remark}
%As in the case of $ \biggermonoid{m}{n} $, notice that if $ \mu = \nu $ then $ \integralaffineflagvarargs{\mu}{\mu}$ and $ \integralaffineflagvarmonoid{\mu}{\mu} $ have an ``identity element''.
%\end{remark}

\begin{definition}
The morphism
\begin{equation*}
p_1 : \biggermonoid{m}{n} \times \integralaffineflagvarmonoid{\mu}{\nu} \longrightarrow \biggermodel{m}{n} \times \integralaffineflagvarargs{\mu}{\nu}
\end{equation*}
is the product of morphisms (\ref{Ebiggermonoidtobiggermodel}) and (\ref{Eaffineflagvarmonoidtoaffineflagvar}).
\end{definition}

\subsection{The convolution scheme}

\begin{center}
\emph{The purpose of this subsection is to define a $ \Z_p $-scheme $ \convolutionobject{m}{n}{\mu}{\nu} $ and a ``twisted action'' morphism $ p_2 $ to $ \convolutionobject{m}{n}{\mu}{\nu} $ which will express the summation that defines convolution for the Iwahori-Hecke algebra. See \S\ref{SSfusionproductcategorifies} (page \pageref{SSfusionproductcategorifies}).}
\end{center}

\begin{convolutionscheme}
The functor
\begin{equation*}
\convolutionobject{m}{n}{\mu}{\nu} : \catalgebras{\Z_p} \longrightarrow \catsets
\end{equation*}
assigns to each commutative $ \Z_p $-algebra $ R $ the set of tuples $ ( \mathcal{L}_0, \ldots, \mathcal{L}_{d/2} ; \mathcal{K}_0, \ldots, \mathcal{K}_{d/2} ) $ of $ \mathcal{R}[t] $-submodules of $ \mathcal{V}(R) $ satisfying
\begin{itemize}
\setlength{\itemsep}{5pt}
\item $ ( \mathcal{L}_0, \ldots, \mathcal{L}_{d/2} ) \in \biggermodel{m}{n} (R) $

\item each $ \mathcal{K}_i $ satisfies $ (t+p)^{\nu} \mathcal{L}_i \subset \mathcal{K}_i \subset (t+p)^{-\mu} \mathcal{L}_i $

\item each inclusion $ \mathcal{K}_i / (t+p)^{\nu} \mathcal{L}_i \hookrightarrow (t+p)^{-\mu} \mathcal{L}_i / (t+p)^{\nu} \mathcal{L}_i $ splits $ R $-linearly

\item the projective rank function $ \spec(R) \rightarrow \N $ of each $ \mathcal{K}_i / (t+p)^{\nu} \mathcal{L}_i $ is the constant function $ \mathfrak{p} \mapsto d(\mu+\nu) $

(note that the previous property implies $ R $-projectivity)

\item $ \mathcal{K}_0 $ is self-dual\footnote{The notion of duality here is similar to the one occurring in \BLMdef{\ref{BLMselfdual}}: it is required that $ \mathcal{K}_0 = \{ x \in \mathcal{V}(R) \suchthat \phi_R ( \mathcal{K}_0, x ) \subset t^{n-m} (t+p)^{\nu-\mu} \mathcal{R}[t] \} $ and similarly for $ \mathcal{K}_{d/2} $.} with respect to $ t^{m-n} (t+p)^{\mu-\nu} \phi_R $ and $ \mathcal{K}_{d/2} $ is self-dual with respect to $ t^{m-n} (t+p)^{\mu-\nu+1} \phi_R $

(the $ \phi $ used here has domain $ \mathcal{V} \times \mathcal{V} $)
\end{itemize}
\end{convolutionscheme}

Similar to $ \biggermodel{m}{n} $ (see \S\ref{SSordinarygrassmannianembedding} (page \pageref{SSordinarygrassmannianembedding})), this $ \convolutionobject{m}{n}{\mu}{\nu} $ is a closed subscheme of a product of (ordinary) Grassmannians. In particular, the $ \Z_p $-scheme $ \convolutionobject{m}{n}{\mu}{\nu} $ is \emph{proper}.

\begin{definition}
The morphism
\begin{equation*}
p_2 : \biggermonoid{m}{n} \times \integralaffineflagvarmonoid{\mu}{\nu} \longrightarrow \convolutionobject{m}{n}{\mu}{\nu}
\end{equation*}
is defined on $ R $-points by $ ( g, h ) \mapsto ( g( t^{-m} \overline{\mathcal{V}}_{\bullet}(R) ) ; g ( h ( t^{-m} (t+p)^{-\mu} \overline{\mathcal{V}}_{\bullet}(R) ) ) ) $ for any commutative $ \Z_p $-algebra $ R $ 

(these images of $ g $ and $ h $ are technically submodules of $ \overline{\mathcal{W}}_{\textup{sup}} (R) $, but as usual I replace them by their corresponding submodules of $ \mathcal{V}(R) $).
\end{definition}

To verify that the codomain of $ p_2 $ really is correct, note that the 1st coordinate of $ p_2 $ is simply the previously verified action morphism $
\biggermonoid{m}{n} \rightarrow \biggermodel{m}{n} $ and that by definition of $ h $,
\begin{equation} \label{Etwistedactionproof1}
(t+p)^{\nu} t^{-m} \overline{\mathcal{V}}_i(R) \subset h ( t^{-m} (t+p)^{-\mu} \overline{\mathcal{V}}_i(R) ) \subset (t+p)^{-\mu} t^{-m} \overline{\mathcal{V}}_i(R)
\end{equation}
and this chain is transformed by $ g $ to the chain
\begin{equation} \label{Etwistedactionproof2}
(t+p)^{\nu} \overline{\mathcal{L}}_i \subset g ( h ( t^{-m} (t+p)^{-\mu}
\overline{\mathcal{V}}_i(R) ) ) \subset (t+p)^{-\mu} \overline{\mathcal{L}}_i
\end{equation}

Finally, I define a $ \Z_p $-scheme $ \convolutiontarget{m}{n}{\mu}{\nu} $ essentially as a target for the 2nd projection from $ \convolutionobject{m}{n}{\mu}{\nu} $:

\begin{convolutionschemebase}
The functor
\begin{equation*}
\convolutiontarget{m}{n}{\mu}{\nu} : \catalgebras{\Z_p} \longrightarrow \catsets
\end{equation*}
assigns to each commutative $ \Z_p $-algebra $ R $ the set of all tuples $ ( \mathcal{K}_0, \ldots, \mathcal{K}_{d/2} ) $ of $ \mathcal{R}[t] $-submodules of $ \mathcal{V}(R) $ satisfying
\begin{itemize}
\setlength{\itemsep}{5pt}
\item each $ \mathcal{K}_i $ satisfies $ t^n (t+p)^{\nu} \mathcal{V}_i(R) \subset \mathcal{K}_i \subset t^{-m} (t+p)^{-\mu} \mathcal{V}_i(R) $

\item each inclusion $ \mathcal{K}_i / t^n (t+p)^{\nu} \mathcal{V}_i(R) \hookrightarrow t^{-m}
(t+p)^{-\mu} \mathcal{V}_i(R) / t^n (t+p)^{\nu} \mathcal{V}_i(R) $
splits $ R $-linearly

\item the projective rank function $ \spec(R) \rightarrow \N $ of each $ \mathcal{K}_i / t^n (t+p)^{\nu} \mathcal{V}_i(R) $ is the constant function $ \mathfrak{p} \mapsto (m+n+\mu+\nu)d $

(the rank here is bigger than the rank in the definition of $ \convolutionobject{m}{n}{\mu}{\nu} $ because the quotient here is also bigger)

\item $ \mathcal{K}_0 $ is self-dual\footnote{The notion of duality here is identical to that for $ \convolutionobject{m}{n}{\mu}{\nu}$.} with respect to $ t^{m-n} (t+p)^{\mu-\nu} \phi_R $ and $ \mathcal{K}_{d/2} $ is self-dual with respect to $ t^{m-n} (t+p)^{\mu-\nu+1} \phi_R $
\end{itemize}
\end{convolutionschemebase}

%\begin{remark}
%Notice that if $ \mu = \nu $ then $ \biggermodel{m}{n} \subset \convolutiontarget{m}{n}{\mu}{\mu} $: all the conditions are obviously satisfied except possibly the rank condition, which is true because, Zariski-locally on $ \spec(R) $,
%\begin{equation*}
%\rank_R ( \mathcal{L}_i / t^n (t+p)^{\nu} \mathcal{V}_i(R) ) = (m+n)d + 2 \nu d = (m+n)d + (\mu+\nu)d.
%\end{equation*}

%This is related to the ``identity element'' $ 1 \in \integralaffineflagvarmonoid{\mu}{\mu} $: if $ g \in \biggermonoid{m}{n} $ maps to $ ( \mathcal{L}_0, \ldots, \mathcal{L}_{d/2} ) \in \biggermodel{m}{n} $, then considered as an element of $ \convolutiontarget{m}{n}{\mu}{\mu} $,
%\begin{equation*}
%( \mathcal{L}_0, \ldots, \mathcal{L}_{d/2} ) = m ( p_2 ( g, 1 ) ).
%\end{equation*}
%\end{remark}

Similar to $ \biggermodel{m}{n} $ (see \S\ref{SSordinarygrassmannianembedding} (page \pageref{SSordinarygrassmannianembedding})), this $ \convolutiontarget{m}{n}{\mu}{\nu} $ is a closed subscheme of a product of (ordinary) Grassmannians. In particular, the $ \Z_p $-scheme $ \convolutiontarget{m}{n}{\mu}{\nu} $ is \emph{proper}.

\begin{definition}
The morphism
\begin{equation*}
m : \convolutionobject{m}{n}{\mu}{\nu} \longrightarrow \convolutiontarget{m}{n}{\mu}{\nu}
\end{equation*}
is defined on $ R $-points by $ ( \mathcal{L}_{\bullet} ; \mathcal{K}_{\bullet} ) \mapsto \mathcal{K}_{\bullet} $ for any commutative $ \Z_p $-algebra $ R $
\end{definition}

The verification that this function has codomain $ \convolutiontarget{m}{n}{\mu}{\nu} $ is very easy: the duality condition is identical for both schemes, the containment relations are verified by concatenating the containment relations satisfied by the $ \mathcal{L}_i $ (i.e. \BLMdef{\ref{BLMbounds}}) onto those satisfied by the $ \mathcal{K}_i $, and it is easy to see that the rank condition is then satisfied.

The morphism $ m $ is automatically proper due to the fact that the domain and codomain are proper schemes.

\subsection{Convolution diagram construction}
Combining all the above objects and morphisms gives, at long last, the convolution diagram:
\begin{equation*}
\biggermodel{m}{n} \times \integralaffineflagvarargs{\mu}{\nu}
\stackrel{ p_1 }{ \longleftarrow } \biggermonoid{m}{n} \times
\integralaffineflagvarmonoid{\mu}{\nu} \stackrel{ p_2 }{
\longrightarrow } \convolutionobject{m}{n}{\mu}{\nu} \stackrel{ m }{
\longrightarrow } \convolutiontarget{m}{n}{\mu}{\nu}
\end{equation*}

\subsection{The ``reversed'' convolution diagram} \label{SSreversedconvolutiondiagram}

I now construct a ``reversed'' convolution diagram, which is used to construct a ``reversed'' convolution product product (see the end of \S\ref{SSfusionproductdef} (page \pageref{SSfusionproductdef}). The statement that the convolution product of two particular functions is commutative is equivalent to the statement that the convolution product and reversed convolution product of the corresponding sheaf complexes are equal.

\begin{reversedconvolutionscheme}
Fix $ m,n,\mu,\nu \in \N $. The functor
\begin{equation*}
\reversedconvolutionobject{\mu}{\nu}{m}{n} : \catalgebras{\Z_p} \longrightarrow \catsets
\end{equation*}
assigns to each commutative $ \Z_p $-algebra $ R $ the set of tuples $ ( \mathcal{L}_0, \ldots, \mathcal{L}_{d/2} ; \mathcal{K}_0, \ldots, \mathcal{K}_{d/2} ) $ of $ \mathcal{R}[t] $-submodules of $ \mathcal{V}(R) $ satisfying
\begin{itemize}
\setlength{\itemsep}{5pt}
\item $ ( \mathcal{K}_0, \ldots, \mathcal{K}_{d/2} ) \in \integralaffineflagvarargs{\mu}{\nu} (R) $

\item each $ \mathcal{L}_i $ satisfies $ t^n \mathcal{K}_i \subset \mathcal{L}_i \subset t^{-m} \mathcal{K}_i $

\item each inclusion $ \mathcal{L}_i / t^n \mathcal{K}_i \hookrightarrow t^{-m} \mathcal{K}_i / t^n \mathcal{K}_i $ splits $ R $-linearly

\item the projective rank function $ \spec(R) \rightarrow \N $ of each $ \mathcal{L}_i / t^n \mathcal{K}_i $ is the constant function $ \mathfrak{p} \mapsto d(m+n) $

\item $ \mathcal{L}_0 $ is self-dual with respect to $ t^{m-n} (t+p)^{\mu-\nu} \phi_R $ and $ \mathcal{L}_{d/2} $ is self-dual with respect to $ t^{m-n} (t+p)^{\mu-\nu+1} \phi_R $

(the $ \phi $ used here has domain $ \mathcal{V} \times \mathcal{V} $)
\end{itemize}
\end{reversedconvolutionscheme}

\begin{definition}
The morphism
\begin{equation*}
\presuperscript{\textup{rev}}{m} : \reversedconvolutionobject{\mu}{\nu}{m}{n} \longrightarrow \convolutiontarget{m}{n}{\mu}{\nu}
\end{equation*}
is defined on $ R $-points by $ ( \mathcal{L}_{\bullet} ; \mathcal{K}_{\bullet} ) \mapsto \mathcal{L}_{\bullet} $ for any commutative $ \Z_p $-algebra $ R $.
\end{definition}
Just as for $ m $, it is easy to verify that this function really has codomain $ \convolutiontarget{m}{n}{\mu}{\nu} $.

\begin{definition}
The morphism
\begin{equation*}
\presuperscript{\textup{rev}}{p}_2 : \integralaffineflagvarmonoid{\mu}{\nu} \times \biggermonoid{m}{n} \longrightarrow \reversedconvolutionobject{\mu}{\nu}{m}{n}
\end{equation*}
is defined on $ R $-points by $ ( h, g ) \mapsto ( h ( (t+p)^{-\mu} \overline{\mathcal{V}}_{\bullet}(R) ) ; h ( g ( t^{-m} (t+p)^{-\mu} \overline{\mathcal{V}}_{\bullet}(R) ) ) ) $ for any commutative $ \Z_p $-algebra $ R $.
\end{definition}
Notice that the 1st coordinate of $ \presuperscript{\textup{rev}}{p}_2 $ is just $ \integralaffineflagvarmonoid{\mu}{\nu} \rightarrow \integralaffineflagvarargs{\mu}{\nu} $ from before.

\begin{definition}
By abuse of notation, $ \alpha_1 $ and $ \alpha_2 $ denote the actions of $ \universalgroup{m}{n}{\mu}{\nu} \times \universalgroup{m}{n}{\mu}{\nu} $ on $ \integralaffineflagvarmonoid{\mu}{\nu} \times \biggermonoid{m}{n} $ by the rules $ ( \gamma, \eta ) \cdot ( h, g ) \defeq ( h \circ \gamma^{-1}, g \circ \eta^{-1} ) $ and $ ( \gamma, \eta ) \cdot ( h, g ) \defeq ( h \circ \gamma^{-1}, \gamma \circ g \circ \eta^{-1} ) $.
\end{definition}

Substituting the new convolution object and morphisms yields a ``reversed'' convolution diagram:
\begin{equation*}
\integralaffineflagvarargs{\mu}{\nu} \times \biggermodel{m}{n} \stackrel{p_1}{\longleftarrow} \integralaffineflagvarmonoid{\mu}{\nu} \times \biggermonoid{m}{n} \stackrel{ \presuperscript{\textup{rev}}{p}_2 }{ \longrightarrow } \reversedconvolutionobject{\mu}{\nu}{m}{n} \stackrel{ \presuperscript{\textup{rev}}{m} }{ \longrightarrow } \convolutiontarget{m}{n}{\mu}{\nu}
\end{equation*}

\section{Properties of the convolution diagram} \label{Sconvdiagramproperties}

Recall that for a point $ ( \mathcal{L}_0, \ldots, \mathcal{L}_{d/2} ; \mathcal{K}_0, \ldots, \mathcal{K}_{d/2} ) \in \convolutionobject{m}{n}{\mu}{\nu} ( R ) $ it is not the case that the tuple $ ( \mathcal{K}_0, \ldots, \mathcal{K}_{d/2} ) $ is a point of $ \integralaffineflagvarargs{\mu}{\nu} (R) $ (indeed, that is the whole point). The following lemma, which says roughly that one can revert (\ref{Etwistedactionproof2}) back to (\ref{Etwistedactionproof1}) (page \pageref{Etwistedactionproof1}) even though an ``$ h $'' may not exist, is therefore useful to derive statements about $ p_2 $ from similar statements about $ p_1 $:

\begin{lemma} \label{Luntwistingconvolutionobject}
Let $ R $ be a \emph{local} commutative $ \Z_p $-algebra.

If $ ( \mathcal{L}_{\bullet} ; \mathcal{K}_{\bullet} ) \in \convolutionobject{m}{n}{\mu}{\nu} ( R ) $ and $ g \in \biggermonoid{m}{n} (R) $ has image $ \mathcal{L}_{\bullet} \in \biggermodel{m}{n} (R) $ then there exists $ \mathcal{F}_{\bullet} \in \integralaffineflagvarargs{\mu}{\nu} (R) $ such that $ g ( t^{-m} \overline{\mathcal{F}}_i ) = \overline{\mathcal{K}}_i $ for all $ i $.
\end{lemma}

\begin{proof}
By definition, $ (t+p)^{\nu} \mathcal{L}_i \subset \mathcal{K}_i \subset (t+p)^{-\mu} \mathcal{L}_i $. Quotient by $ t^{m+n} (t+p)^{\mu+\nu+1} \mathcal{L}_0 $ (so that the leftmost and rightmost modules involve $ \widetilde{\mathcal{L}}_i $, in the sense of $ \biggermonoid{m}{n} $), apply $ \widetilde{g}^{-1} $ (since $ R $ is assumed local, $ \widetilde{g} $ exists globally on $ \spec(R) $), scale by $ t^m $, and quotient by $ t^n (t+p)^{\nu} \mathcal{R}[t]^d $ to get a module $ \overline{\mathcal{F}}_i $ satisfying
\begin{equation} \label{Econtainmentrelation}
(t+p)^{\nu} \overline{\mathcal{V}}_i (R) \subset \overline{\mathcal{F}}_i \subset (t+p)^{-\mu} \overline{\mathcal{V}}_i (R)
\end{equation}
Because $ \widetilde{g} $ induces $ g $, it is true that $ g ( t^{-m} \overline{\mathcal{F}}_i ) = \overline{\mathcal{K}}_i $.

Recall the conditions for membership in $ \integralaffineflagvarargs{\mu}{\nu} (R) $. The containments (\ref{Econtainmentrelation}) are one of those conditions. Because $ \mathcal{K}_0 \subset \cdots \subset \mathcal{K}_{d/2} $, it is also true that $ \mathcal{F}_0 \subset \cdots \subset \mathcal{F}_{d/2} $. The projectivity condition and rank condition are satisfied because $ \widetilde{g} $ is an isomorphism and because of the similar properties of $ ( \mathcal{K}_0, \ldots, \mathcal{K}_{d/2} ) $. The duality condition is not totally obvious, but follows from the similitude property of $ \widetilde{g} $ and the similar duality property (page \pageref{labelintegralcompleteaffineflagvarietydef}) of $ ( \mathcal{K}_0, \ldots, \mathcal{K}_{d/2} ) $:
\begin{align*}
\phi_R ( \mathcal{F}_0, \mathcal{F}_0 ) &= \phi_R ( t^m \widetilde{g}^{-1} ( \mathcal{K}_0 ), t^m \widetilde{g}^{-1} ( \mathcal{K}_0 ) ) \\
&= t^{2m} t^{-(m+n)} c(g)^{-1} \phi_R ( \mathcal{K}_0 ) \\
&= t^{m-n} c(g)^{-1} \phi_R ( \mathcal{K}_0, \mathcal{K}_0 )
\end{align*}
so
\begin{equation*}
\phi_R ( \mathcal{K}_0, \mathcal{K}_0 ) \subset t^{n-m} (t+p)^{\nu-\mu} \mathcal{R}[t] \Longleftrightarrow \phi_R ( \mathcal{F}_0, \mathcal{F}_0 ) \subset (t+p)^{\nu-\mu} \mathcal{R}[t]
\end{equation*}
etc.
\end{proof}

The following is the analogue of Lemma 19 from \cite{HN}:

\begin{prop} \label{Psmoothsurjections}
The morphisms $ p_1 $ and $ p_2 $ are smooth and for any $ \Z_p $-field $ K $, the functions $ p_1(K) $ and $ p_2(K) $ are surjective.
\end{prop}

\begin{proof}
\framebox{smoothness} To prove that $ \biggermonoid{m}{n} \rightarrow \biggermodel{m}{n} $ is smooth, I must show that for
\begin{itemize}
\setlength{\itemsep}{5pt}
\item a local commutative $ \Z_p $-algebra $ R $,
\item an ideal $ I \subset R $ satisfying $ I^2 = 0 $,
\item $ \mathcal{L}_{\bullet} \in \biggermodel{m}{n} (R) $, and
\item $ g_{R/I} \in \biggermonoid{m}{n}(R/I) $ such that
\begin{equation*}
g_{R/I} ( t^{-m} \overline{\mathcal{V}}_i(R/I) ) = \overline{\mathcal{L}}_i \otimes_R ( R / I )
\end{equation*}
for each $ 0 \leq i \leq d/2 $
\end{itemize}
there exists a $ g_R \in \biggermonoid{m}{n}(R) $ such that $ g_R ( t^{-m} \overline{\mathcal{V}}_i(R) ) = \overline{\mathcal{L}}_i $ for each $ 0 \leq i \leq d/2 $ and $ g_R \mapsto g_{R/I} $ under $ \biggermonoid{m}{n}(R) \rightarrow \biggermonoid{m}{n}(R/I) $. Since $ R $ is local, denote by $ \widetilde{g}_{R/I} $ a lift of $ g_{R/I} $ as guaranteed by the definition of $ \biggermonoid{m}{n} $.

Some notation. Set
\begin{equation*}
\mathcal{S} \defeq \mathcal{R}[t] / t^{m+n} (t+p)^{\mu+\nu+1} \mathcal{R}[t]
\end{equation*}
and let $ \mathcal{I}_t $ be the extension in $ \mathcal{S} $ of the ideal $ I $. I use $ \mathcal{L}_{i,R} $ and $ \mathcal{L}_{i,R/I} $ refer to $ \mathcal{L}_i $ and $ \mathcal{L}_i \otimes_R ( R / I ) $, and so on. Let $ \jmath : \mathcal{S} \rightarrow \mathcal{S} $ be the involution induced by the non-trivial element of $ \gal ( F / \Q_p ) $.

First, a partial result:

\begin{lemma} \label{Lpreliminaryfreeness}
With $ R, I, \mathcal{L}_{\bullet}, g_{R/I} $ and the notation as above, $ \widetilde{\mathcal{L}}_0 $ is free over $ \mathcal{S} $ and the hermitian form $ \widetilde{\mathcal{L}}_0 \times \widetilde{\mathcal{L}}_0 \rightarrow \mathcal{S} $ induced by $ t^{m-n} \phi_R $ is perfect.
\end{lemma}

Note that these assertions are not automatic because it is not known \emph{a priori} that $ \mathcal{L}_{\bullet} $ is the image of some $ g_R $.

\begin{proof}
After normalization, I can assume that the domain of $ \widetilde{g} $ is $ ( \mathcal{S} / \mathcal{I}_t )^d $. Let $ \widetilde{v}_1, \ldots, \widetilde{v}_d $ be arbitrary lifts to $ \widetilde{\mathcal{L}}_0 $ of the basis $ \widetilde{g} ( e_1 ), \ldots, \widetilde{g} ( e_d ) $. By Nakayama's Lemma, $ \widetilde{v}_1, \ldots, \widetilde{v}_d $ generates $ \widetilde{\mathcal{L}}_0 $. Let
\begin{equation} \label{EpresentationoverR}
0 \longrightarrow K \longrightarrow \mathcal{S}^d \longrightarrow \widetilde{\mathcal{L}}_{0,R} \longrightarrow 0
\end{equation}
be the presentation so defined. Since $ \overline{\mathcal{L}}_{0,R} $ is $ R $-projective by \BLMdef{\ref{BLMselfdual}} (page \pageref{BLMselfdual}), and since the kernel of $ \widetilde{\mathcal{L}}_{0,R} \twoheadrightarrow \overline{\mathcal{L}}_{0,R} $ is identified $ R $-linearly with $ t^{-m} (t+p)^{-\mu} \mathcal{V}_0(R) / \mathcal{L}_{0,R} $, which is also $ R $-projective by \BLMdef{\ref{BLMselfdual}}, it follows that $ \widetilde{\mathcal{L}}_{0,R} $ is $ R $-projective. So (\ref{EpresentationoverR}) splits and
\begin{equation} \label{EpresentationoverRmodI}
0 \longrightarrow K \otimes_R R/I \longrightarrow \mathcal{S}^d \otimes_R R/I \longrightarrow \widetilde{\mathcal{L}}_{0,R} \otimes_R R/I \longrightarrow 0
\end{equation}
is still exact. The middle module is just $ ( \mathcal{S} / \mathcal{I}_t )^d $ and, again by $ R $-projectivity of $ \widetilde{\mathcal{L}}_{0,R} $, the rightmost module is just $ \widetilde{\mathcal{L}}_{0, R/I} $. This means that the presentation (\ref{EpresentationoverRmodI}) is just the one given by the \emph{isomorphism} $ \widetilde{g}_{R/I} $, which means $ K \otimes_R R/I = K / I K = 0 $. By Nakayama's Lemma (note that $ K $ is finitely-generated by the splitting of (\ref{EpresentationoverR})), $ K = 0 $ and so the lift $ \mathcal{S}^d \longrightarrow \widetilde{\mathcal{L}}_{0,R} $ of $ \widetilde{g}_{R/I} $ is an isomorphism.

For perfection, recall that the form is perfect if and only if the associated adjoint map
\begin{equation} \label{EadjointmapoverR}
\widetilde{\mathcal{L}}_{0,R} \longrightarrow \myhom_{\mathcal{S}\textup{-lin}} ( \widetilde{\mathcal{L}}_{0,R}, \mathcal{S} )
\end{equation}
is \emph{surjective}. By definition of $ \widetilde{g}_{R/I} $, the corresponding form modulo $ I $ \emph{is} perfect, i.e. the adjoint map
\begin{equation*}
\widetilde{\mathcal{L}}_{0,R/I} \longrightarrow \myhom_{(\mathcal{S} / \mathcal{I}_t)\textup{-lin}} ( \widetilde{\mathcal{L}}_{0,R/I}, \mathcal{S} / \mathcal{I}_t )
\end{equation*}
\emph{is} surjective. Since $ \widetilde{\mathcal{L}}_{0,R} $ is a free $ \mathcal{S} $-module and $ \widetilde{\mathcal{L}}_{0,R/I} = \widetilde{\mathcal{L}}_{0,R} \otimes_R R/I $,
\begin{equation*}
\myhom_{(\mathcal{S} / \mathcal{I}_t)\textup{-lin}} ( \widetilde{\mathcal{L}}_{0,R/I}, \mathcal{S} / \mathcal{I}_t ) = \myhom_{\mathcal{S}\textup{-lin}} ( \widetilde{\mathcal{L}}_{0,R}, \mathcal{S} ) \otimes_R R/I
\end{equation*}
(this is the trivial case of ``Localization of Hom-Sets'') so Nakayama's Lemma implies that (\ref{EadjointmapoverR}) must also be surjective.
\end{proof}

I return to the proof of Proposition \ref{Psmoothsurjections}. Let
\begin{equation*}
\widetilde{w}_1, \ldots, \widetilde{w}_d \in \widetilde{\mathcal{L}}_{0, R/I}
\end{equation*}
be the images (necessarily a basis) under $ \widetilde{g}_{R/I} $ of the standard basis. Let
\begin{equation*}
\widetilde{v}_1, \ldots, \widetilde{v}_d \in \widetilde{\mathcal{L}}_{0,R}
\end{equation*}
be lifts of $ \widetilde{w}_1, \ldots, \widetilde{w}_d $ such that $ (t+p) \widetilde{v}_i \in (t+p) \widetilde{\mathcal{L}}_{i, R} $ (this is possible because of the hypotheses on $ \widetilde{g} $). By Lemma \ref{Lpreliminaryfreeness}, $ \widetilde{v}_1, \ldots, \widetilde{v}_d $ is a \emph{basis} of $ \widetilde{\mathcal{L}}_{0,R} $. The normalized hermitian form $ t^{m-n} \phi_R : \mathcal{L}_0 \times \mathcal{L}_0 \rightarrow \mathcal{R}[t] $ descends to $ \widetilde{\mathcal{L}}_{0,R} \times \widetilde{\mathcal{L}}_{0,R} $ and takes values in $ \mathcal{S} $. By Lemma \ref{Lpreliminaryfreeness} again, it is \emph{perfect}.

Now that I have the basic ingredients of freeness and perfection, I can use the same method used to prove that $ \biggergroup{m}{n} \rightarrow \spec (\Z_p) $ was smooth (page \pageref{PsmoothJ}).

Let $ c \in \mathcal{S} $ be any representative of $ c( g_{R/I} ) $. Set
\begin{equation*}
C_{R/I} \defeq t^{m+n} (t+p)^{2\mu+2} c( g_{R/I} )
\end{equation*}
and
\begin{equation*}
C_R \defeq t^{m+n} (t+p)^{2\mu+2} c
\end{equation*}

Since
\begin{equation*}
\phi_{R/I} ( \widetilde{w}_i, \widetilde{w}_j ) = C_{R/I} \delta_{i, d+1-j} = \phi_{R/I} ( w_i, w_j )
\end{equation*}
there are $ x_{i,j} \in \mathcal{I}_t $ such that
\begin{equation} \label{Esmoothnesserrorterms}
\phi_R ( \widetilde{v}_i, \widetilde{v}_j ) = C_R \delta_{i, d+1-j} + x_{i,j} = \phi_R ( v_i, v_j )
\end{equation}

For each $ i $, use freeness to define an $ \mathcal{S} $-linear functional
\begin{equation*}
f_i : \widetilde{\mathcal{L}}_{0,R} \longrightarrow \mathcal{S}
\end{equation*}
by $ f_i ( \widetilde{v}_j ) = - \frac{1}{2} x_{i,j} $. Using perfection, there is an $ \widetilde{m}_i \in \widetilde{\mathcal{L}}_{0,R} $ such that
\begin{equation*}
f_i = t^{m-n} \phi_R ( -, \widetilde{m}_i ).
\end{equation*}
Let $ v_i $ and $ m_i $ be the images of $ \widetilde{v}_i $ and $ \widetilde{m}_i $ in $ \overline{\mathcal{L}}_0 $. Automatically, $ m_i \in \mathcal{I}_t \overline{\mathcal{W}}_{\textup{sup}}(R) $ (since $ \image (f_i) \subset \mathcal{I}_t \mathcal{S} $). It is automatic from definition that $ x_{j,i} = \jmath ( x_{i,j} ) $ so
\begin{equation*}
\phi_R ( \widetilde{v}_i + \widetilde{m}_i, \widetilde{v}_j + \widetilde{m}_j) = C_R \delta_{i, d+1-j} = \phi_R ( v_i + m_i, v_j + m_j )
\end{equation*}
I must verify that $ (t+p) m_i \in (t+p) \overline{\mathcal{L}}_i $ for each $ i $. The proof will then be finished by defining $ \widetilde{g}_R $ and $ g_R $ to be the maps sending the respective standard bases to $ \{ \widetilde{v}_i + \widetilde{m}_i \} $ and $ \{ t^{-(\mu + 1)} ( v_i + m_i ) \} $.

Since $ v_1, \ldots, v_i, (t+p) v_{i+1}, \ldots, (t+p) v_d $ generates $ (t+p) \overline{\mathcal{L}}_i $ for each $ i $, and since \BLMdef{\ref{BLMselfdual}} implies that
\begin{equation*}
\overline{\mathcal{L}}_i = \{ x \in \overline{\mathcal{W}}_{\textup{sup}}(R) \suchthat \phi_R ( x, \overline{\mathcal{L}}_{d-i} ) \subset t^{n-m} (t+p) \mathcal{S} \},
\end{equation*}
it suffices to show that
\begin{equation*}
\phi_R ( v_1, (t+p) m_i ), \ldots, \phi_R ( v_{d-i}, (t+p) m_i ) \in t^{n-m} (t+p)^2 \mathcal{S}
\end{equation*}
Note that the containments for $ \phi_R ( (t+p) v_j, (t+p) m_i ) $ are automatic since the defining relation (\ref{Esmoothnesserrorterms}) implies that $ x_{i,j} $, and therefore $ \phi_R ( v_j, m_i ) $, belongs to $ t^{n-m} \mathcal{S} $.

Since $ 1 \leq j \leq d-i $ implies that $ i+j \neq d+1 $, the defining equality (\ref{Esmoothnesserrorterms}) implies that
\begin{equation*}
\phi_R ( v_j, m_i ) = - \frac{1}{2} \phi_R ( v_j, v_i ) \textaftermath{$ j = 1, \ldots, d-i $}
\end{equation*}
It is now automatic from the above duality that $ \phi_R ( v_j, m_i ) $ satisfies the necessary condition.

The proof that $ \integralaffineflagvarmonoid{\mu}{\nu} \rightarrow \integralaffineflagvarargs{\mu}{\nu} $ is smooth is nearly identical.

Using Lemma \ref{Luntwistingconvolutionobject}, smoothness of $ p_2 $ is essentially a \emph{formal consequence} of smoothness of the individual factors of $ p_1 $. Let $ R $ be a local commutative $ \Z_p $-algebra and $ I \subset R $ a nilpotent ideal. I must show that for all
\begin{itemize}
\setlength{\itemsep}{5pt}
\item $ ( \mathcal{L}_0, \ldots \mathcal{L}_{d/2} ; \mathcal{K}_0, \ldots \mathcal{K}_{d/2} ) \in \convolutionobject{m}{n}{\mu}{\nu} ( R ) $

\item $ ( g_{R/I}, h_{R/I} ) \in \biggermonoid{m}{n} ( R / I ) \times \integralaffineflagvarmonoid{\mu}{\nu} ( R / I ) $
\end{itemize}
satisfying
\begin{align*}
g_{R/I} ( \ldots ) &= \overline{\mathcal{L}}_i \otimes_R R / I \\
g_{R/I} ( h_{R/I} ( \ldots ) ) &= \overline{\mathcal{K}}_i \otimes_R R / I
\end{align*}
there exists $ ( g_R, h_R ) \in \biggermonoid{m}{n} ( R ) \times \integralaffineflagvarmonoid{\mu}{\nu} ( R ) $ such that
\begin{align*}
g_R ( \ldots ) &= \overline{\mathcal{L}}_i \\
g_R ( h_R ( \ldots ) ) &= \overline{\mathcal{K}}_i
\end{align*}
and $ ( g_R, h_R ) \mapsto ( g_{R/I}, h_{R/I} ) $.

Invoke smoothness of $ \biggermonoid{m}{n} \rightarrow \biggermodel{m}{n} $ with the data $ \left\{ R, I, g_{R/I}, \mathcal{L}_{\bullet} \right\} $ to get $ g_R $. Let $ ( \mathcal{F}_0, \ldots, \mathcal{F}_{d/2} ) \in \integralaffineflagvarargs{\mu}{\nu} (R) $ be the point guaranteed by Lemma \ref{Luntwistingconvolutionobject} (page \pageref{Luntwistingconvolutionobject}). Invoke smoothness of $ \integralaffineflagvarmonoid{\mu}{\nu} \rightarrow \integralaffineflagvarargs{\mu}{\nu} $ with the data $ \{ R, I, h_{R/I}, \mathcal{F}_{\bullet} \} $ to get $ h_R $. By Lemma \ref{Luntwistingconvolutionobject},
\begin{equation*}
g_R ( h_R ( t^{-m} (t+p)^{-\mu} \overline{\mathcal{V}}_i(R) ) ) = \overline{\mathcal{K}}_i
\end{equation*}
so this pair $ ( g_R, h_R ) $ satisfies the requirements.

\framebox{surjectivity} Fix a $ \Z_p $-field $ K $ and set $ \mathcal{K} := K \otimes_{\Z_p} \mathcal{O} $. Assume first that $ \mychar ( K ) = p $, and note that $ \mathcal{K} = K \otimes_{\F_p} \F $ also. Fix $ \mathcal{L}_{\bullet} \in \biggermodel{m}{n} (K) $. By the embedding $ \biggermodel{m}{n}_{\F_p} \hookrightarrow \affineflagvariety $, there is a lift $ \mathcal{F}_{\bullet} \in \affineflagvariety ( K ) $ of $  \mathcal{L}_{\bullet} $ which, by \AFVdef{}, is a polarized lattice chain in the sense of A.41 of \cite{RZ}. By Proposition A.43 of \cite{RZ}, there are\footnote{The necessary content of the Appendix of \cite{RZ}, especially Proposition A.43, applies in the equi-characteristic case $ \F((t)) / \F_p((t)) $ here despite the use of the mixed-characteristic case in \cite{RZ}. The dualizing shift ``$ a $'' of \cite{RZ} depends on the multiplier from \AFVdef{\ref{AFVsimduality}}. } mutually compatible $ \mathcal{K}[[t]] $-linear isomorphisms $ g_i : t^{-(m+1)} \mathcal{K}[[t]]^i \oplus t^{-m} \mathcal{K}[[t]]^{d-i} \directedisom \mathcal{F}_i $ such that $ t^{2m+1} \Phi_K ( x, y ) = t^{m-n+1} \Phi_K ( g_i ( x ), g_i ( y ) ) $ for all $ x, y $ in the domain of $ g_i $. By $ \mathcal{K}[t] $-linearity, $ g_i ( t^{n+\mu+\nu+1} \mathcal{K}[[t]]^d ) \subset t^{m+n+\mu+\nu+1} \mathcal{F}_i $ and so $ g_0 $ descends to a $ \mathcal{K}[[t]] $-linear map $ \widetilde{g} : t^{-m} \widetilde{\mathcal{V}}_0 ( K ) \rightarrow \widetilde{\mathcal{L}}_0 $ (see the definition of $ \biggermonoid{m}{n} $ for these objects). Counting $ K $-dimensions shows that $ \widetilde{g} $ is an isomorphism. In other words, $ \widetilde{g} $ satisfies (\ref{Dbiggermonoidlifting}) of the definition of $ \biggermonoid{m}{n} $. Since $ t^{-m} \widetilde{\mathcal{V}}_i $ is simply a ``shift'' of $ \overline{\mathcal{W}}_{\SUP} $ by $ (t+p)^{\mu+1} $, this $ \widetilde{g} $ induces a $ \mathcal{K}[t] $-linear map $ g : \overline{\mathcal{W}}_{\SUP} \rightarrow \overline{\mathcal{W}}_{\SUP} $ and it is clear from the properties of $ \widetilde{g} $ that $ g \in \biggermonoid{m}{n} ( K ) $. By construction, the image of $ g $ in $ \biggermodel{m}{n} ( K ) $ is $ \mathcal{L}_{\bullet} $. If $ \mychar ( K ) = 0 $ then the above ideas can be used in combination with the Chinese Remainder Theorem (as in the proof of Lemma \ref{Lgenericfiberisom} (page \pageref{Lgenericfiberisom})) to prove surjectivity of $ \biggermonoid{m}{n} ( K ) \rightarrow \biggermodel{m}{n} ( K ) $. A nearly identical proof shows that $ \integralaffineflagvarmonoid{\mu}{\nu} ( K ) \rightarrow \integralaffineflagvarargs{\mu}{\nu} ( K ) $. Together, this establishes surjectivity of $ p_1 $. By the use of Lemma \ref{Luntwistingconvolutionobject}, surjectivity of $ p_2 ( K ) $ is essentially a formal consequence: For $ ( \mathcal{L}_{\bullet} ; \mathcal{K}_{\bullet} ) \in \convolutionobject{m}{n}{\mu}{\nu} ( K ) $, invoke surjectivity of $ \biggermonoid{m}{n} ( K ) \rightarrow \biggermodel{m}{n} ( K ) $ to get $ g $, use Lemma \ref{Luntwistingconvolutionobject} (page \pageref{Luntwistingconvolutionobject}) to get $ \mathcal{F}_{\bullet} \in \integralaffineflagvarargs{\mu}{\nu} ( K ) $, and invoke surjectivity of $ \integralaffineflagvarmonoid{\mu}{\nu} ( K ) \rightarrow \integralaffineflagvarargs{\mu}{\nu} ( K ) $ to get $ h $. It follows that that $ p_2 ( g, h ) = ( \mathcal{L}_{\bullet} ; \mathcal{K}_{\bullet} ) $.
\end{proof}

\begin{remark}
Uses of Proposition \ref{Psmoothsurjections}:
\begin{itemize}
\setlength{\itemsep}{5pt}
\item Smoothness of $ p_1 $ is used in \S\ref{SSfusionproductdef} (page \pageref{SSfusionproductdef}) and in the proof of Lemma 23 in \cite{HN} (which I invoke).

\item Smoothness of $ p_2 $ is used to satisfy the hypotheses of Lemma 21 in \cite{HN} (page \pageref{labelHNdescendsheaf} here) and in the proof of Lemma 23 in \cite{HN} (which I invoke).

\item The surjectivity statement for $ p_1 $ is not used.

\item The surjectivity statement for $ p_2 $ implies that the underlying map of topological spaces is surjective, which is used to satisfy the hypotheses of Lemma 21 in \cite{HN} (page \pageref{labelHNdescendsheaf} here).
\end{itemize}
\end{remark}

\subsection{The special fiber of the convolution diagram} \label{SSdiagramcollapsemodp}

Two obvious but important simplifications occur over $ \F_p $ in the convolution diagram.

First, if $ ( m, n ) = ( \mu , \nu ) $ then the definitions of $ \biggermodel{m}{n} $ and $ \integralaffineflagvarargs{m}{n} $ are equal modulo $ p $, i.e.
\begin{equation*}
\integralaffineflagvarargs{m}{n}_{\F_p} = \biggermodel{m}{n}_{\F_p}
\end{equation*}

Similarly,
\begin{align*}
\biggermonoid{m}{n}_{\F_p} &= \integralaffineflagvarmonoid{m}{n}_{\F_p} \\
\biggergroup{m}{n}_{\F_p} &= \integraliwahoriargs{m}{n}_{\F_p}
\end{align*}

Second, for any $ m, n, \mu, \nu \in \N $ it is immediate by looking at the definition that
\begin{equation*}
\convolutiontarget{m}{n}{\mu}{\nu}_{\F_p} = \integralaffineflagvarmonoid{m+\mu}{n+\nu}_{\F_p}
\end{equation*}

These observations are important for understanding (see \S\ref{SSfusionproductcategorifies}) why the convolution product of sheaf complexes (not yet defined) induces the convolution product of functions.

\subsection{The generic fiber of the convolution diagram}

The following result (an almost identical copy of Lemma 24 from \cite{HN}) is unique to the generic fiber, because of the fact that $ (t) $ and $ (t+p) $ are \emph{comaximal} in $ \Q_p[t] $ and therefore the Chinese Remainder Theorem can be applied.
\begin{lemma} \label{Lgenericfiberisom}
The extended morphisms $ m_{\Q_p} $ and $ \presuperscript{\textup{rev}}{m}_{\Q_p} $ are
isomorphisms, and there are isomorphisms $ i, \presuperscript{\textup{rev}}{i} $ such that
the square formed by these 4 morphisms is (after allowing the morphisms to be inverted) commutative:
\begin{equation*}
\begin{CD}
\biggermodel{m}{n}_{\Q_p} \times \integralaffineflagvarargs{\mu}{\nu}_{\Q_p} @<i<< \convolutionobject{m}{n}{\mu}{\nu}_{\Q_p} \\
@A\presuperscript{\textup{rev}}{i}AA @VVmV \\
\reversedconvolutionobject{\mu}{\nu}{m}{n}_{\Q_p} @>>\presuperscript{\textup{rev}}{m}> \convolutiontarget{m}{n}{\mu}{\nu}_{\Q_p}
\end{CD}
\end{equation*}
\end{lemma}

\begin{proof}
Because the ideals $ (t) $ and $ (t+p) $ are comaximal in $ F[t] $, the Chinese remainder theorem implies that the $ F[t] $-module $ \overline{\mathcal{W}}_{\textup{sup}}(\Q_p) $ can be written as the direct sum
\begin{equation} \label{EQchineseremaindertheorem}
\overline{\mathcal{W}}_{\textup{sup}}(\Q_p) \cong \frac{ t^{-m} F[t]^d }{ t^n F[t]^d } \oplus \frac{ (t+p)^{-\mu-1} F[t]^d }{ (t+p)^{\nu} F[t]^d }
\end{equation}
Similarly decompose each $ \overline{\mathcal{V}}_i(\Q_p) $ into
\begin{equation*}
\overline{\mathcal{V}}_i(\Q_p) \cong \overline{\mathcal{V}}_i^{(t)} (\Q_p) \oplus \overline{\mathcal{V}}_i^{(t+p)} (\Q_p)
\end{equation*}
Let $ R $ be a commutative $ \Q_p $-algebra. Take
\begin{equation*}
( \mathcal{L}_0, \ldots, \mathcal{L}_{d/2} ; \mathcal{K}_0, \ldots, \mathcal{K}_{d/2} ) \in \convolutionobject{m}{n}{\mu}{\nu} (R).
\end{equation*}
Denote by $ \overline{\mathcal{L}}_i $ and $ \overline{\mathcal{K}}_i $ the images in $ \overline{\mathcal{W}}_{\textup{sup}} (R) $. In particular,
\begin{align}
\label{EQlocallabel1} t^n \overline{\mathcal{V}}_i(R) \subset
\overline{\mathcal{L}}_i \subset t^{-m}
\overline{\mathcal{V}}_i(R) \\
\label{EQlocallabel2} (t+p)^{\nu} \overline{\mathcal{L}}_i \subset
\overline{\mathcal{K}}_i \subset (t+p)^{-\mu}
\overline{\mathcal{L}}_i
\end{align}
Decompose each
\begin{align*}
\overline{\mathcal{L}}_i &\cong \overline{\mathcal{L}}_i^{(t)}
\oplus \overline{\mathcal{L}}_i^{(t+p)} \\
\overline{\mathcal{K}}_i &\cong \overline{\mathcal{K}}_i^{(t)}
\oplus \overline{\mathcal{K}}_i^{(t+p)} \\
\end{align*}
Since the images under the 2nd projection from (\ref{EQchineseremaindertheorem}) of $ t^k \overline{\mathcal{V}}_i(R) $ is the same, always equal to $ ( (t+p)^{-1} \mathcal{R}[t] / \mathcal{R}[t] )^i $, regardless of $ k \in \Z $, the inclusions in (\ref{EQlocallabel1}) force
\begin{equation} \label{EQlocallabel3}
\overline{\mathcal{L}}_i^{(t+p)} = \overline{\mathcal{V}}_i^{(t+p)}(R)
\end{equation}
Similarly, applying the 1st projection of from (\ref{EQchineseremaindertheorem}) to the inclusions in (\ref{EQlocallabel2}) shows that
\begin{equation} \label{EQlocallabel4}
\overline{\mathcal{K}}_i^{(t)} = \overline{\mathcal{L}}_i^{(t)}
\end{equation}
In other words, the function
\begin{equation*}
m_R ( \mathcal{L}_0, \ldots, \mathcal{L}_{d/2} ; \mathcal{K}_0, \ldots, \mathcal{K}_{d/2} ) = ( \mathcal{K}_0, \ldots, \mathcal{K}_{d/2} )
\end{equation*}
is injective. Conversely, for any $ ( \mathcal{K}_0, \ldots, \mathcal{K}_{d/2} ) \in \convolutiontarget{m}{n}{\mu}{\nu} (R) $, \emph{defining} $ ( \mathcal{L}_0, \ldots, \mathcal{L}_{d/2} ) $ by (\ref{EQlocallabel3}) and (\ref{EQlocallabel4}) yields a point of $ \convolutionobject{m}{n}{\mu}{\nu} (R) $. So $ m_R $ is an \emph{isomorphism} for any commutative $ \Q_p $-algebra $ R $.

Take $ ( \mathcal{L}_0, \ldots, \mathcal{L}_{d/2} ; \mathcal{K}_0, \ldots, \mathcal{K}_{d/2} ) \in \convolutionobject{m}{n}{\mu}{\nu} (R) $. Using the preceding argument, write $ ( \mathcal{K}_0, \ldots, \mathcal{K}_{d/2} ) $ as
\begin{equation*}
( \mathcal{L}_0^{(t)} \oplus \mathcal{K}_0^{(t+p)}, \ldots, \mathcal{L}_{d/2}^{(t)} \oplus \mathcal{K}_{d/2}^{(t+p)} )
\end{equation*}
I claim that the chain $ \overline{\mathcal{K}}_i^{(t+p)} $ (i.e. discarding the 1st summand from each $ \mathcal{K}_i $) is an element of $ \integralaffineflagvarargs{\mu}{\nu} (R) $. This is clear by applying the 2nd projection from (\ref{EQchineseremaindertheorem}) to the inclusions in (\ref{EQlocallabel2}) and then using the equality in (\ref{EQlocallabel3}) (this shows that $ \overline{\mathcal{K}}_i^{(t+p)} $ has the necessary bounds, and the other properties are automatic from the definitions).

I claim that the function
\begin{align*}
i_R : \convolutionobject{m}{n}{\mu}{\nu} (R) &\longrightarrow \biggermodel{m}{n} (R) \times \integralaffineflagvarargs{\mu}{\nu} (R) \\
( \mathcal{L}_0, \ldots, \mathcal{L}_{d/2} ; \mathcal{K}_0, \ldots, \mathcal{K}_{d/2} ) &\longmapsto \left( ( \overline{\mathcal{L}}_0^{(t)}, \ldots, \overline{\mathcal{L}}_{d/2}^{(t)} ), ( \overline{\mathcal{K}}_0^{(t+p)}, \ldots, \overline{\mathcal{K}}_{d/2}^{(t+p)} ) \right)
\end{align*}
is a bijection. Injectivity is obvious because the discarded summand $ \mathcal{L}_i^{(t)} $ in $ \mathcal{K}_i $ is not truly discarded by $ i $: it is retained by the 1st coordinate. Surjectivity is also obvious for the same reason: for any $ ( ( \overline{\mathcal{L}}_i ), ( \overline{\mathcal{F}}_i ) ) \in \biggermodel{m}{n} (R) \times \integralaffineflagvarargs{\mu}{\nu} (R) $, simply supply the respective missing summands $ \overline{\mathcal{V}}_i^{(t+p)}(R) $ and $ \overline{\mathcal{L}}_i^{(t)} $.

The proof for $ \presuperscript{\textup{rev}}{m}_{\Q_p} $ and $ \presuperscript{\textup{rev}}{i} $ is nearly identical, and the fact that the square commutes is then obvious.
\end{proof}

\subsection{An automorphism group for the convolution diagram} \label{SSuniversalgroupaction}

I also need a $ \Z_p $-group $ \universalgroup{m}{n}{\mu}{\nu} $ that acts on both $ \biggermonoid{m}{n} $ and $ \integralaffineflagvarmonoid{\mu}{\nu} $ and factors through both $ \biggergroup{m}{n} $ and $ \integraliwahoriargs{\mu}{\nu} $. The definition is a straightforward enlargement of the definition of $ \biggergroup{m}{n} $ plus a lifting condition: 

\begin{definition}
The functor
\begin{equation*}
\universalgroup{m}{n}{\mu}{\nu} : \catalgebras{\Z_p} \longrightarrow \catsets
\end{equation*}
assigns to each commutative $ \Z_p $-algebra $ R $ the set of all $ \mathcal{R}[t] $-linear automorphisms $ \gamma $ of $ \overline{\mathcal{W}}_{\textup{sup}} (R) $ satisfying:
\begin{itemize}
\setlength{\itemsep}{5pt}
\item $ \gamma ( \overline{\mathcal{V}}_i(R) ) = \overline{\mathcal{V}}_i(R) $ for all $ i $

\item there exists $ c(g) \in R[t] $ representing a \emph{unit} of $ R[t] / t^{m+n} (t+p)^{\mu+\nu+1} R[t] $ such that $ \overline{\phi}_R ( \gamma ( x ), \gamma ( y ) ) = c(g)  \overline{\phi}_R ( x, y ) $ for all $ x, y \in \overline{\mathcal{W}}_{\textup{sup}} (R) $.

\item $ \gamma $ is induced, Zariski-locally\footnote{The meaning of ``Zariski-locally'' here is as in the definition of $ \biggermonoid{m}{n} $.} on $ \spec(R) $, by some $ \mathcal{R}[t] $-linear automorphism $ \widetilde{\gamma} $ of $ t^{-m} (t+p)^{-\mu-1} \mathcal{R}[t]^d / t^{m+n} (t+p)^{\mu+\nu} \mathcal{R}[t]^d $ (of which $ \overline{\mathcal{W}}_{\textup{sup}} (R) $ is a quotient)
\end{itemize}
\end{definition}

Notice that any $ \gamma \in \universalgroup{m}{n}{\mu}{\nu} (R) $ restricts and descends from $ \overline{\mathcal{W}}_{\textup{sup}}(R) $ to $ \mathcal{R}[t] $-linear automorphisms $ \gamma_{\mathcal{V}} $ of $ \overline{\mathcal{V}}_{\textup{sup}} (R) $ and $ \gamma_{\mathcal{U}} $ of $ \overline{\mathcal{U}}_{\textup{sup}} (R) $. It is clear that these induced automorphisms $ \gamma_{\mathcal{V}}, \gamma_{\mathcal{U}} $ are similitudes for the appropriate forms $ \overline{\phi}_R $, and the multiplier $ c(\gamma) $ specified from $ \universalgroup{m}{n}{\mu}{\nu} (R) $ is also appropriate, so $ \gamma \mapsto \gamma_{\mathcal{V}} $ and $ \gamma \mapsto \gamma_{\mathcal{U}} $ define morphisms
\begin{align*}
\universalgroup{m}{n}{\mu}{\nu} &\longrightarrow \biggergroup{m}{n} \\
\universalgroup{m}{n}{\mu}{\nu} &\longrightarrow \integraliwahoriargs{\mu}{\nu}
\end{align*}
of $ \Z_p $-group schemes.

%\begin{remark}The purpose of this group is to express equivariance properties of sheaves uniformly regardless of which object in the convolution diagram supports the sheaves.\end{remark}

By the exact same process used in \S\ref{SSbruhattitsdecomp} (page \pageref{SSbruhattitsdecomp}) to define $ \iwahori ( \F_p ) \rightarrow \biggergroup{m}{n} ( \F_p ) $, one has a group homomorphism
\begin{equation*}
\iwahori ( \F_p ) \longrightarrow \universalgroup{m}{n}{\mu}{\nu} ( \F_p )
\end{equation*}
(note that the existence of ``$ \widetilde{\gamma} $'' is trivial) such that the composition
\begin{equation*}
\iwahori ( \F_p ) \rightarrow \universalgroup{m}{n}{\mu}{\nu} ( \F_p ) \rightarrow \biggergroup{m}{n} ( \F_p )
\end{equation*}
is exactly the group homomorphism from \S\ref{SSbruhattitsdecomp}.

Two group actions, one tailored to $ p_1 $ and one tailored to $ p_2 $, are needed to define the convolution product:

\begin{definition}
Define the group action $ \alpha_1 $ of $ \universalgroup{m}{n}{\mu}{\nu} \times \universalgroup{m}{n}{\mu}{\nu} $ on $ \biggermonoid{m}{n} \times \integralaffineflagvarmonoid{\mu}{\nu} $ by the rule
\begin{equation*}
\alpha_1 ( \gamma, \eta ; g, h ) \defeq ( g \circ \gamma^{-1}, h \circ \eta^{-1} )
\end{equation*}

Define another action $ \alpha_2 $ by the rule 
\begin{equation*}
\alpha_2 ( \gamma, \eta ; g, h ) \defeq ( g \circ \gamma^{-1}, \gamma \circ h \circ \eta^{-1} )
\end{equation*}
\end{definition}

It is perhaps helpful to briefly explain why these are legitimate actions on $ \biggermonoid{m}{n} \times \integralaffineflagvarmonoid{\mu}{\nu} $. Take $ g \in \biggermonoid{m}{n} ( R ) $. Since any $ \gamma \in \universalgroup{m}{n}{\mu}{\nu} ( R ) $ is a similitude of $ \overline{\mathcal{W}}_{\textup{sup}} (R) $ and stabilizes all $ \overline{\mathcal{V}}_i $, it is clear that $ g \circ \gamma $ satisfies all the conditions of $ \biggermonoid{m}{n} ( R ) $ except possibly the lifting property (the existence of a certain ``$ \widetilde{g \circ \gamma} $''). For this, note that $ t^{-m} (t+p)^{-\mu-1} \mathcal{R}[t]^d / t^{m+n} (t+p)^{\mu+\nu} \mathcal{R}[t]^d $ is a subquotient of the domain of $ \widetilde{\gamma} $, and define the desired lift of $ g \circ \gamma $ by composing $ \widetilde{g} $ with the automorphism induced by $ \widetilde{\gamma} $ on that subquotient. By a nearly identical argument, $ h \circ \gamma \in \integralaffineflagvarmonoid{\mu}{\nu} ( R ) $ for any $ h \in \integralaffineflagvarmonoid{\mu}{\nu} ( R ) $ and $ \gamma \in \universalgroup{m}{n}{\mu}{\nu} ( R ) $.

Note that the action via $ \alpha_1 $ stabilizes fibers of $ p_1 $, since elements $ \universalgroup{m}{n}{\mu}{\nu} $ stabilize all $ \overline{\mathcal{V}}_i $. Similarly, the action via $ \alpha_2 $ stabilizes fibers of $ p_2 $ since the 2nd coordinate of $ p_2 $ uses $ ( g \circ \gamma^{-1} ) \circ ( \gamma \circ h \circ \eta^{-1} ) = g \circ h \circ \eta^{-1} $.

For the convenience of the reader, here is the customary notion of ``fixer'' for the above two actions:
\begin{definition}
If $ A $ is a $ \Z_p $-algebra and and $ g, h $ are $ A $-points then the functor
\begin{equation*}
\fixer_2 ( g, h ) : \catalgebras{A} \longrightarrow \catgroups
\end{equation*}
assigns to any $ A $-algebra $ R $ the subgroup of all $ ( \gamma, \eta ) $ such that $ g_R \circ \gamma^{-1} = g_R $ and $ \gamma \circ h_R \circ \eta^{-1} = h_R $. The fixers
\begin{equation*}
\fixer ( g ), \fixer ( h ), \fixer_1 ( g, h ) : \catalgebras{A} \longrightarrow \catgroups
\end{equation*}
are defined similarly.
\end{definition}

\begin{prop} \label{Puniversalgroupisfinitetype}
$ \universalgroup{m}{n}{\mu}{\nu} $ is a finite-type $ \Z_p $-scheme.
\end{prop}

\begin{proof}
The proof is nearly identical (in fact easier since the codomain of the Zariski-local lifts is not varying) to that given for $ \biggermonoid{m}{n} $ (page \pageref{Pbiggermonoidisrepresentable}).
\end{proof}

\subsection{The automorphism group is smooth}

\begin{prop} \label{PsmoothuniversalJ}
$ \universalgroup{m}{n}{\mu}{\nu} $ is a smooth $ \Z_p $-scheme.
\end{prop}

\begin{proof}
Let $ \universalgroup{m}{n}{\mu}{\nu}_{\textup{weak}} $ be the $ \Z_p  $-group scheme defined using only the 1st and 2nd conditions (i.e. \emph{excluding} the Zariski-local lifting condition). It is obvious that $ \universalgroup{m}{n}{\mu}{\nu}_{\textup{weak}} $ is finite-type, so to show that $ \universalgroup{m}{n}{\mu}{\nu}_{\textup{weak}} $ is smooth, it suffices to verify the infinitesimal lifting property (page \pageref{formalsmoothness}), and for this, the proof that $ \biggergroup{m}{n} $ is smooth, Proposition \ref{PsmoothJ} (page \pageref{PsmoothJ}), works almost verbatim: for a local commutative $ \Z_p $-algebra $ R $ and an ideal $ I \subset R $ satisfying $ I^2 = 0 $, simply use
\begin{itemize}
\setlength{\itemsep}{5pt}
\item $ M := \overline{\mathcal{W}}_{\textup{sup}} (R) $
\item $ S := R[t] / t^{m+n} (t+p)^{\mu+\nu+1} R[t] $
\item $ \overline{\phi} $ to be the product on $ \overline{\mathcal{W}}_{\textup{sup}} $
\item continue to use $ \sigma := (t+p) $
\end{itemize}
and define all ideals and submodules as before using the new objects just listed. This proves that $ \universalgroup{m}{n}{\mu}{\nu}_{\textup{weak}} \rightarrow \spec(\Z_p) $ is smooth. By Proposition \ref{Puniversalgroupisfinitetype}, $ \universalgroup{m}{n}{\mu}{\nu} $ is finite-type, so it again suffices to verify the infinitesimal lifting property. Suppose
$ \gamma_{R/I} \in \universalgroup{m}{n}{\mu}{\nu} ( R / I ) $. Since $ \universalgroup{m}{n}{\mu}{\nu} ( R / I ) \subset \universalgroup{m}{n}{\mu}{\nu}_{\textup{weak}} ( R / I ) $, smoothness implies that there is $ \gamma_R \in \universalgroup{m}{n}{\mu}{\nu}_{\textup{weak}} ( R ) $ such that $ \gamma_R \mapsto \gamma_{R/I} $. Since $ R $ is local, to show that in fact $ \gamma_R \in \universalgroup{m}{n}{\mu}{\nu} ( R ) $, I must show the existence of $ \widetilde{\gamma}_R $ globally on $ \spec(R) $. Let $ \widetilde{\gamma}_R $ be an arbitrary lift of $ \gamma_R $ to an $ \mathcal{R}[t] $-linear endomorphism of $ t^{-m} (t+p)^{-\mu-1} \mathcal{R}[t]^d / t^{m+n} (t+p)^{\mu+\nu} \mathcal{R}[t]^d $. Since $ \gamma_{R/I} \in \universalgroup{m}{n}{\mu}{\nu} ( R / I ) $ (note that $ \widetilde{\gamma}_{R/I} $ exists globally on $ \spec(R) $), Nakayama's Lemma implies that $ \gamma_{R/I} $ is \emph{surjective}, and since the domain and codomain are the same rank $ d $ free module, ``Linear Independence of Minimal Generating Sets'' then implies that $ \gamma_{R/I} $ is an \emph{automorphism}.
\end{proof}

\subsection{Properties of the actions on the convolution diagram}

To deduce statements about $ \alpha_2 $ from statements about $ \alpha_1 $, it is helpful to isolate and record the following:

\begin{lemma} \label{Lactionhelper}
Let $ R $ be a \emph{local} commutative $ \Z_p $-algebra. Fix $ g \in \biggermonoid{m}{n} ( R ) $ and $ h \in \integralaffineflagvarmonoid{\mu}{\nu} ( R ) $. If $ \gamma \in \universalgroup{m}{n}{\mu}{\nu} ( R ) $ is such that $ g \circ \gamma = g $ then $ \gamma \circ h $ and $ h $ are in the same fiber of $ \integralaffineflagvarmonoid{\mu}{\nu} \rightarrow \integralaffineflagvarargs{\mu}{\nu} $.
\end{lemma}

\begin{proof}
Let $ \mathcal{F}_{\bullet} \in \integralaffineflagvarargs{\mu}{\nu} ( R ) $ be the image of $ h $. Lifting $ \mathcal{F}_0 $ to $ \mathcal{V} ( R ) = \mathcal{R} [ t, t^{-1}, (t+p)^{-1} ]^d $ and scaling by $ t^{-m} $ yields a submodule of $ t^{-m} (t+p)^{-\mu} \mathcal{R}[t]^d $ which then descends to a $ \mathcal{R}[t] $-submodule $ \mathcal{F}_0^{+} \subset \overline{\mathcal{W}}_{\SUP} ( R ) $. Note that $ h ( t^{-m} (t+p)^{-\mu} \mathcal{R}[t]^d ) = \mathcal{F}_0^{+} $ and that $ t^{-m} (t+p)^{\nu} \mathcal{R}[t]^d \subset \mathcal{F}_0^{+} $. I claim first that $ \kernel ( g ) \subset \mathcal{F}_0^{+} $. By the previous containment, it suffices to show that $ \kernel ( g ) \subset t^{-m} (t+p)^{\nu} \mathcal{R}[t]^d $. By definition of $ \overline{\mathcal{W}}_{\SUP} $, it is equivalent to show that $ t^{m+n} \cdot \kernel ( g ) = 0 $ inside $ \overline{\mathcal{W}}_{\SUP} ( R ) $. Let $ \widetilde{g} $ be a lift of $ g $ (recall that $ R $ is local) as guaranteed by the definition of $ \biggermonoid{m}{n} $, so that $ \kernel ( g ) = \widetilde{g}^{-1} ( t^n (t+p)^{\nu} \mathcal{R}[t]^d ) $. Let $ \mathcal{L}_{\bullet} \in \biggermonoid{m}{n} ( R ) $ be the image of $ g $. By $ \mathcal{R}[t] $-linearity of $ \widetilde{g} $, it is equivalent to show that $ t^{m + 2n} (t+p)^{\nu} \mathcal{R}[t]^d \subset t^{m+n} (t+p)^{\nu} \mathcal{L}_0 $, but this is immediate by \BLMdef{\ref{BLMbounds}}. Now return to the main result. By hypothesis on $ \gamma $, it is true that $ \gamma ( h ( v ) ) - h ( v ) \in \kernel ( g ) $ for all $ v \in t^{-m} (t+p)^{-\mu} \mathcal{R}[t]^d $. By what was already proved, $ \gamma ( h ( t^{-m} (t+p)^{-\mu} \mathcal{R}[t]^d ) ) = h ( t^{-m} (t+p)^{-\mu} \mathcal{R}[t]^d ) $, and the proof is finished by scaling both sides by $ t^m $.
\end{proof}

The following is the analogue of Lemma 20 from \cite{HN}:

\begin{prop} \label{Ptransitiveactions}
Let $ K $ be a $ \Z_p $-field.
\begin{itemize}
\item The action $ \alpha_1 $ (resp. $ \alpha_2 $) is transitive on each fiber of $ p_1 $ (resp. $ p_2 $) over a $ K $-point $ ( g, h ) $.

\item The fixers $ \fixer_1 ( g, h ) $ and $ \fixer_2 ( g, h ) $ are smooth and connected for all $ K $-points $ ( g, h ) $.
\end{itemize}
\end{prop}

\begin{proof}
\framebox{transitivity} Let $ R $ be a \emph{local} commutative $ \Z_p $-algebra. Suppose $ g, h \in \biggermonoid{m}{n} ( R ) $ are in the same fiber of $ \biggermonoid{m}{n} \rightarrow \biggermodel{m}{n} $. Let $ \widetilde{g} $ and $ \widetilde{h} $ be the (global!) lifts of $ g $ and $ h $ guaranteed by the definition of $ \biggermonoid{m}{n} $. The assumption that $ g $ and $ h $ are in the same fiber means that both $ \widetilde{g} $ and $ \widetilde{h} $ have the same codomain and so $ \widetilde{g}^{-1} \circ \widetilde{h} $ is an $ \mathcal{R}[t] $-linear automorphism of $ t^{-m} \widetilde{\mathcal{V}}_0 ( R ) $. Since this module is simply a ``shift'' of $ \overline{\mathcal{W}}_{\SUP} ( R ) $ by $ (t+p)^{\mu+1} $, the automorphism $ \widetilde{g}^{-1} \circ \widetilde{h} $ can be considered as an $ \mathcal{R}[t] $-linear automorphism $ \gamma $ of $ \overline{\mathcal{W}}_{\SUP} ( R ) $. By the part of property (\ref{Dbiggermonoidlifting}) of $ \widetilde{g} $ that concerns chains and the fact that $ g $ and $ h $ are in the same fiber, it follows that $ \gamma $ stabilizes the chain $ \overline{\mathcal{V}}_{\bullet} ( R ) $. Finally, it is immediate that $ \gamma $ has the required similitude property, with multiplier $ c(g)^{-1} \cdot c(h) \in ( R[t] / t^{m+n} (t+p)^{\mu+\nu+1} R[t] )^{\times} $. Thus, $ \gamma \in \universalgroup{m}{n}{\mu}{\nu} ( R ) $ and $ g \circ \gamma = h $. A nearly identical proof shows transitivity after replacing $ \biggermonoid{m}{n} \rightarrow \biggermodel{m}{n} $ with $ \integralaffineflagvarmonoid{\mu}{\nu} \rightarrow \integralaffineflagvarargs{\mu}{\nu} $, so this establishes transitivity of the action of $ \universalgroup{m}{n}{\mu}{\nu} ( R ) \times \universalgroup{m}{n}{\mu}{\nu} ( R ) $ via $ \alpha_1 $ on fibers of $ p_1 $ over points of $ \biggermonoid{m}{n} ( R ) $ whenever $ R $ is local.

Now consider $ \alpha_2 $ and $ p_2 $. Let $ R $ be a local commutative $ \Z_p $-algebra. If $ g_1, h_1 \in \biggermonoid{m}{n} ( R ) $ and $ g_2, h_2 \in \integralaffineflagvarmonoid{\mu}{\nu} ( R ) $ are such that $ p_2 ( g_1, g_2 ) = p_2 ( h_1, h_2 ) $ then the goal is to find $ \gamma, \eta \in \universalgroup{m}{n}{\mu}{\nu} ( R ) $ such that $ g_1 \circ \gamma^{-1} = h_1 $ and $ \gamma \circ g_2 \circ \eta^{-1} = h_2 $. By definition of $ p_2 $, the assumed equality forces that $ g_1 $ and $ h_1 $ are in the same fiber of $ \biggermonoid{m}{n} \rightarrow \biggermodel{m}{n} $. By transitivity on fibers of $ p_1 $ (proved already), there is $ \gamma \in \universalgroup{m}{n}{\mu}{\nu} ( R ) $ such that $ g_1 \circ \gamma^{-1} = h_1 $. By this equality and the definition of $ p_2 $, both $ g_1 \circ g_2 = h_1 \circ \gamma \circ g_2 $ and $ h_1 \circ h_2 $ have the same image on $ t^{-m} (t+p)^{-\mu} \mathcal{R}[t]^d $. Lemma \ref{Lactionhelper} then implies\footnote{The equality in hand, $ h_1 ( \gamma ( g_2 ( t^{-m} (t+p)^{-\mu} \mathcal{R}[t]^d ) ) ) = h_1 ( h_2 ( t^{-m} (t+p)^{-\mu} \mathcal{R}[t]^d ) ) $, is logically weaker than the element-wise equality that was used to conclude the proof of Lemma \ref{Lactionhelper} but is nonetheless sufficient for that conclusion.} that $ \gamma \circ g_2 $ and $ h_2 $ are in the same fiber of $ \integralaffineflagvarmonoid{\mu}{\nu} \rightarrow \integralaffineflagvarargs{\mu}{\nu} $ so transitivity on fibers of $ p_1 $ yields $ \eta \in \universalgroup{m}{n}{\mu}{\nu} ( R ) $ such that $ \gamma \circ g_2 \circ \eta^{-1} = h_2 $.

\framebox{connectedness, smoothness} Let $ K $ be a $ \Z_p $-field and set $ \mathcal{K} := K \otimes_{\Z_p} \mathcal{O} $. Assume first that $ \mychar ( K ) = p $ and note that $ \mathcal{K} = K \otimes_{\F_p} \F $ also (this uses that $ F / \Q_p $ is unramified). Fix $ g \in \biggermonoid{m}{n} ( K ) $ and let $ \mathcal{L}_{\bullet} \in \biggermodel{m}{n} ( K ) $ be the image via $ \biggermonoid{m}{n} \rightarrow \biggermodel{m}{n} $. Suppose $ \gamma \in \universalgroup{m}{n}{\mu}{\nu} ( K ) $ is such that $ g \circ \gamma^{-1} = g $, which is the same as to say: $ \gamma $ belongs to the $ K[t] / t^N K[t] $-points of the unitary similitude group $ \mathbb{G} $ of the hermitian form $ \overline{\phi}_{\F_p} $ ($ N := m + n + \mu + \nu + 1 $), stabilizes the standard ``lattice'' chain $ \overline{\mathcal{V}}_{\bullet} ( K ) $, and satisfies $ \gamma ( v ) - v \in \kernel ( g ) $ for all $ v \in \overline{\mathcal{W}}_{\SUP} ( K ) $. By the proof of the embedding $ \biggermodel{m}{n}_{\F_p} \hookrightarrow \affineflagvariety $, the chain $ \mathcal{L}_{\bullet} $ lifts to a chain $ \mathcal{F}_{\bullet} \in \affineflagvariety ( K ) $, which is a polarized lattice chain in the sense of A.41 of \cite{RZ} relative to the products $ t^{n - m} \Phi_K $. By Proposition A.43 of \cite{RZ}, there is $ \mathbf{g} \in \GU_d ( K (( t )) ) $ which is an isomorphism $ ( t^{-m} \lambda_{\bullet} \widehat{\otimes}_{\F_p} K, t^{2 m} \Phi_K ) \directedisom ( \mathcal{F}_{\bullet}, t^{n - m} \Phi_K ) $. This $ \mathbf{g} $ clearly induces some $ g^{\prime} $ in the fiber of $ \biggermonoid{m}{n} ( K ) \rightarrow \biggermodel{m}{n} ( K ) $ over $ \mathcal{L}_{\bullet} $. Since it was already proved that $ \universalgroup{m}{n}{\mu}{\nu} $ is transitive on such fibers, replacing $ g $ with $ g^{\prime} $ simply conjugates the fixer to which $ \gamma $ belongs. Hence, I may assume that $ g $ is induced by $ \mathbf{g} $. By the Bruhat-Tits Decomposition, there are $ \alpha, \beta \in \iwahori ( K ) $ such that $ \alpha \cdot \mathbf{g} \cdot \beta $ is ``standard'', i.e. is a nice representative of the extended affine Weyl group and transforms any standard lattice $ t^{a_i} \mathcal{K}[[t]] \oplus \cdots \oplus t^{a_d} \mathcal{K}[[t]] $ to another standard lattice. Conjugating the fixer of $ g $ again, by $ \beta $, I may assume that $ g $ is induced by a standard $ \mathbf{g} $ as just described. This implies that $ \kernel ( g ) $ is a standard lattice chain, generated over $ \mathcal{K}[t] $ by $ t^{a_1} e_1, \ldots, t^{a_d} e_d $ for some $ a_1, \ldots, a_d \in \Z $. Altogether, $ \gamma \in \universalgroup{m}{n}{\mu}{\nu} ( K ) $ fixes $ g $ if and only if $ \gamma \in \mathbb{G} ( K[t] / t^N K[t] ) $ is in the image of $ \iwahori ( K ) $ and $ \gamma ( v ) - v \in t^{a_1} \mathcal{K}[[t]] \oplus \cdots \oplus t^{a_d} \mathcal{K}[[t]] $ for all $ v \in \overline{\mathcal{W}}_{\SUP} ( K ) $. By the Iwahori Factorization, it is easy to identify this subgroup: it is a product of groups $ R ( \Gm ) $ and $ R ( \Ga ) $ and $ R^0 ( \Ga ) $, where $ R $ denotes restriction-of-scalars from $ \F $ to $ \F_p $ and $ R^0 \subset R $ is the subgroup defined by $ \trace_{\F / \F_p} ( - ) = 0 $. Applying the above argument twice proves the claim for the action $ \alpha_1 $. If $ \mychar ( K ) = 0 $ then the above ideas may be used after invoking the Chinese Remainder Theorem as in the proof of Lemma \ref{Lgenericfiberisom} (page \pageref{Lgenericfiberisom}) to decompose all relevant rings and modules into a $ t $ part and a $ t + p $ part.

Now consider the $ \alpha_2 $-action. Let $ K $ be a $ \Z_p $-field. The two maps $ \eta \mapsto ( \identity, \eta ) $ and $ ( \gamma, \eta ) \mapsto \gamma $ yield an exact sequence $ 1 \rightarrow \fixer ( h ) \rightarrow \fixer ( g, h ) \rightarrow \fixer ( g ) $ of $ K $-groups. Let $ R $ be a \emph{local} commutative $ K $-algebra and suppose $ \gamma \in \fixer ( g ) ( R ) $. By Lemma \ref{Lactionhelper}, $ \gamma \circ h $ is in the same fiber as $ h $ so transitivity of $ \universalgroup{m}{n}{\mu}{\nu} $ (proved already) implies the existence of $ \eta $ such that $ \gamma \circ h \circ \eta^{-1} = h $, i.e. $ \fixer ( g, h ) ( R ) \rightarrow \fixer ( g ) ( R ) $ is surjective. It follows that $ \fixer ( g, h ) \rightarrow \fixer ( g ) $ an \emph{epimorphism} and the above sequence extends to a short-exact-sequence. Since extensions of connected affine algebraic groups are connected, $ \fixer ( g, h ) ( K ) $ is connected whenever $ K $ is separably closed. To prove that $ \fixer ( g, h ) $ is smooth is equivalent to verify the infinitesimal lifting property. So, let $ R $ be a \emph{local} commutative $ K $-algebra and let $ I \subset R $ be a nilpotent ideal. By the earlier part of this paragraph, there is a short-exact-sequence $ 1 \rightarrow \fixer ( h ) ( R ) \rightarrow \fixer ( g, h ) ( R ) \rightarrow \fixer ( g ) ( R ) \rightarrow 1 $, a similar sequence with $ R / I $ in place of $ R $, and the obvious three morphisms $ * ( R ) \rightarrow * ( R / I ) $ connecting the two. Since $ \fixer ( * ) ( R ) \rightarrow \fixer ( * ) ( R / I ) $ is surjective for $ * = g, h $ (it was already proved that both fixers are smooth), a diagram-chase proves that $ \fixer ( g, h ) ( R ) \rightarrow \fixer ( g, h ) ( R / I ) $ is also surjective.
\end{proof}

\begin{remark}
Uses of Proposition \ref{Ptransitiveactions}:
\begin{itemize}
\item Transitivity for $ \alpha_1 $ is used in \S\ref{SSfusionproductdef} (page \pageref{SSfusionproductdef}).

\item Connectivity and Smoothness for $ \alpha_1 $ are not used.

\item Transitivity, Connectivity, and Smoothness for $ \alpha_2 $ are used to satisfy the hypotheses of Lemma 21 of \cite{HN} (page \pageref{labelHNdescendsheaf} here).
\end{itemize}
\end{remark}

All of this section's results involving $ p_2 $ and $ m $ are also true of $ \presuperscript{\textup{rev}}{p}_2 $ and $ \presuperscript{\textup{rev}}{m} $:
\begin{itemize}
\setlength{\itemsep}{5pt}
\item $ \presuperscript{\textup{rev}}{p}_2 $ is smooth and and $ \presuperscript{\textup{rev}}{p}_2 (K) $ is surjective for any $ \Z_p $-field $ K $,

\item the action via $ \alpha_2 $ stabilizes and is transitive on fibers of $ \presuperscript{\textup{rev}}{p}_2 $ over $ K $-points for any $ \Z_p $-field $ K $,

\item the fixer via $ \alpha_2 $ of any $ K $-point is smooth and connected for any $ \Z_p $-field $ K $, and

\item $ \presuperscript{\textup{rev}}{m} $ is proper.
\end{itemize}

\section{The convolution product} \label{Sfusionproduct}

\subsection{Construction of the convolution product} \label{SSfusionproductdef}

I use the following general Proposition 4.2.5 on page 109 of \cite{BBD}
\begin{BBDpullbackofperverse} \label{labelBBDpullbackofperverse}
If $ f : X \rightarrow Y $ is a smooth morphism of schemes with relative dimension $ n $ and the geometric fibers of $ f $ are connected, then the shifted pullback $ f^{*}[n] $ is a fully-faithful functor from perverse sheaves on $ Y $ to perverse sheaves on $ X $.
\end{BBDpullbackofperverse}
The phrase ``geometric fiber'' here has the customary meaning: the fiber of $ f $ over a $ K $-point of $ Y $ for a separably-closed field $ K $.

Fix a $ \Z_p $-field $ K $ and $ m, n, \mu, \nu \in \N $.

Let $ \mathcal{A} $ be a perverse $ \biggergroup{m}{n}_K $-equivariant $ \ell $-adic sheaf on $ \biggermodel{m}{n}_K $ and $ \mathcal{B} $ a perverse $ \integraliwahoriargs{\mu}{\nu}_K $-equivariant $ \ell $-adic sheaf on $ \integralaffineflagvarargs{\mu}{\nu}_K $. Because of the morphisms from \S\ref{SSuniversalgroupaction} (page \pageref{SSuniversalgroupaction}), I can unify these equivariance properties by saying that both are $ \universalgroup{m}{n}{\mu}{\nu}_K $-equivariant.

The external tensor product $ \mathcal{A} \boxtimes_K \mathcal{B} $ (ordinary derived tensor product of the pullbacks along both projections) is another perverse $ \universalgroup{m}{n}{\mu}{\nu}_K $-equivariant $ \ell $-adic sheaf on $ \biggermodel{m}{n}_K \times \integralaffineflagvarargs{\mu}{\nu}_K $. By Proposition \ref{Psmoothsurjections} (page \pageref{Psmoothsurjections}), $ p_1 $ is smooth.

The action by $ \universalgroup{m}{n}{\mu}{\nu}_K $ comes from the loop group $ R \mapsto \mathcal{P}(R[[t]]) $ of a parahoric $ \mathcal{P} $ (depending on $ \mychar(K) $), which is \emph{connected}, so by the transitivity statement of Proposition \ref{Ptransitiveactions} (page \pageref{Ptransitiveactions}) the geometric fibers of $ p_1 $ are connected. By Proposition 4.2.5 of \cite{BBD}, the pullback $ p_1^{*} ( \mathcal{A} \boxtimes_K \mathcal{B} ) $ is a perverse $ \ell $-adic sheaf.

Since the action of $ \universalgroup{m}{n}{\mu}{\nu} \times \universalgroup{m}{n}{\mu}{\nu} $ by $ \alpha_1 $ stabilizes $ p_1 $-fibers, $ p_1^{*} ( \mathcal{A} \boxtimes_K \mathcal{B} ) $ is \emph{trivially} $ \alpha_1 $-equivariant: $ p_1 \circ \alpha_1 = p_1 \circ \pr $ already. Since the difference between $ \alpha_1 $ and $ \alpha_2 $ is the action
\begin{align*}
\universalgroup{m}{n}{\mu}{\nu} \times ( \biggermonoid{m}{n} \times \integralaffineflagvarmonoid{\mu}{\nu} ) &\longrightarrow ( \biggermonoid{m}{n} \times \integralaffineflagvarmonoid{\mu}{\nu} ) \\
( \gamma , ( g, h ) ) &\longmapsto ( g, \gamma \circ h )
\end{align*}
the initial assumption that $ \mathcal{A} $ and $ \mathcal{B} $ were $ \universalgroup{m}{n}{\mu}{\nu} $-equivariant implies that $ p_1^{*} ( \mathcal{A} \boxtimes_K \mathcal{B} ) $ is $ \alpha_2 $-equivariant.

I use the following general Lemma 21 from \cite{HN}:
\begin{HNdescendsheaf} \label{labelHNdescendsheaf}
Let $ \pi : X \rightarrow Y $ be a morphism of finite-type $ \Z_p $-schemes. Let $ G $ be a $ \Z_p $-group scheme. Let $ a_X : G \times X \rightarrow X $ a group action over $ \spec(\Z_p) $. Let $ G $ act trivially on $ Y $. Let $ \mathcal{F} $ be a perverse $ a_X $-equivariant \etale $ \ell $-adic sheaf on $ X $. Assume that $ \pi $ is smooth and surjective on the level of topological spaces, that $ G $ is smooth, that for any $ \Z_p $-field $ K $ the action of $ G $ on $ X $ is transitive on fibers of $ \pi $ over $ K $-points, $ G_K $ is connected, the fixer subscheme of a $ K $-point of $ X $ is a smooth connected subgroup of $ G_K $. \textbf{Assertion}: There is a unique perverse $ \ell $-adic sheaf $ \mathcal{G} $ on $ Y $ such that $ \mathcal{F} \cong \pi^{*}( \mathcal{G} ) $.
\end{HNdescendsheaf}

\begin{definition}
The sheaf complex
\begin{equation*}
\mathcal{A} \odot_K \mathcal{B}
\end{equation*}
on $ \convolutionobject{m}{n}{\mu}{\nu} $ is the complex ``$ \mathcal{G} $'' guaranteed by Lemma 21 of \cite{HN} for
\begin{itemize}
\item the morphism $ p_2 : \biggermonoid{m}{n} \times \integralaffineflagvarmonoid{\mu}{\nu} \longrightarrow \convolutionobject{m}{n}{\mu}{\nu} $,

(hypothesis provided by Proposition \ref{Psmoothsurjections})

\item the action of $ \universalgroup{m}{n}{\mu}{\nu} \times \universalgroup{m}{n}{\mu}{\nu} $ via $ \alpha_2 $, and

(hypotheses provided by Proposition \ref{PsmoothuniversalJ}, the discussion in \S\ref{SSfusionproductdef}, and Proposition \ref{Ptransitiveactions})

\item the sheaf complex $ p_1^{*} ( \mathcal{A} \boxtimes_K \mathcal{B} ) $.

(hypotheses provided by the discussion earlier in this subsection)
\end{itemize}
\end{definition}
Note that Proposition \ref{Psmoothsurjections} (page \pageref{Psmoothsurjections}) and Lemma \ref{Lgenericfiberisom} (page \pageref{Lgenericfiberisom}) together imply that $ p_1 $ and $ p_2 $ have the same relative dimension over each component of $ \convolutionobject{m}{n}{\mu}{\nu} $ and $ \biggermodel{m}{n} \times \integralaffineflagvarargs{\mu}{\nu} $: smoothness of $ p_1 $ and $ p_2 $ imply \emph{constant} relative dimension, but at the same time $ \convolutionobject{m}{n}{\mu}{\nu}_{\Q_p} \cong \biggermodel{m}{n}_{\Q_p} \times \integralaffineflagvarargs{\mu}{\nu}_{\Q_p} $. This means that $ \mathcal{A} \odot_K \mathcal{B} $ is already perverse (i.e. no shift is needed).

\begin{convolutionproduct}
The convolution product $ *_K $ is defined by
\begin{equation*} \label{Efusionproductdef}
\mathcal{A} *_K \mathcal{B} \defeq R m_{!} ( \mathcal{A} \odot_K \mathcal{B} ),
\end{equation*}
a complex of $ \ell $-adic sheaves on $ \convolutiontarget{m}{n}{\mu}{\nu}_K $.
\end{convolutionproduct}
Note that $ m_{*} = m_{!} $ since $ m $ is a proper morphism.

Repeating the above using the \emph{reversed} convolution diagram from \S\ref{SSreversedconvolutiondiagram} (page \pageref{SSreversedconvolutiondiagram}) produces the product $ \mathcal{B} \odot_K \mathcal{A} $ on $ \reversedconvolutionobject{\mu}{\nu}{m}{n} $ and the \emph{reversed} convolution product:
\begin{reversedconvolutionproduct}
The ``reversed'' convolution product $ \mathcal{B} *_K \mathcal{A} $ on $ \convolutiontarget{m}{n}{\mu}{\nu}_K $ is defined by
\begin{equation*}
\mathcal{B} *_K \mathcal{A} \defeq R ( \presuperscript{\textup{rev}}{m}_{!} ) ( \mathcal{B} \odot_K \mathcal{A} )
\end{equation*}
\end{reversedconvolutionproduct}

%There is no ambiguity between the original and reversed convolution products because the complexes $ \mathcal{A} $ and $ \mathcal{B} $ have different base spaces.

\subsection{The product of sheaves categorifies the product of functions}
\label{SSfusionproductcategorifies}

\begin{center}
\emph{It is natural to ask exactly how the convolution product of sheaf complexes is related to the convolution product of functions in the Hecke algebra. This is apparently well-known, but since I have not seen it in print, I explain it. This subsection is not logically required for any other part of the paper.}
\end{center}

Let $ m, n, \mu, \nu \in \N $ be arbitrary. Let $ \mathcal{A} $ and $ \mathcal{B} $ be (bounded, constructible) complexes of $ \ell $-adic sheaves on $ \biggermodel{m}{n}_{\overline{\F}_p} $ and $ \integralaffineflagvarargs{\mu}{\nu}_{\overline{\F}_p} $ equipped with actions of $ \gal ( \overline{\Q}_p / \Q_p ) $ that are consistent with the action of $ \gal ( \overline{\F}_p / \F_p ) $ on $ \biggermodel{m}{n}_{\overline{\F}_p} $ and $ \integralaffineflagvarargs{\mu}{\nu}_{\overline{\F}_p} $. Then $ \mathcal{A} *_{\overline{\F}_p} \mathcal{B} $ is a (bounded, constructible) complex of $ \ell $-adic sheaves on $ \convolutiontarget{m}{n}{\mu}{\nu}_{\overline{\F}_p} $ with all the same properties.

By \S\ref{SSdiagramcollapsemodp} (page \pageref{SSdiagramcollapsemodp}), the associated trace functions under consideration are:
\begin{align*}
\sstrace{ \mathcal{A} } &: \integralaffineflagvarargs{m}{n} ( \F_p ) \longrightarrow \overline{\Q}_{\ell} \\
\sstrace{ \mathcal{B} } &: \integralaffineflagvarargs{\mu}{\nu} ( \F_p ) \longrightarrow \overline{\Q}_{\ell} \\
\sstrace{ \mathcal{A} * \mathcal{B} } &: \integralaffineflagvarargs{m+\mu}{n+\nu} ( \F_p ) \longrightarrow \overline{\Q}_{\ell}
\end{align*}

Let $ x \in \integralaffineflagvarargs{m+\mu}{n+\nu} ( \F_p ) $ be arbitrary and set
\begin{equation*}
\mathfrak{f} \defeq \{ y \in \integralaffineflagvarargs{m}{n} ( \F_p ) \suchthat ( y, x ) \in \convolutionobject{m}{n}{\mu}{\nu} ( \F_p ) \}
\end{equation*}
which is essentially the fiber $ m ( \F_p )^{-1} ( x ) \subset \convolutionobject{m}{n}{\mu}{\nu} ( \F_p ) $.

Because of the semisimplification done in \S\ref{SStracedef} (page \pageref{SStracedef}), the $ \Gamma_0 $-invariants operation is exact, and the following general Proposition 10 from \cite{HN} results:
\begin{HNfibersum}
Let $ f : X \rightarrow Y $ be a morphism of $ \F_p $-schemes and let $ \mathcal{C} $ be a complex of $ \ell $-adic sheaves on $ X $ with an action $ \gal ( \overline{\Q}_p / \Q_p ) $ compatible with that on $ X(\F_p) $. Then for any $ y \in Y ( \F_p ) $,
\begin{equation*}
\sstrace{ f_{!}( \mathcal{C} ) } (y) = \sum_{ \substack{ x \in X(\F_p) \\ f(x)=y } } \sstrace{ \mathcal{C} }(x)
\end{equation*}
\end{HNfibersum}
This implies that
\begin{equation*}
\sstrace{ \mathcal{A} * \mathcal{B} } ( x ) = \sum_{ y \in \mathfrak{f} } \sstrace{ \mathcal{A} \odot \mathcal{B} } ( y, x )
\end{equation*}

For any $ y \in \mathfrak{f} $, if $ ( h, k ) \in \integralaffineflagvarmonoid{m}{n} ( \F_p ) \times \integralaffineflagvarmonoid{\mu}{\nu} ( \F_p ) $ is such that $ p_2 ( h, k ) = ( y, x ) $ (such elements exist by Proposition \ref{Psmoothsurjections} (page \pageref{Psmoothsurjections}))) then setting $ ( y, z ) := p_1 ( h, k ) $ (recall that the first coordinates $ p_1 $ and $ p_2 $ are the same), it is true that
\begin{equation*}
\sstrace{ \mathcal{A} \odot \mathcal{B} } ( y, x ) = \sstrace{ \mathcal{A} } ( y ) \cdot \sstrace{ \mathcal{B} } ( z )
\end{equation*}
(this follows from general sheaf theory: the way the operations $ \boxtimes $ (external tensor product), $ p_1^{*} $ and $ p_2^{*} $ interact with stalks)

Let $ \mathfrak{e} \subset \integralaffineflagvarargs{\mu}{\nu} ( \F_p ) $ be the set of all $ z $ occuring in this way. Then the above sum can be rewritten
\begin{equation*}
\sstrace{ \mathcal{A} * \mathcal{B} } ( x ) = \sum_{ \substack{ y \in \mathfrak{f} \\ z \in \mathfrak{e} } } \sstrace{ \mathcal{A} } ( y ) \cdot \sstrace{ \mathcal{B} } ( z )
\end{equation*}

To see this as a convolution, inflate $ \sstrace{ \mathcal{A} } $ and $ \sstrace{ \mathcal{B} } $ to $ \integralaffineflagvarmonoid{m}{n} ( \F_p ) $ and $ \integralaffineflagvarmonoid{\mu}{\nu} ( \F_p ) $ and recall the ``twisting'' that occurs in the 2nd coordinate of $ p_2 $. Then for any $ x \in \integralaffineflagvarargs{m+\mu}{n+\nu} ( \F_p ) $,
\begin{equation*}
\sstrace{ \mathcal{A} * \mathcal{B} } ( x ) = \sum_{ \substack{ h \in \integralaffineflagvarmonoid{m}{n} ( \F_p ) \\ k \in \integralaffineflagvarmonoid{\mu}{\nu} ( \F_p ) \\ h ( k ( - ) ) = x } } \sstrace{ \mathcal{A} } ( h ) \cdot \sstrace{ \mathcal{B} } ( k )
\end{equation*}
(so $ h $ plays the role of ``$ y $'' and $ k $ plays the role of ``$ y^{-1} x $'' in the expression ``$ (f*g) (x) = \sum_{y} f(y)g(y^{-1}x) $'')

\section{Proof of the Main Theorem} \label{Smainproof}

Let $ w \in \widetilde{W} $ be arbitrary. There exists $ \mu, \nu \in \N $ (infinitely many, all with the same difference $ \Delta = \mu - \nu $) such that the Schubert variety $ \overline{C}_w $ is contained in $ \integralaffineflagvarargs{\mu}{\nu}_{ \overline{\F}_p } $. Let $ \ICbar{w} $ be the (perverse) \etale $ \ell $-adic intersection complex associated to the cell $ C_w $ in the Bruhat-Tits Decomposition of $ \integralaffineflagvarargs{\mu}{\nu}_{ \overline{\F}_p } $. The function
\begin{equation*}
\sstrace{ \ICbar{w} } : \integralaffineflagvarargs{m}{n} ( \F_p ) \longrightarrow \overline{\Q}_{\ell}
\end{equation*}
is also an element of the Iwahori-Hecke algebra $ \mathcal{H} $. By the main theorems of \cite{KL1} and \cite{KL2}, the set of these functions $ \sstrace{ \ICbar{w} } $ for all $ w \in \widetilde{W} $ forms a vector-space basis for $ \mathcal{H} $ .

Therefore, to show that $ \sstrace{\mu} \in Z ( \mathcal{H} ) $, it suffices to show that
\begin{equation*}
\sstrace{\mu} * \sstrace{ \ICbar{w} } \iseq \sstrace{ \ICbar{w} } * \sstrace{\mu}
\end{equation*}
(ordinary convolution of functions) for every $ w \in \widetilde{W} $.

%\begin{remark}Notice that none of the $ \Z_p $-schemes $ \biggermodel{m}{n} $ are genuinely needed in the proof except the support $ \biggermodel{0}{1} = \localmodel $ of $ \sstrace{\mu} $ (all $ \integralaffineflagvarargs{\mu}{\nu} $ are needed, however). But it is not much harder to prove things for a general $ \biggermodel{m}{n} $ than for $ \biggermodel{0}{1} $, and it is interesting in any case to have such a degeneration in the unramified unitary case. For similar applications to trace functions $ \sstrace{\lambda} $ for \emph{non}-minuscule $ \lambda $, as is the case in \cite{HN}, the larger schemes $ \biggermodel{m}{n} $ really are necessary.\end{remark}

Recall from \S\ref{SSintegralcompleteaffineflagvariety} (page \pageref{SSintegralcompleteaffineflagvariety}) that if $ \IC{w} $ is the (perverse) \etale $ \ell $-adic intersection complex associated to the cell $ C_w $ in the Bruhat-Tits Decomposition of $ \affineflagvariety_{\Q_p} $ (note the base field here), then $ \ICbar{w} \directedisom \nearbycycles ( \IC{w} ) $. Recall from \S\ref{SStracedef} (page \pageref{SStracedef}) that \emph{by definition} if $ \IC{\mu} $ is the (perverse) \etale $ \ell $-adic intersection complex associated to the cell $ O_{\mu} $ in the Cartan Decomposition of $ \affinegrassmannian_{\Q_p} $ then $ \sstrace{\mu} = \sstrace{ \nearbycycles ( \IC{\mu} ) } $. Using these two identities, it suffices to prove that
\begin{equation*}
\sstrace{ \nearbycycles ( \IC{\mu} ) } * \sstrace{\nearbycycles ( \IC{w} )} \iseq \sstrace{\nearbycycles ( \IC{w} )} * \sstrace{ \nearbycycles ( \IC{\mu} ) }
\end{equation*}
By \S\ref{SSfusionproductcategorifies} (page \pageref{SSfusionproductcategorifies}), the convolution product of sheaves induces the convolution product of functions, so it suffices to show\footnote{Note that the \emph{reversed} convolution product occurs on the right-hand-side here.} that
\begin{equation*}
\nearbycycles ( \IC{\mu} ) *_{ \overline{\F}_p } \nearbycycles ( \IC{w} ) \iseq \nearbycycles ( \IC{w} ) *_{ \overline{\F}_p } \nearbycycles ( \IC{\mu} )
\end{equation*}

By the general Lemma 23 of \cite{HN} (``nearby cycles commutes with convolution product''), this equality is equivalent to the equality
\begin{equation} \label{Efinalsheafisom}
\nearbycycles ( \IC{\mu} *_{ \Q_p } \IC{w} ) \iseq \nearbycycles ( \IC{w} *_{ \Q_p } \IC{\mu} )
\end{equation}

\begin{remark}
Lemma 23 of \cite{HN} applies because the fields involved here are algebraically-closed: an argument similar to that given in \S6.3.3 of \cite{Ha2} proves that the schemes used here simplify to the $ \GL $-versions after passing to the algebraic closure.
\end{remark}

\begin{remark}
Lemma 23 in \cite{HN} uses ``smooth base-change'' for $ p_1 $ and $ p_2 $, and therefore requires Proposition \ref{Psmoothsurjections} (page \pageref{Psmoothsurjections}).
\end{remark}

The following lemma implies that this last isomorphism (\ref{Efinalsheafisom}) is true.
\begin{lemma} \label{Lgenericfibercommutativity}
Recall the isomorphisms from Lemma \ref{Lgenericfiberisom} (page \pageref{Lgenericfiberisom}):
\begin{equation*}
\begin{CD}
\biggermodel{m}{n}_{\Q_p} \times \integralaffineflagvarargs{\mu}{\nu}_{\Q_p} @<i<< \convolutionobject{m}{n}{\mu}{\nu}_{\Q_p} \\
@A\presuperscript{\textup{rev}}{i}AA @VVmV \\
\reversedconvolutionobject{\mu}{\nu}{m}{n}_{\Q_p} @>>\presuperscript{\textup{rev}}{m}> \convolutiontarget{m}{n}{\mu}{\nu}_{\Q_p}
\end{CD}
\end{equation*}

\textbf{Assertion}: if $ \mathcal{A} $ and $ \mathcal{B} $ are complexes of $ \universalgroup{m}{n}{\mu}{\nu}_{\Q_p} $-equivariant $ \ell $-adic sheaves on $ \biggermodel{m}{n}_{\Q_p} $ and $ \integralaffineflagvarargs{\mu}{\nu}_{\Q_p} $ respectively, then
\begin{align*}
i^{*} ( \mathcal{A} \boxtimes_{\Q_p} \mathcal{B} ) &\directedisom \mathcal{A} \odot_{\Q_p} \mathcal{B} \\
\presuperscript{\textup{rev}}{i}^{*} ( \mathcal{A} \boxtimes_{\Q_p} \mathcal{B} ) &\directedisom \mathcal{B} \odot_{\Q_p} \mathcal{A}
\end{align*}
Applying $ R m_{!} $ and $ R ( \presuperscript{\textup{rev}}{m}_{!} ) $ to these isomorphisms and using commutativity of the above square implies that
\begin{equation*}
\mathcal{A} *_{\Q_p} \mathcal{B} \cong \mathcal{B} *_{\Q_p} \mathcal{A}.
\end{equation*}
\end{lemma}

\begin{proof}
The proof is nearly identical to the one occuring for the 2nd part of Lemma 24 in \cite{HN}, replacing the objects and morphisms used there by the slightly modified objects and morphisms used in this paper for Lemma \ref{Lgenericfiberisom} (page \pageref{Lgenericfiberisom}).
\end{proof}


\begin{bibdiv}
\begin{biblist}

\bib{AM}{book}{
   author={Atiyah, M. F.},
   author={Macdonald, I. G.},
   title={Introduction to commutative algebra},
   publisher={Addison-Wesley Publishing Co., Reading, Mass.-London-Don Mills, Ont.},
   date={1969},
   %pages={ix+128},
   %review={\MR{0242802 (39 \#4129)}},
}

\bib{Bo}{book}{
   author={Bourbaki, Nicolas},
   title={Commutative algebra. Chapters 1--7},
   series={Elements of Mathematics (Berlin)},
   note={Translated from the French; Reprint of the 1989 English translation},
   publisher={Springer-Verlag, Berlin},
   date={1998},
   %pages={xxiv+625},
   isbn={3-540-64239-0},
   %review={\MR{1727221 (2001g:13001)}},
}

\bib{BBD}{article}{
   author={Be{\u\i}linson, A. A.},
   author={Bernstein, J.},
   author={Deligne, P.},
   title={Faisceaux pervers},
   language={French},
   conference={
      title={Analysis and topology on singular spaces, I},
      address={Luminy},
      date={1981},
   },
   book={
      series={Ast\'erisque},
      volume={100},
      publisher={Soc. Math. France, Paris},
   },
   date={1982},
   pages={5--171},
   %review={\MR{751966}},
}

% \bib{BLR}{book}{
%    author={Bosch, Siegfried},
%    author={L{\"u}tkebohmert, Werner},
%    author={Raynaud, Michel},
%    title={N\'eron models},
%    series={Ergebnisse der Mathematik und ihrer Grenzgebiete (3) [Results in
%    Mathematics and Related Areas (3)]},
%    volume={21},
%    publisher={Springer-Verlag, Berlin},
%    date={1990},
%    %pages={x+325},
%    isbn={3-540-50587-3},
%    %review={\MR{1045822 (91i:14034)}},
%    %doi={10.1007/978-3-642-51438-8},
% }

\bib{DG}{book}{
   author={Demazure, Michel},
   author={Gabriel, Peter},
   title={Introduction to algebraic geometry and algebraic groups},
   series={North-Holland Mathematics Studies},
   volume={39},
   note={Translated from the French by J. Bell},
   publisher={North-Holland Publishing Co., Amsterdam-New York},
   date={1980},
   %pages={xiv+357},
   isbn={0-444-85443-6},
   %review={\MR{563524 (82e:14001)}},
}

\bib{Ei}{book}{
   author={Eisenbud, David},
   title={Commutative algebra, With a view toward algebraic geometry},
   series={Graduate Texts in Mathematics},
   volume={150},
   publisher={Springer-Verlag, New York},
   date={1995},
   %pages={xvi+785},
   %isbn={0-387-94268-8},
   isbn={0-387-94269-6},
   %review={\MR{1322960 (97a:13001)}},
   %doi={10.1007/978-1-4612-5350-1},
}

\bib{Ga}{article}{
   author={Gaitsgory, D.},
   title={Construction of central elements in the affine Hecke algebra via nearby cycles},
   journal={Invent. Math.},
   volume={144},
   date={2001},
   number={2},
   pages={253--280},
   %issn={0020-9910},
   %review={\MR{1826370}},
   %doi={10.1007/s002220100122},
}

\bib{Go1}{article}{
author={G\"{o}rtz, Ulrich},
title={Flatness of local models of certain Shimura varieties of PEL type},
journal={Math. Ann.},
volume={321},
pages={689--727},
date={2001}
}

\bib{Go2}{article}{
   author={G{\"o}rtz, Ulrich},
   title={Affine Springer fibers and affine Deligne-Lusztig varieties},
   conference={
      title={Affine flag manifolds and principal bundles},
   },
   book={
      series={Trends Math.},
      publisher={Birkh\"auser/Springer Basel AG, Basel},
   },
   date={2010},
   pages={1--50},
   %review={\MR{3013026}},
   %doi={10.1007/978-3-0346-0288-4_1},
}

\bib{GW}{book}{
   author={G{\"o}rtz, Ulrich},
   author={Wedhorn, Torsten},
   title={Algebraic geometry I},
   series={Advanced Lectures in Mathematics},
   note={Schemes with examples and exercises},
   publisher={Vieweg + Teubner, Wiesbaden},
   date={2010},
   %pages={viii+615},
   isbn={978-3-8348-0676-5},
   %review={\MR{2675155 (2011f:14001)}},
   %doi={10.1007/978-3-8348-9722-0},
}

\bib{Ha1}{article}{
   author={Haines, Thomas J.},
   title={The combinatorics of Bernstein functions},
   journal={Trans. Amer. Math. Soc.},
   volume={353},
   date={2001},
   number={3},
   pages={1251--1278 (electronic)},
   %issn={0002-9947},
   %review={\MR{1804418}},
   %doi={10.1090/S0002-9947-00-02716-1},
}

\bib{Ha2}{article}{
   author={Haines, Thomas J.},
   title={Introduction to Shimura varieties with bad reduction of parahoric type},
   conference={
      title={Harmonic analysis, the trace formula, and Shimura varieties},
   },
   book={
      series={Clay Math. Proc.},
      volume={4},
      publisher={Amer. Math. Soc., Providence, RI},
   },
   date={2005},
   pages={583--642},
   %review={\MR{2192017}},
}

\bib{HN}{article}{
author={Haines, Thomas},
author={Ng\^{o}, B\'{a}u Ch\^{a}u}*{inverted={yes}},
title={Nearby cycles for local models of some Shimura varieties},
journal={Compositio Math.},
volume={133},
pages={117--150},
date={2002}
}

% \bib{HRa}{article}{
% label={HRa08},
% author={Haines, Thomas},
% author={Rapoport, Michael},
% title={Appendix: On parahoric subgroups},
% status={appendix to ``Twisted loop groups and their affine flag varieties'', Adv. Math. 219, 2008},
% date={2008}
% }

% \bib{HRo}{article}{
% label={HRo10},
% author={Haines, Thomas},
% author={Rostami, Sean},
% title={The Satake isomorphism for special maximal parahoric Hecke algebras},
% journal={Represent. Theory},
% volume={14},
% pages={264--284},
% date={2010}
% }

\bib{KL1}{article}{
author={Kazhdan, David},
author={Lusztig, George},
title={Representations of Coxeter groups and Hecke algebras},
journal={Invent. Math.},
volume={53},
pages={165--184},
date={1979}
}

\bib{KL2}{article}{
author={Kazhdan, David},
author={Lusztig, George},
title={Schubert varieties and Poincar\'{e} duality},
journal={Proc. Sympos. Pure Math.},
volume={36},
pages={185--203},
date={1980}
}

\bib{Ko1}{article}{
author={Kottwitz, Robert},
title={On the $ \lambda $-adic representations associated to some simple Shimura varieties},
journal={Invent. Math.},
volume={108},
pages={653--665},
date={1992}
}

% \bib{Ko2}{article}{
% author={Kottwitz, Robert},
% title={Isocrystals with additional structure, 2},
% journal={Compositio Math.},
% volume={109},
% pages={255--339},
% date={1997}
% }

\bib{KR}{article}{
   author={Kottwitz, R.},
   author={Rapoport, M.},
   title={Minuscule alcoves for ${\rm GL}_n$ and $G{\rm Sp}_{2n}$},
   journal={Manuscripta Math.},
   volume={102},
   date={2000},
   number={4},
   pages={403--428}
}

\bib{Lu}{article}{
author={Lusztig, George},
title={Cells in affine Weyl groups and tensor categories},
journal={Adv. Math.},
volume={129},
pages={85--98},
date={1997}
}

% \bib{Ma}{book}{
%    author={Macdonald, I. G.},
%    title={Spherical functions on a group of $p$-adic type},
%    note={Publications of the Ramanujan Institute, No. 2},
%    publisher={Ramanujan Institute, Centre for Advanced Study in
%    Mathematics,University of Madras, Madras},
%    date={1971},
%    %pages={vii+79},
%    %review={\MR{0435301 (55 \#8261)}},
% }

\bib{P}{article}{
author={Pappas, Georgios},
title={On the arithmetic moduli schemes of PEL Shimura varieties},
journal={J. Algebraic Geom.},
volume={9},
pages={577--605},
date={2000}
}

\bib{PR1}{article}{
author={Pappas, Georgios},
author={Rapoport, Michael},
title={Local models in the ramified case 1: the EL case},
journal={J. Algebraic Geom.},
volume={12},
pages={107--145},
date={2003}
}

\bib{PR2}{article}{
author={Pappas, Georgios},
author={Rapoport, Michael},
title={Local models in the ramified case 2: splitting models},
journal={Duke Math. J.},
volume={127},
pages={193--250},
date={2005}
}

\bib{PR3}{article}{
author={Pappas, Georgios},
author={Rapoport, Michael},
title={Local models in the ramified case 3: unitary groups},
journal={J. Inst. Math. Jussieu},
volume={8},
pages={507--564},
date={2009}
}

% \bib{PR}{article}{
% author={Pappas, Georgios},
% author={Rapoport, Michael},
% title={Twisted loop groups and their affine flag varieties},
% journal={Adv. Math.},
% volume={219},
% pages={118--198},
% date={2008}
% }

\bib{PZ}{article}{
   author={Pappas, G.},
   author={Zhu, X.},
   title={Local models of Shimura varieties and a conjecture of Kottwitz},
   journal={Invent. Math.},
   volume={194},
   date={2013},
   number={1},
   pages={147--254},
   %issn={0020-9910},
   %review={\MR{3103258}},
   %doi={10.1007/s00222-012-0442-z},
}

\bib{R}{article}{
author={Rapoport, Michael},
title={On the bad reduction of Shimura varieties},
booktitle={Automorphic forms, Shimura varieties, and $ L $-functions},
journal={Perspect. Math.},
volume={11},
pages={253--321},
date={1988}
}

\bib{RZ}{book}{
   author={Rapoport, M.},
   author={Zink, Th.},
   title={Period spaces for $p$-divisible groups},
   series={Annals of Mathematics Studies},
   volume={141},
   publisher={Princeton University Press, Princeton, NJ},
   date={1996},
   %pages={xxii+324},
   %isbn={0-691-02782-X},
   isbn={0-691-02781-1},
   %review={\MR{1393439 (97f:14023)}},
}

\bib{roro}{article}{
   author={Rostami, Sean},
   title={The Bernstein presentation for general connected reductive groups},
   journal={J. London Math. Soc.},
   volume={9},
   date={2015},
   number={2},
   pages={514--536},
   %eprint={arXiv:1312.7374}
}

\bib{Sc}{book}{
   author={Scharlau, Winfried},
   title={Quadratic and Hermitian forms},
   series={Grundlehren der Mathematischen Wissenschaften [Fundamental
   Principles of Mathematical Sciences]},
   volume={270},
   publisher={Springer-Verlag, Berlin},
   date={1985},
   %pages={x+421},
   isbn={3-540-13724-6},
   %review={\MR{770063 (86k:11022)}},
   %doi={10.1007/978-3-642-69971-9},
}

\bib{Sm1}{article}{
   author={Smithling, Brian D.},
   title={Topological flatness of orthogonal local models in the split, even case. I},
   journal={Math. Ann.},
   volume={350},
   date={2011},
   number={2},
   pages={381--416},
   %issn={0025-5831},
   %review={\MR{2794915 (2012h:14068)}},
   %doi={10.1007/s00208-010-0562-y},
}

\bib{Sm2}{article}{
   author={Smithling, Brian D.},
   title={Topological flatness of local models for ramified unitary groups. I. The odd dimensional case},
   journal={Adv. Math.},
   volume={226},
   date={2011},
   number={4},
   pages={3160--3190},
   %issn={0001-8708},
   %review={\MR{2764885 (2012a:14058)}},
   %doi={10.1016/j.aim.2010.10.004},
}

\bib{Sm3}{article}{
   author={Smithling, Brian D.},
   title={Topological flatness of local models for ramified unitary groups. II. The even dimensional case},
   journal={J. Inst. Math. Jussieu},
   volume={13},
   date={2014},
   number={2},
   pages={303--393},
   %issn={1474-7480},
   %review={\MR{3177281}},
   %doi={10.1017/S1474748013000157},
}

% \bib{tits}{article}{
%    author={Tits, J.},
%    title={Reductive groups over local fields},
%    conference={
%       title={Automorphic forms, representations and $L$-functions},
%       address={Proc. Sympos. Pure Math., Oregon State Univ., Corvallis,
%       Ore.},
%       date={1977},
%    },
%    book={
%       series={Proc. Sympos. Pure Math., XXXIII},
%       publisher={Amer. Math. Soc., Providence, R.I.},
%    },
%    date={1979},
%    pages={29--69},
%    %review={\MR{546588 (80h:20064)}},
% }

\end{biblist}
\end{bibdiv}
\end{document}